\documentclass[reqno]{amsart}
\usepackage[margin=1.15in]{geometry}
\usepackage{microtype}
\usepackage{cite}
\usepackage{enumitem}
\usepackage[hidelinks]{hyperref}

\theoremstyle{plain}
\newtheorem{theorem}{Theorem}
\newtheorem{proposition}[theorem]{Proposition}
\newtheorem{corollary}[theorem]{Corollary}
\newtheorem{lemma}[theorem]{Lemma}
\newtheorem{problem}[theorem]{Open problem}

\theoremstyle{definition}
\newtheorem{definition}{Definition}

\theoremstyle{remark}
\newtheorem{remark}{Remark}

\allowdisplaybreaks

\begin{document}

\author{Phuong Le}
\address{Phuong Le$^{1,2}$ (ORCID: 0000-0003-4724-7118)\newline
	$^1$Faculty of Economic Mathematics, University of Economics and Law, Ho Chi Minh City, Vietnam; \newline
	$^2$Vietnam National University, Ho Chi Minh City, Vietnam}
\email{phuongl@uel.edu.vn}


\title[Rigidity and dimensional transition in a half-space]{The critical boundary value for a nonlinear Schr\"odinger equation in a half-space: rigidity and dimensional transition}

\begin{abstract}
Let $w_0$ be the positive homoclinic solution of $-w''+w=w^p$ and set $c_p=w_0(0)$. We study bounded positive solutions of
\[
-\Delta v+v=v^p\quad\text{in }\mathbb R^N_+,
\qquad v=c_p\quad\text{on }\partial\mathbb R^N_+,
\qquad v(x',x_N)\to0
\]
uniformly as $x_N\to\infty$. This is the threshold boundary value left open by Fern\'andez and Weth [Math.\ Ann.\ \textbf{383} (2022), 361--397]. We prove a dimension-dependent rigidity/nonuniqueness picture. For every $p>1$, the profile $w_0(x_N)$ is the unique bounded positive solution when $2\le N\le5$, whereas for $N\ge8$ there is a one-parameter family of bounded positive non-one-dimensional solutions. In the critical dimension $N=7$, nonuniqueness holds whenever a cubic coefficient $\kappa(p)$ of the exact reduced nonlinearity is positive; in particular $\kappa(p)>0$ for all $p\ge2$, and continuity extends this range below $2$. The corresponding amplitudes have Fowler ends, with their limiting energy, neck size, bulge size, and logarithmic period determined to leading order. The case $N=6$, as well as the complementary range in $N=7$, remain open.

The proof combines a boundary-adapted Modica estimate, projection onto the kernel of the one-dimensional linearized operator, an automatic tail non-concentration estimate, and a nonlinear cell problem yielding the exact reduced nonlinearity. For $N\ge8$ the reduced linearization is invertible in weighted spaces. In $N=7$ the critical resonance is bypassed by solving the projected radial equation nonlinearly and closing the transverse equation by a Schauder--Tychonoff fixed point.
\end{abstract}

\maketitle

\section{Introduction}

\subsection{The problem}

For $N\ge 1$ we write $\mathbb{R}^N_+=\{x=(x',x_N)\in\mathbb{R}^{N-1}\times\mathbb{R}:x_N>0\}$, and we consider the half-space Dirichlet problem
\begin{equation}\label{eq_Pc}
\begin{cases}
-\Delta v+v=|v|^{p-1}v &\text{in }\mathbb{R}^N_+,\\
v=c &\text{on }\partial\mathbb{R}^N_+,\\
\displaystyle\lim_{x_N\to\infty}v(x',x_N)=0 &\text{uniformly in }x'\in\mathbb{R}^{N-1},
\end{cases}
\end{equation}
where $p>1$ and $c>0$ are constants. By a bounded positive solution of \eqref{eq_Pc} we mean, as in \cite{FW}, a function $v\in C^2(\mathbb{R}^N_+)\cap C(\overline{\mathbb{R}^N_+})\cap L^\infty(\mathbb{R}^N_+)$ with $v>0$ in $\mathbb{R}^N_+$ which satisfies \eqref{eq_Pc} pointwise.

The one-dimensional equation $-w''+w=w^p$ has, up to translations and sign, a unique nontrivial solution vanishing at $\pm\infty$, namely
\begin{equation}\label{eq_w0}
w_0(t)=c_p\left[\cosh\left(\frac{p-1}{2}t\right)\right]^{-\frac{2}{p-1}},
\qquad
c_p:=\left(\frac{p+1}{2}\right)^{\frac{1}{p-1}}=w_0(0)=\max_{t\in\mathbb{R}}w_0.
\end{equation}
Fern\'andez and Weth \cite{FW} proved that the number $c_p\in(1,\sqrt e)$ separates two regimes. If $N\ge 2$ and $0<c<c_p$, then \eqref{eq_Pc} has infinitely many bounded positive solutions, at least three of which are geometrically distinct. If $c>c_p$, then \eqref{eq_Pc} has no bounded positive solution. The borderline value $c=c_p$ was left open. We quote the question from \cite[Remark 1.1(e)]{FW}: it remains an interesting open question whether $x\mapsto w_0(x_N)$ is the unique bounded positive solution to \eqref{eq_Pc} when $c=c_p$.

Sourdis subsequently proved uniqueness for $N=2,3$ and every $p>1$ \cite{So1}. He later obtained uniqueness for $2\le N\le5$ when $1<p<3$, and for $2\le N\le4$ when $p=3$ \cite{So2}, by reducing the problem to a Lane--Emden supersolution inequality in the tangential variables. The remaining ranges include all $N\ge6$, as well as higher exponents in dimensions $N\ge4$.

Throughout this paper $N\ge 2$ and $p>1$, and we write $n:=N-1$ for the dimension of the boundary $\partial\mathbb{R}^N_+\simeq\mathbb{R}^{n}$, which is where the reduced problem of Sections \ref{sec_reduced} to \ref{c:sec_fixed} lives. The boundary value is $c=c_p$ except in Theorem \ref{thm_main} and Lemma \ref{lem_reg}, which hold for every $c>0$. We set
\begin{equation}\label{eq_W}
W(s):=\frac{s^2}{2}-\frac{|s|^{p+1}}{p+1},
\end{equation}
so that the equation in \eqref{eq_Pc} reads $\Delta v=W'(v)$. The function $W$ vanishes at $0$ and at $c_p$, is positive on $(0,c_p)$, and is negative and strictly decreasing on $(c_p,\infty)$. The homoclinic profile has zero energy,
\begin{equation}
\tfrac12 (w_0')^2=W(w_0)\quad\text{on }\mathbb{R},
\end{equation}
and this is the reason why $c_p$ is the threshold. Indeed, in dimension $N=1$ the quantity $\tfrac12(v')^2-W(v)$ is constant, the decay of $v$ at $+\infty$ forces the constant to be $0$, and $W\ge 0$ then forces $v\le c_p$. The whole difficulty for $N\ge 2$ is that this energy is no longer conserved along the $x_N$ direction, since it can be transported in the tangential directions.

\subsection{Three mechanisms and their thresholds}
Three mechanisms govern the critical problem.

\emph{Sliding.} The sliding argument of \cite{FW} for $c>c_p$ still gives
$v\ge w_0(x_N)$ when $c=c_p$, as observed in \cite{So1}, but the sliding parameter stops exactly at zero: the strict boundary inequality used for $c>c_p$ is lost at the profile $w_0$ itself. Thus sliding provides the ordering, but not uniqueness.

\emph{Projection onto the kernel.} Put $\phi=v-w_0$ and $Z=-w_0'$. Since $Z$ spans the kernel of the one-dimensional linearized operator, projection in the $x_N$ variable gives an exact scalar equation for
\[
\lambda(x')=\frac1\nu\int_0^\infty \phi(x',t)Z(t)\,dt.
\]
Its leading form is the Lane--Emden equation $-\Delta_{x'}\lambda=C_p\lambda^2$. The exponent two is below the Serrin threshold in boundary dimensions at most four, critical for the Sobolev theory in boundary dimension six, and supercritical above it. The missing point in the earlier projection argument is a uniform quadratic lower bound for all $p$. Section \ref{sec_reduced} supplies it by proving that the $\phi Z$ mass cannot concentrate at large $x_N$; this closes the uniqueness argument for every $p>1$ when $N\le5$.

\emph{Gradient bounds.} The classical Modica estimate \cite{Mo} requires nonnegativity of the potential on the range, which here is equivalent to the a priori conclusion $\sup v\le c_p$. The available half-space estimate with zero boundary data \cite[Theorem 1.4]{FV1} does not apply to the constant datum $c_p$. What replaces these hypotheses is the combination of flatness and decay: because $v$ is constant on the flat boundary, the normal derivative of the $P$-function vanishes there, while uniform decay in $x_N$ controls the opposite end. This yields Theorem \ref{thm_main}.

\subsection{Main results}

We write $M:=\sup_{\mathbb R^N_+}v$. Our first result is a sharp boundary-adapted Modica estimate. With $W$ as in \eqref{eq_W},
\begin{equation}\label{eq_minW}
\min_{[0,M]}W=\min\{0,W(M)\}=
\begin{cases}
0,&M\le c_p,\\
W(M),&M>c_p.
\end{cases}
\end{equation}

\begin{theorem}\label{thm_main}
Let $N\ge2$, $p>1$ and $c>0$, and let $v$ be a bounded positive solution of \eqref{eq_Pc}. Then
\begin{equation}\label{eq_main}
\frac12|\nabla v|^2\le W(v)-\min_{[0,M]}W
\qquad\text{in }\overline{\mathbb R^N_+}.
\end{equation}
If $M\ge c_p$, equality holds at every point where $v$ attains the value $M$.
\end{theorem}

The normalization by the minimum of the potential is sharp already in one dimension. At the critical boundary value it gives the following rigidity criterion.

\begin{theorem}\label{thm_char}
Let $v$ be a bounded positive solution of \eqref{eq_Pc} with $c=c_p$. Then
\begin{equation}\label{eq_bdry}
\sup_{x'\in\mathbb R^{N-1}}\big(\partial_Nv(x',0)\big)^2
\le \max\left\{0,\frac{2M^{p+1}}{p+1}-M^2\right\}.
\end{equation}
Moreover $v=w_0(x_N)$ if and only if $M=c_p$. Equivalently, every other bounded positive solution satisfies $\sup v>c_p$.
\end{theorem}

Our rigidity theorem closes the problem in dimensions up to five.

\begin{theorem}[Low-dimensional rigidity]\label{thm_unique_Nle5_allp}
Let $2\le N\le5$ and $p>1$. Then $w_0(x_N)$ is the unique bounded positive solution of \eqref{eq_Pc} with $c=c_p$.
\end{theorem}

The construction starts from the kernel $Z=-w_0'$ of the one-dimensional linearized operator. The projected amplitude is governed to leading order by
\begin{equation}\label{eq_LE}
-\Delta_{x'}\lambda=C_p\lambda^2,\qquad
C_p:=\frac{p(p-1)}{2}\,
\frac{\int_0^\infty w_0^{p-2}Z^3\,dt}{\int_0^\infty Z^2\,dt}>0.
\end{equation}
The exponent two crosses the Serrin threshold in boundary dimension four and the Sobolev threshold in boundary dimension six. This is the source of the dimensional transition.

\begin{problem}\label{prob_N6}
For $N=6$ and $p>1$, is $w_0(x_N)$ the unique bounded positive solution of \eqref{eq_Pc} with $c=c_p$?
\end{problem}

We leave Problem \ref{prob_N6} open. The scalar projection alone does not decide it: in $\mathbb R^5$ the inequality $-\Delta\lambda\ge C\lambda^2$ admits positive solutions, the pure equation $-\Delta\lambda=C\lambda^2$ does not, and an $O(\lambda^3)$ correction can regularize the singular $|x'|^{-2}$ profile; see Proposition \ref{prop_nogo}.

In higher dimensions we obtain genuine non-one-dimensional solutions.

\begin{theorem}[High-dimensional nonuniqueness]\label{thm_construction}
Let $N\ge8$ and $p>1$. There is $\alpha_0>0$ such that for every $\alpha\in(0,\alpha_0)$ problem \eqref{eq_Pc} with $c=c_p$ has a bounded positive solution $v_\alpha$, radially symmetric in $x'$, which is not a function of $x_N$ alone. Moreover
\begin{equation}
v_\alpha=w_0+\lambda_\alpha Z+\lambda_\alpha^2h+O(\alpha^3)
\end{equation}
in $L^\infty(\mathbb R^N_+)$, where $\lambda_\alpha(0)=\alpha$ and
\begin{equation}\label{eq_tailC}
\lambda_\alpha(x')=\frac{2(N-5)}{C_p}|x'|^{-2}(1+o(1))
\qquad\text{as }|x'|\to\infty.
\end{equation}
\end{theorem}

The boundary dimension $n=6$, namely $N=7$, is critical for \eqref{eq_LE}. Let $\kappa(p)$ be the cubic coefficient of the exact reduced nonlinearity defined in Remark \ref{c:rem_kappa}.

\begin{theorem}[Critical nonuniqueness in dimension seven]\label{thm_construction_critical}
Let $N=7$, $p>1$, and assume $\kappa(p)>0$. Then, for every sufficiently small $\alpha>0$, problem \eqref{eq_Pc} with $c=c_p$ has a bounded positive non-one-dimensional solution $v_\alpha$, radial in $x'$. In particular, there exists $p_+\in(1,2)$ such that this conclusion holds for every $p>p_+$. If
\[
a_\alpha(x'):=\frac1\nu\int_0^\infty (v_\alpha(x',t)-w_0(t))Z(t)\,dt,
\]
then $a_\alpha(0)=\alpha$, $a_\alpha(x')\to0$, and with $\varepsilon^2=C_p\alpha$,
\begin{equation}\label{eq_critical_amp_bounds}
0<a_\alpha(x')\le C\alpha(1+\varepsilon|x'|)^{-2},\qquad
|\nabla a_\alpha(x')|\le C\alpha\varepsilon(1+\varepsilon|x'|)^{-3}.
\end{equation}
Moreover $v_\alpha=w_0+a_\alpha Z+a_\alpha^2h+O(\alpha^3)$ in $L^\infty$. Writing $M_\alpha(t)=e^{2t}a_\alpha(e^t)$, one has
\begin{equation}\label{eq_critical_fowler_main}
M_\alpha(t)=M_{E_\alpha}(t+\vartheta_\alpha)+O_\alpha\big(e^{-\delta t}\big),\qquad
E_\alpha=-\frac{1152}{35}\frac{\kappa}{C_p^3}\alpha+o(\alpha)<0,
\end{equation}
for a positive periodic Fowler solution of $\ddot M-4M+C_pM^2=0$, a phase $\vartheta_\alpha\in\mathbb R$ and some $\delta>0$ independent of $\alpha$; see Proposition \ref{c:prop_forced_fowler} for its value. Its neck, bulge and minimal period satisfy
\begin{equation}
M_-(E_\alpha)=\frac{24}{\sqrt{35}}\frac{\sqrt\kappa}{C_p^{3/2}}\alpha^{1/2}(1+o(1)),\quad
M_+(E_\alpha)=\frac6{C_p}+O(\alpha),\quad
\mathcal T_\alpha^{\rm F}=\frac12|\log\alpha|+O(1).
\end{equation}
\end{theorem}

The coefficient $\kappa$ is continuous, satisfies $\kappa(2)=15/1024$, and is positive for every $p\ge2$. Hence it is positive on $(p_+,\infty)$ for some $p_+<2$. On the other hand, $\kappa(p)<0$ for $p>1$ sufficiently close to $1$; see Appendix \ref{app_kappa}. Thus the reduced critical problem exhibits a genuine sign transition in $(1,2)$. Determining the exact transition, and the half-space problem in the complementary range of exponents in $N=7$, remain open.

The two construction regimes require different closures. For $N\ge8$ the exact reduced equation is solved first and its linearization is inverted in weighted spaces. The admissible interval of weights is $(2,(N-3)/2)$, which is nonempty exactly for $N\ge8$. In $N=7$ that interval collapses at the critical weight. Rather than invert the resonant reduced operator, we solve the projected radial equation nonlinearly for each transverse correction; a Pohozaev quantity preserves positivity and the critical $r^{-2}$ bounds, and the remaining transverse equation is closed by a Schauder--Tychonoff fixed point.

We also record in Proposition \ref{prop_flux} an exact boundary-flux identity. It quantifies the tangential leakage of the one-dimensional energy and is useful independently of the classification.

\subsection{Organization}
Sections \ref{sec_prelim}--\ref{sec_reduced} establish the a priori estimates, the projection method, and uniqueness for $N\le5$. Sections \ref{c:sec_prelim}--\ref{c:sec_fixed} develop the cell problem, reduced equation, linear theory, and fixed-point construction for $N\ge8$. Section \ref{c:sec_critical_full} treats the resonant case $N=7$, and Appendix \ref{app_kappa} analyzes the cubic coefficient $\kappa$.

\section{Preliminaries}\label{sec_prelim}

Throughout Sections \ref{sec_prelim} to \ref{sec_char} we fix $N\ge2$, $p>1$, a constant $c>0$ and a bounded positive solution $v$ of \eqref{eq_Pc}, and we keep the notation \eqref{eq_w0}, \eqref{eq_W} and $M=\sup v$. Since $v$ is continuous up to $\partial\mathbb{R}^N_+$ and equals $c$ there, we have $M\ge c$; in particular $M\ge c_p$ when $c=c_p$. Lemma \ref{lem_order}, Theorem \ref{thm_char} and Proposition \ref{prop_flux} require $c=c_p$ and we say so in their statements; everything else in these sections holds for every $c>0$.

\begin{lemma}\label{lem_order}
Assume $c=c_p$. Then $v(x)\ge w_0(x_N)$ for every $x\in\overline{\mathbb{R}^N_+}$. Moreover, either $v(x)=w_0(x_N)$ for every $x$, or $v(x)>w_0(x_N)$ for every $x\in\mathbb{R}^N_+$ and $\partial_Nv(x',0)>0$ for every $x'\in\mathbb{R}^{N-1}$.
\end{lemma}

\begin{proof}
The first assertion is stated in \cite[(1.3)]{So1}. We recall why the argument of \cite[Theorem 6.1]{FW} gives it. For $t\in\mathbb{R}$ put $v_t(x)=w_0(x_N+t)$, and set
\[
t^\star=\inf\left\{t\in\mathbb{R}:v\ge v_s\text{ in }\mathbb{R}^N_+\text{ for every }s\ge t\right\}.
\]
Step 1 of the proof of \cite[Theorem 6.1]{FW} shows that this set is nonempty, so $t^\star<\infty$. That step uses only that $v_t\le(2p)^{-1/(p-1)}$ in $\overline{\mathbb{R}^N_+}$ for $t$ large, together with $v>v_t$ on $\partial\mathbb{R}^N_+$, and both hold here because $v=c_p>w_0(t)=\max_{\overline{\mathbb{R}^N_+}}v_t$ on $\partial\mathbb{R}^N_+$ for every $t>0$. Suppose $t^\star>0$. Then $w_0(t)<c_p$ for every $t$ in a neighborhood of $t^\star$, so the strict boundary inequality $v>v_t$ on $\partial\mathbb{R}^N_+$ persists for $t$ slightly below $t^\star$, and Step 2 of \cite[Theorem 6.1]{FW} applies without change and contradicts the definition of $t^\star$. Hence $t^\star\le0$, and by continuity $v\ge v_0=w_0(x_N)$.

For the second assertion, put $\phi=v-w_0(x_N)\ge0$. Both $v$ and $w_0$ solve the same equation, so
\begin{equation}
-\Delta\phi+a(x)\phi=0\quad\text{in }\mathbb{R}^N_+,\qquad \phi=0\ \text{ on }\partial\mathbb{R}^N_+,
\end{equation}
where $a(x)=1-\left(v^p-w_0^p\right)/\phi$ when $\phi(x)\ne0$ and $a(x)=1-pw_0^{p-1}$ otherwise. Since $0\le w_0,v\le\max\{c_p,M\}$, the function $a$ is bounded. Since $\Delta\phi=a\phi\le a^+\phi$ and $\phi\ge0$, the function $\phi$ is a supersolution of the operator $\Delta-a^+$, whose zeroth order coefficient is nonpositive. If $\phi$ vanishes at some interior point, then $\phi$ attains its minimum there and \cite[Theorem 3.5]{GT} gives $\phi\equiv0$. Otherwise $\phi>0$ in $\mathbb{R}^N_+$, and the Hopf boundary point lemma \cite[Lemma 3.4]{GT} applied at any $(x',0)$, using that the half-space satisfies the interior ball condition at every boundary point, gives $\partial_N\phi(x',0)>0$. Since $w_0'(0)=0$, this reads $\partial_Nv(x',0)>0$.
\end{proof}

\begin{lemma}\label{lem_reg}
The following hold for every $c>0$ and every bounded positive solution $v$ of \eqref{eq_Pc}. Part (ii) holds, with the same proof, for every bounded solution of the equation in \eqref{eq_Pc} equal to $c$ on $\partial\mathbb{R}^N_+$, whether or not it is positive.
\begin{enumerate}[label=(\roman*)]
\item[(0)] For every $K>0$ one has $\inf\{v(x):x\in\overline{\mathbb{R}^N_+},\ x_N\le K\}>0$.
\item $v\in C^\infty(\overline{\mathbb{R}^N_+})$, and for every $K>0$ the norm of $v$ in $C^3\left(\overline{\mathbb{R}^N_+}\cap\{x_N\le K\}\right)$ is bounded by a constant depending only on $N$, $p$, $M$, $K$ and $c$.
\item $|\nabla v|+|D^2v|\le C$ on $\overline{\mathbb{R}^N_+}$, with $C=C(N,p,M)$.
\item There is $C=C(N,p,M,c)>0$ such that
\begin{equation}\label{eq_decay}
v(x)+|\nabla v(x)|\le Ce^{-x_N/\sqrt2}\qquad\text{for every }x\in\overline{\mathbb{R}^N_+}.
\end{equation}
\item There is $C=C(N,p,M)>0$ such that $v(x)\le Ce^{-x_N}$ for every $x\in\overline{\mathbb{R}^N_+}$. The rate $1$ is that of the linearization of \eqref{eq_Pc} at $0$, and it is attained by $w_0$.
\end{enumerate}
\end{lemma}

\begin{proof}
We prove (ii) first, since it does not use (0). Then we prove (0), and then (i).

For (ii), let $f_*(s)=s-|s|^{p-1}s$, which is Lipschitz on $[0,M]$ because $p>1$, and note that $\Delta v=f_*(v)$ with $\|f_*(v)\|_{L^\infty}\le C(p,M)$. Fix $q>N$. Interior and boundary $W^{2,q}$ estimates on the sets $B_1(x)\cap\overline{\mathbb{R}^N_+}$, with the constant Dirichlet datum $c$ on the flat part of the boundary, together with the embedding $W^{2,q}\hookrightarrow C^{1,\gamma}$, give $\|v\|_{C^{1,\gamma}(B_{1/2}(x)\cap\overline{\mathbb{R}^N_+})}\le C$ with $C$ independent of $x$. Then $f_*(v)$ is bounded in $C^{0,\gamma}$ on these sets, uniformly in $x$, and Schauder estimates give $\|v\|_{C^{2,\gamma}(B_{1/4}(x)\cap\overline{\mathbb{R}^N_+})}\le C$, again uniformly in $x$.

For (0), suppose that $v(x^j)\to0$ for some sequence with $x^j_N\le K$. By (ii) the translates $v^j(x):=v(x'+(x^j)',x_N)$ are bounded in $C^{2,\gamma}$ uniformly, so a subsequence converges in $C^2_{\rm loc}(\overline{\mathbb{R}^N_+})$ to a solution $V\ge0$ of the equation with $V=c$ on $\partial\mathbb{R}^N_+$, and $V(\bar x)=0$ for some $\bar x$ with $\bar x_N\le K$. Since $V\ge0$, the point $\bar x$ is not on the boundary, where $V=c>0$, so it is an interior minimum. But $\Delta V-V=-V^p\le0$ and $V\ge0$, so the strong maximum principle \cite[Theorem 3.5]{GT} forces $V\equiv0$ in $\mathbb{R}^N_+$, contradicting $V=c$ on the boundary.

For (i), by (0) and the positivity of $v$ the function $v$ takes, on each slab $\{x_N\le K\}$, values in a compact subset of $(0,\infty)$, where $s\mapsto|s|^{p-1}s$ is smooth. Since $\partial\mathbb{R}^N_+$ is flat and $v$ is constant on it, Schauder estimates up to the boundary \cite[Chapter 6]{GT} and a bootstrap give $v\in C^\infty(\overline{\mathbb{R}^N_+})$, with constants on $\{x_N\le K\}$ depending only on $N$, $p$, $M$, $K$ and $c$, the dependence on $c$ entering through the infimum in (0).

For (iii), fix $R_0>0$ such that $v(x)^{p-1}\le\frac12$ whenever $x_N\ge R_0$, which is possible by the uniform decay in \eqref{eq_Pc}. Then $\Delta v=v-v^p\ge\frac12v$ in $D:=\{x_N>R_0\}$. Let $\psi(x)=Me^{-(x_N-R_0)/\sqrt2}$ with $1/\sqrt2$, so that $\Delta\psi=\frac12\psi$, and let $w=v-\psi$. Then $\Delta w\ge\frac12w$ in $D$, $w\le0$ on $\partial D$, and $w\to0$ as $x_N\to\infty$ uniformly in $x'$. Fix $\mathsf b\in(0,2^{-1/2})$ and, for $\varepsilon>0$, put
\[
w_\varepsilon(x)=w(x)-\varepsilon\sum_{i=1}^{N-1}\cosh(\mathsf b x_i).
\]
Then $\Delta w_\varepsilon\ge\frac12w_\varepsilon+\varepsilon\left(\frac12-\mathsf b^2\right)\sum_{i<N}\cosh(\mathsf b x_i)>\frac12 w_\varepsilon$ in $D$, while $w_\varepsilon\le 0$ on $\partial D$, $w_\varepsilon\to-\infty$ as $|x'|\to\infty$ and $w_\varepsilon\le 0$ for $x_N$ large. If $\sup_Dw_\varepsilon>0$, then the supremum is attained at some $z$ in the interior of $D$, where $\Delta w_\varepsilon(z)\le0$ and $\frac12w_\varepsilon(z)>0$, a contradiction. Hence $w_\varepsilon\le0$ in $D$ for every $\varepsilon>0$, and letting $\varepsilon\to0$ gives $v\le\psi$ in $D$. Since $|s-|s|^{p-1}s|\le Cs$ for $s\in[0,M]$, the gradient estimate on the ball of radius $1$ centered at $x$ gives $|\nabla v(x)|\le C\sup_{B_1(x)\cap\overline{\mathbb{R}^N_+}}v$ for every $x$ with $x_N\ge R_0+1$, whence \eqref{eq_decay} there. For $x_N\le R_0+1$ the bound \eqref{eq_decay} follows from (ii) and $v\le M$ after enlarging $C$.

For (iv), write the equation as $\Delta v=a(x)v$ with $a=1-v^{p-1}$. By (iii) there are $C_1>0$ and $\beta:=(p-1)/\sqrt2>0$ with
\[
a(x)\ge 1-C_1e^{-\beta x_N}\qquad\text{in }\overline{\mathbb{R}^N_+}.
\]
Put $\varepsilon:=\frac{2C_1}{\beta(2+\beta)}$ and $\psi(t):=\exp\left(-t-\varepsilon e^{-\beta t}\right)$. Writing $\psi=e^{-t+g}$ with $g(t)=-\varepsilon e^{-\beta t}$, we have $\psi''/\psi=(g'-1)^2+g''$, and since $g'=\varepsilon\beta e^{-\beta t}$ and $g''=-\varepsilon\beta^2e^{-\beta t}$,
\[
\frac{\psi''}{\psi}=1-\varepsilon\beta(2+\beta)e^{-\beta t}+\varepsilon^2\beta^2e^{-2\beta t}
=1-2C_1e^{-\beta t}+\varepsilon^2\beta^2e^{-2\beta t}.
\]
Choose $R_2\ge R_0$ so large that $\varepsilon^2\beta^2e^{-\beta t}\le C_1$ and $a\ge\frac12$ for $t\ge R_2$. Then $\psi''\le\left(1-C_1e^{-\beta t}\right)\psi\le a\psi$ on $\{x_N\ge R_2\}$. Let $\lambda:=M/\psi(R_2)$ and $w:=v-\lambda\psi$. Then $w\le0$ on $\{x_N=R_2\}$, $w\to0$ as $x_N\to\infty$ uniformly in $x'$, $w$ is bounded, and $\Delta w\ge aw\ge\frac12w$ wherever $w\ge0$ in $\{x_N>R_2\}$. The argument used for (iii), with the same functions $\cosh(\mathsf b x_i)$, gives $w\le0$ in $\{x_N>R_2\}$. Since $\psi(t)\le e^{-t}$, we get $v\le\lambda e^{-x_N}$ there, and on $\{x_N\le R_2\}$ we use $v\le M\le Me^{R_2}e^{-x_N}$.
\end{proof}

\section{The gradient bound}

Set
\begin{equation}
G(s):=W(s)-\min_{[0,M]}W,
\qquad
P(x):=|\nabla v(x)|^2-2G(v(x)).
\end{equation}
Then $G'=W'$, so $\Delta v=G'(v)$, and by \eqref{eq_minW}
\begin{equation}\label{eq_Gsign}
G\ge0\ \text{ on }[0,M],
\qquad
G(0)=-\min_{[0,M]}W\ge0.
\end{equation}
Theorem \ref{thm_main} is the assertion that $P\le0$ on $\overline{\mathbb{R}^N_+}$. Every statement of this section is valid, with the same proof and with the same function $G$, for any bounded positive solution $\tilde v$ of \eqref{eq_Pc}, with the same $c$, such that $\sup\tilde v\le M$, because \eqref{eq_Gsign} only requires $G\ge0$ on the range of the solution. We use this in Step 1 of the proof of Theorem \ref{thm_main}. We use the differential inequality for $P$ of Caffarelli, Garofalo and Seg\`ala.

\begin{lemma}\label{lem_CGS}
At every $x\in\mathbb{R}^N_+$ with $\nabla v(x)\ne0$,
\begin{equation}\label{eq_CGS}
|\nabla v|^2\,\Delta P-2G'(v)\,\nabla v\cdot\nabla P\ \ge\ \frac{|\nabla P|^2}{2}.
\end{equation}
In particular, on every open set $U\subset\mathbb{R}^N_+$ with $\overline U$ compact and $\overline U\subset\{\nabla v\ne0\}$ one has $\Delta P+b\cdot\nabla P\ge0$ with $b=-2G'(v)|\nabla v|^{-2}\nabla v\in L^\infty(U;\mathbb{R}^N)$.
\end{lemma}

\begin{proof}
Inequality \eqref{eq_CGS} is formula (2.7) of \cite{CGS}, quoted as formula (8.3) in \cite{FV1}; it holds for any $C^3$ solution of $\Delta v=G'(v)$ at any point where $\nabla v\ne0$, and $v\in C^\infty(\overline{\mathbb{R}^N_+})$ by Lemma \ref{lem_reg}. The second assertion follows by dividing by $|\nabla v|^2$, which is bounded away from $0$ on $\overline U$, and discarding the nonnegative right hand side.
\end{proof}

\begin{lemma}\label{lem_neumann}
$\partial_NP(x',0)=0$ for every $x'\in\mathbb{R}^{N-1}$.
\end{lemma}

\begin{proof}
Since $v\equiv c$ on $\partial\mathbb{R}^N_+$, all the derivatives of $v$ along the boundary vanish there, so on $\partial\mathbb{R}^N_+$ we have $\partial_iv=0$ and $\partial_{ij}v=0$ for $i,j\le N-1$, and the equation gives
\[
\partial_{NN}v=\Delta v-\sum_{i\le N-1}\partial_{ii}v=G'(v).
\]
Therefore, on $\partial\mathbb{R}^N_+$,
\[
\partial_NP=2\sum_{i=1}^{N}\partial_iv\,\partial_{iN}v-2G'(v)\partial_Nv
=2\partial_Nv\,\partial_{NN}v-2G'(v)\partial_Nv=0. \qedhere
\]
\end{proof}

\begin{proof}[Proof of Theorem \ref{thm_main}]
Put $\Lambda:=\sup_{\overline{\mathbb{R}^N_+}}P$, which is finite by Lemma \ref{lem_reg}(ii). We claim that $\Lambda\le0$, which is \eqref{eq_main}. Suppose on the contrary that $\Lambda>0$, and let $(x^k)_k\subset\overline{\mathbb{R}^N_+}$ be such that $P(x^k)\to\Lambda$.

\emph{Step 1: the sequence $(x^k_N)_k$ is bounded, and a limit solution exists.} By \eqref{eq_Pc} and Lemma \ref{lem_reg}(iii), $v(x)\to0$ and $|\nabla v(x)|\to0$ as $x_N\to\infty$, uniformly in $x'$, so $P(x)\to-2G(0)\le0$ as $x_N\to\infty$, uniformly in $x'$. Since $\Lambda>0$, the sequence $(x^k_N)_k$ is bounded, and after passing to a subsequence $x^k_N\to\bar x_N\in[0,\infty)$. Let $v^k(x):=v(x'+(x^k)',x_N)$. Each $v^k$ is a bounded positive solution of \eqref{eq_Pc} with the same bounds and the same modulus of decay as $v$. By Lemma \ref{lem_reg}(i) the family $(v^k)_k$ is bounded in $C^3$ on every slab $\overline{\mathbb{R}^N_+}\cap\{x_N\le K\}$, so after a further subsequence $v^k\to V$ in $C^2_{loc}(\overline{\mathbb{R}^N_+})$, where $V$ solves $\Delta V=G'(V)$ in $\mathbb{R}^N_+$, $V=c$ on $\partial\mathbb{R}^N_+$, $0\le V\le M$, and $V(x)\to0$ as $x_N\to\infty$ uniformly in $x'$, with the same modulus as $v$. Moreover $V$ satisfies the bound of Lemma \ref{lem_reg}(0) with the same constants as $v$, since that bound passes to the limit, so $V>0$ on $\overline{\mathbb{R}^N_+}$ and Lemma \ref{lem_reg} applies to $V$ as well. Writing $P_V:=|\nabla V|^2-2G(V)$ we have
\begin{equation}\label{eq_limitP}
P_V\le\Lambda\ \text{ on }\overline{\mathbb{R}^N_+},
\qquad
P_V(\bar x)=\Lambda,\qquad \bar x:=(0',\bar x_N).
\end{equation}

\emph{Step 2: $\nabla V$ does not vanish where $P_V=\Lambda$.}  If $y\in\overline{\mathbb{R}^N_+}$ satisfies $\nabla V(y)=0$ and $P_V(y)=\Lambda$, then $\Lambda=-2G(V(y))\le0$ by \eqref{eq_Gsign}, since $V(y)\in[0,M]$. This contradicts $\Lambda>0$. In particular $\nabla V(\bar x)\ne0$.

\emph{Step 3: $P_V$ is constant on a component.} Let $\mathcal{N}:=\{x\in\overline{\mathbb{R}^N_+}:\nabla V(x)\ne0\}$, a relatively open subset of $\overline{\mathbb{R}^N_+}$, and let $\mathcal{O}$ be its connected component containing $\bar x$. We claim that $P_V\equiv\Lambda$ on $\mathcal{O}$. The set $S:=\{x\in\mathcal{O}:P_V(x)=\Lambda\}$ is relatively closed in $\mathcal{O}$ and contains $\bar x$ by \eqref{eq_limitP}. We show that it is relatively open in $\mathcal{O}$.

Let $x_0\in S$. Suppose first that $x_0\in\mathbb{R}^N_+$. Choose $r>0$ with $\overline{B_r(x_0)}\subset\mathcal{O}\cap\mathbb{R}^N_+$, so that $\inf_{\overline{B_r(x_0)}}|\nabla V|>0$. By Lemma \ref{lem_CGS}, $P_V$ is a subsolution of a uniformly elliptic operator with bounded coefficients and no zeroth order term in $B_r(x_0)$, and it attains its maximum $\Lambda$ at the interior point $x_0$. The strong maximum principle \cite[Theorem 3.5]{GT} gives $P_V\equiv\Lambda$ on $B_r(x_0)$.

Suppose now that $x_0\in\partial\mathbb{R}^N_+$. Choose $r>0$ such that the compact set $\overline{B_{2r}(x_0)}\cap\overline{\mathbb{R}^N_+}$ is contained in $\mathcal{O}$, and put $B_r^+:=B_r(x_0)\cap\overline{\mathbb{R}^N_+}$, so that $\inf_{B_r^+}|\nabla V|>0$. The open half-ball $B_r(x_0)\cap\mathbb{R}^N_+$ is connected, so by the strong maximum principle, as in the previous paragraph, either $P_V\equiv\Lambda$ there, or $P_V<\Lambda$ there. In the first case $P_V\equiv\Lambda$ on $B_r^+$ by continuity and we are done. In the second case, pick a ball $B\subset B_r(x_0)\cap\mathbb{R}^N_+$ with $x_0\in\partial B$, which exists because $x_0$ is a boundary point of a half-space. Then $P_V<\Lambda=P_V(x_0)$ in $B$, and the Hopf boundary point lemma \cite[Lemma 3.4]{GT}, applied in $B$ to the operator of Lemma \ref{lem_CGS}, gives $\partial_\nu P_V(x_0)>0$, where $\nu=-e_N$ is the outer normal of $B$ at $x_0$. Thus $\partial_NP_V(x_0)<0$, which contradicts Lemma \ref{lem_neumann} applied to $V$. So only the first case occurs.

Hence $S$ is relatively open and relatively closed in the connected set $\mathcal{O}$, so $S=\mathcal{O}$.

\emph{Step 4: conclusion.} Two cases are possible. If $\mathcal{O}=\overline{\mathbb{R}^N_+}$, then $P_V\equiv\Lambda$ on $\overline{\mathbb{R}^N_+}$, while $P_V(x)\to-2G(0)\le0$ as $x_N\to\infty$ by Step 1 applied to $V$. Hence $\Lambda\le0$, a contradiction. If $\mathcal{O}\ne\overline{\mathbb{R}^N_+}$, then, $\overline{\mathbb{R}^N_+}$ being connected and $\mathcal{O}$ being nonempty and relatively open, the relative boundary of $\mathcal{O}$ in $\overline{\mathbb{R}^N_+}$ contains a point $y$. Since $\mathcal{N}$ is relatively open and $y\notin\mathcal{O}$, we must have $y\notin\mathcal{N}$, that is $\nabla V(y)=0$. On the other hand $P_V(y)=\Lambda$ by Step 3 and continuity. This contradicts Step 2.

In both cases we reach a contradiction, so $\Lambda\le0$ and \eqref{eq_main} holds.

Finally, assume $M\ge c_p$, which holds whenever $c\ge c_p$ because $M\ge c$, and let $x_0\in\overline{\mathbb{R}^N_+}$ be a point with $v(x_0)=M$. If $x_0\in\mathbb{R}^N_+$, then $x_0$ is an interior maximum point of $v$, so $\nabla v(x_0)=0$. If $x_0\in\partial\mathbb{R}^N_+$, then $M=c$, and since $M\ge c_p$ while $W(M)\le0$ on $[c_p,\infty)$, either $M>c_p$, in which case $\min_{[0,M]}W=W(M)$ and \eqref{eq_main} forces $\nabla v(x_0)=0$, or $M=c_p$, in which case $\min_{[0,M]}W=0=W(c_p)=W(v(x_0))$ and \eqref{eq_main} forces $\nabla v(x_0)=0$ again. In both cases the left hand side of \eqref{eq_main} vanishes at $x_0$, while the right hand side equals $W(M)-\min_{[0,M]}W$, which is $0$ by \eqref{eq_minW} whether $M=c_p$ or $M>c_p$.
\end{proof}

\begin{remark}
The proof uses the decay in \eqref{eq_Pc} only through the statement that $P(x)\to-2G(0)\le0$ as $x_N\to\infty$, uniformly in $x'$. It uses the flatness of $\partial\mathbb{R}^N_+$ only through Lemma \ref{lem_neumann}, and it uses no property of the boundary value $c$ beyond its being a positive constant. If $\Omega$ is a smooth domain and $v$ is constant on $\partial\Omega$, the computation of Lemma \ref{lem_neumann}, combined with the identity $\partial_{\nu\nu}v=\Delta v-(N-1)\mathcal{H}\,\partial_\nu v$ on $\partial\Omega$, where $\mathcal{H}$ is the mean curvature of $\partial\Omega$ with respect to the exterior normal $\nu$, gives $\partial_\nu P=-2(N-1)\mathcal{H}\,(\partial_\nu v)^2$. This is nonpositive when $\mathcal{H}\ge0$, which is the inequality needed in Step 3, and it is the mechanism behind the gradient bounds of \cite{FV2}. The remaining ingredient, the compactness under translations parallel to the boundary, is specific to the half-space, so we do not pursue this extension here.
\end{remark}

\section{Characterization and the boundary-flux identity}\label{sec_char}

\begin{proof}[Proof of Theorem \ref{thm_char}]
Let $x'\in\mathbb{R}^{N-1}$. On $\partial\mathbb{R}^N_+$ we have $v=c_p$, so $W(v)=W(c_p)=0$ and $\nabla v(x',0)=g(x')e_N$ with $g(x')=\partial_Nv(x',0)$. Theorem \ref{thm_main} therefore gives
\[
\frac12 g(x')^2\le -\min_{[0,M]}W=\max\{0,-W(M)\}=\frac12\max\left\{0,\frac{2M^{p+1}}{p+1}-M^2\right\},
\]
which is \eqref{eq_bdry}.

Recall that $M\ge c_p$. Suppose now that $M=c_p$. Then the right hand side of \eqref{eq_bdry} vanishes by \eqref{eq_minW}, so $g\equiv0$ on $\partial\mathbb{R}^N_+$. Let $\phi=v-w_0(x_N)$, which is nonnegative by Lemma \ref{lem_order}. Then $\phi=0$ and $\partial_N\phi=g-w_0'(0)=0$ on $\partial\mathbb{R}^N_+$. By the last part of Lemma \ref{lem_order}, either $\phi\equiv0$, or $\partial_N\phi(x',0)>0$ for every $x'$. The second alternative is excluded, so $v(x)=w_0(x_N)$.

Conversely, if $v(x)=w_0(x_N)$ then $M=\max w_0=c_p$. This proves the stated equivalence, and the last assertion is its contrapositive.
\end{proof}

\begin{remark}
The bound \eqref{eq_bdry} is sharp at leading order in the following sense. The kernel of the linearization of \eqref{eq_Pc} at $w_0$ in the $x_N$ variable is spanned by $Z=-w_0'$, so a branch of solutions emanating from $w_0$ would have the form $v_a=w_0+aZ+o(a)$ as $a\to0^+$. For the functions $w_0+aZ$ one computes $\partial_N(w_0+aZ)(0)=aZ'(0)=a\beta$ with $\beta:=c_p^p-c_p>0$, and $\max_{t\ge0}(w_0+aZ)(t)=c_p+\frac{\beta}{2}a^2+O(a^3)$, whence
\[
\frac{2M^{p+1}}{p+1}-M^2=-2W(M)=-2W'(c_p)\left(M-c_p\right)+O\!\left((M-c_p)^2\right)=\beta^2a^2+O(a^3),
\]
because $-W'(c_p)=c_p^p-c_p=\beta$. The two sides of \eqref{eq_bdry} thus agree to leading order along such a family. The computation is formal, since the functions $w_0+aZ$ are not solutions of \eqref{eq_Pc}, but it shows that the constant $1$ in \eqref{eq_bdry} is the one dictated by the linearization.
\end{remark}

\begin{proposition}\label{prop_flux}
Let $N\ge2$, $p>1$, and let $v$ be a bounded positive solution of \eqref{eq_Pc} with $c=c_p$. Put $g(x'):=\partial_Nv(x',0)$ and
\begin{equation}\label{eq_Phi}
\Phi(x'):=\int_0^\infty \partial_Nv(x',t)\,\nabla_{x'}v(x',t)\,dt.
\end{equation}
Then $\Phi\in C^1(\mathbb R^{N-1};\mathbb R^{N-1})\cap L^\infty$ and
\begin{equation}\label{eq_flux}
\operatorname{div}_{x'}\Phi=\frac12 g^2\qquad\text{in }\mathbb R^{N-1}.
\end{equation}
Consequently
\begin{equation}\label{eq_fluxbound}
\int_{|x'|<R}g(x')^2\,dx'\le CR^{N-2}\qquad(R\ge1).
\end{equation}
\end{proposition}

For its proof define
\begin{equation}\label{eq_H}
H:=\frac12(\partial_Nv)^2-\frac12|\nabla_{x'}v|^2-W(v).
\end{equation}

\begin{proof}[Proof of Proposition \ref{prop_flux}]
By Lemma \ref{lem_reg}(iii) the integrand in \eqref{eq_Phi} is bounded by $C^2e^{-2\alpha x_N}$, and by Lemma \ref{lem_reg}(ii) and (iii) its first order $x'$ derivatives are bounded by $Ce^{-\alpha x_N}$. Hence $\Phi$ is well defined, $\|\Phi\|_{L^\infty}\le C^2/(2\alpha)$, and $\Phi$ is of class $C^1$ with derivatives obtained by differentiating under the integral sign.

Let $H$ be as in \eqref{eq_H}. Using $\partial_{NN}v=\Delta v-\Delta_{x'}v=W'(v)-\Delta_{x'}v$ we get, at every point of $\mathbb{R}^N_+$,
\begin{align*}
\partial_NH
&=\partial_Nv\,\partial_{NN}v-\nabla_{x'}v\cdot\nabla_{x'}\partial_Nv-W'(v)\partial_Nv\\
&=-\partial_Nv\,\Delta_{x'}v-\nabla_{x'}v\cdot\nabla_{x'}\partial_Nv
=-\operatorname{div}_{x'}\left(\partial_Nv\,\nabla_{x'}v\right).
\end{align*}
Fix $x'$ and integrate this identity over $x_N\in(0,T)$:
\[
H(x',T)-H(x',0)=-\operatorname{div}_{x'}\int_0^T\partial_Nv\,\nabla_{x'}v\,dx_N.
\]
On $\partial\mathbb{R}^N_+$ we have $\nabla_{x'}v=0$ and $W(v)=W(c_p)=0$, so $H(x',0)=\frac12g(x')^2$. By \eqref{eq_decay}, $H(x',T)\to0$ as $T\to\infty$, and $\int_0^T\partial_Nv\,\nabla_{x'}v\,dx_N\to\Phi$ in $C^1_{loc}(\mathbb{R}^{N-1})$. Letting $T\to\infty$ gives \eqref{eq_flux}.

Integrating \eqref{eq_flux} over $\{|x'|<R\}$ and applying the divergence theorem gives
\[
\frac12\int_{\{|x'|<R\}}g^2\,dx'=\int_{\{|x'|=R\}}\Phi\cdot\nu\,d\sigma\le \|\Phi\|_{L^\infty}\,\left|\{|x'|=R\}\right| ,
\]
which is \eqref{eq_fluxbound} since $\left|\{|x'|=R\}\right|=(N-1)\omega_{N-1}R^{N-2}$, where $\omega_{N-1}$ is the volume of the unit ball of $\mathbb{R}^{N-1}$. Dividing by $\left|\{|x'|<R\}\right|=\omega_{N-1}R^{N-1}$ gives the last assertion.
\end{proof}

\section{Projection and low-dimensional rigidity}\label{sec_reduced}

In this section we project the defect onto the kernel of the one-dimensional linearized operator. The exact projected identity, together with an automatic tail estimate, yields uniqueness for $N\le5$. We then record the Lane--Emden trichotomy that explains both the unresolved dimension $N=6$ and the construction regimes above it. We keep $\phi=v-w_0(x_N)$ and $Z=-w_0'$, and we assume $v\not\equiv w_0(x_N)$, so that $\phi>0$ in $\mathbb{R}^N_+$ by Lemma \ref{lem_order}.

Differentiating the equation of $w_0$ shows that
\begin{equation}\label{eq_Zeq}
-Z''+\left(1-pw_0^{p-1}\right)Z=0\ \text{ on }(0,\infty),
\qquad Z(0)=0,\qquad Z>0\ \text{ on }(0,\infty),
\end{equation}
so $0$ is the principal eigenvalue of the operator $-\partial_{tt}+1-pw_0^{p-1}$ on $(0,\infty)$ with the Dirichlet condition. Subtracting the two equations and applying Taylor's formula with Lagrange remainder gives
\begin{equation}\label{eq_phieq}
\Delta\phi=\left(1-pw_0^{p-1}\right)\phi-\frac{p(p-1)}{2}\xi^{p-2}\phi^2
\quad\text{in }\mathbb{R}^N_+,
\qquad w_0<\xi<v.
\end{equation}
Multiplying \eqref{eq_phieq} by $Z$, integrating over $x_N\in(0,\infty)$ and using \eqref{eq_Zeq} together with $Z(0)=0=\phi(x',0)$ to cancel the boundary terms, Sourdis obtained in \cite[(2.9)]{So2} the identity
\begin{equation}\label{eq_proj}
-\nu\,\Delta_{x'}\lambda(x')=\frac{p(p-1)}{2}\int_0^\infty\xi^{p-2}\phi^2Z\,dx_N ,
\end{equation}
where
\[
\lambda(x'):=\frac1\nu\int_0^\infty\phi(x',x_N)Z(x_N)\,dx_N>0,
\qquad \nu:=\int_0^\infty Z^2\,dt.
\]
Bounding $\xi^{p-2}$ from below and applying the Cauchy--Schwarz inequality to the right hand side of \eqref{eq_proj}, one obtains $-\Delta_{x'}\lambda\ge C\lambda^2$ with $C>0$ when $1<p<3$, and $-\Delta_{x'}\lambda\ge C\lambda^3$ when $p=3$. If the exponent does not exceed the Serrin exponent of $\mathbb{R}^{N-1}$, the Liouville theorem for positive supersolutions of the Lane--Emden inequality, recalled in \cite[Theorem A.1]{So2}, is contradicted, so the assumption $\phi>0$ fails and $\phi\equiv0$. This is the argument of \cite{So2}, and it gives $2\le N\le5$ for $1<p<3$ and $2\le N\le4$ for $p=3$. We write
\begin{equation}\label{eq_psg}
p_{sg}(n)=\begin{cases}\infty & \text{if }n\le2,\\ \dfrac{n}{n-2} & \text{if }n\ge3,\end{cases}
\end{equation}
for the Serrin exponent of $\mathbb{R}^n$, so that \cite[Theorem A.1]{So2} states that $-\Delta\lambda\ge\lambda^q$ has no positive solution in $\mathbb{R}^n$ when $1<q\le p_{sg}(n)$.

\subsection{A quadratic projected inequality}

\begin{lemma}\label{lem_two}
Let $p\ge2$ and $a,b>0$. Then
\begin{equation}
(a+b)^p-a^p-pa^{p-1}b\ \ge\ \max\left\{\frac{p(p-1)}{2}a^{p-2}b^2,\ b^p\right\}.
\end{equation}
\end{lemma}

\begin{proof}
Taylor's formula with Lagrange remainder gives $(a+b)^p-a^p-pa^{p-1}b=\frac{p(p-1)}{2}\xi^{p-2}b^2$ for some $\xi\in(a,a+b)$, and $\xi^{p-2}\ge a^{p-2}$ because $p\ge2$. For the second bound, put $h(b)=(a+b)^p-a^p-pa^{p-1}b-b^p$. Then $h(0)=0$ and $h'(b)=p\left[(a+b)^{p-1}-a^{p-1}-b^{p-1}\right]\ge0$, since $t\mapsto t^{p-1}$ is superadditive on $[0,\infty)$ for $p\ge2$. Hence $h\ge0$.
\end{proof}

\begin{proposition}\label{prop_conditional}
Let $p\ge2$ and $N\ge2$. Assume that there are $K\ge0$ and $\vartheta\in(0,1)$ such that
\begin{equation}\label{eq_noconc}
\int_K^\infty\phi(x',x_N)Z(x_N)\,dx_N\ \le\ \vartheta\,\nu\lambda(x')\qquad\text{for every }x'\in\mathbb{R}^{N-1}.
\end{equation}
Then $-\Delta_{x'}\lambda\ge C\lambda^2$ in $\mathbb{R}^{N-1}$ with $C=\frac{p(p-1)(1-\vartheta)^2\nu}{2A_K}>0$, where $A_K=\int_0^Kw_0^{2-p}Z\,dx_N<\infty$. If moreover $2\le N\le5$, then $v\equiv w_0(x_N)$.
\end{proposition}

\begin{proof}
By \eqref{eq_noconc} we have $\int_0^K\phi Z\,dx_N\ge(1-\vartheta)\nu\lambda$. The Cauchy--Schwarz inequality and Lemma \ref{lem_two} give
\[
(1-\vartheta)^2\nu^2\lambda^2\le\left(\int_0^K\phi Z\,dx_N\right)^2\le A_K\int_0^Kw_0^{p-2}\phi^2Z\,dx_N\le\frac{2\nu A_K}{p(p-1)}\left(-\Delta_{x'}\lambda\right),
\]
and $A_K<\infty$ because $w_0^{2-p}Z$ is continuous on $[0,K]$. If $2\le N\le5$ then $N-1\le4$ and $p_{sg}(N-1)\ge2$ by \eqref{eq_psg}, so \cite[Theorem A.1]{So2} is contradicted unless $\phi\equiv0$.
\end{proof}

\subsection{Automatic tail control and uniqueness}

\begin{proposition}[Automatic tail non-concentration]\label{prop_automatic_tail}
Let $N\ge2$, $p>1$, and let $v\not\equiv w_0$ be a bounded positive solution of \eqref{eq_Pc} with $c=c_p$.  Put $\phi=v-w_0$ and let $\lambda$ be given by \eqref{eq_proj}.  Then there are $T>0$ and $C>0$ such that
\begin{equation}\label{eq_tail_pointwise_auto}
 \phi(x',t)\le C\lambda(x')e^{-(t-T)/\sqrt2}
 \qquad\text{for every }x'\in\mathbb R^{N-1},\ t\ge T.
\end{equation}
Consequently, for every $\vartheta\in(0,1)$ there exists $K\ge T$ such that
\begin{equation}\label{eq_auto_noconc}
 \int_K^\infty\phi(x',t)Z(t)\,dt
 \le \vartheta\nu\lambda(x')
 \qquad\text{for every }x'\in\mathbb R^{N-1}.
\end{equation}
Thus the hypothesis \eqref{eq_noconc} of Proposition~\ref{prop_conditional} is automatic.
\end{proposition}

\begin{proof}
By Lemma~\ref{lem_order}, $\phi>0$ and
\begin{equation}\label{eq_phi_linear_auto}
 -\Delta\phi+a(x',t)\phi=0,
 \qquad
 a(x',t):=1-\frac{v(x',t)^p-w_0(t)^p}{v(x',t)-w_0(t)}.
\end{equation}
The coefficient is bounded.  Since both $v$ and $w_0$ tend to zero uniformly as $t\to\infty$, after increasing $T$ we have
\begin{equation}\label{eq_tail_coercive_auto}
 a(x',t)\ge\frac12
 \qquad\text{for all }x'\in\mathbb R^{N-1},\ t\ge T-2.
\end{equation}
On the other hand, \eqref{eq_proj} gives $-\Delta_{x'}\lambda>0$, so $\lambda$ is positive and superharmonic on $\mathbb R^{N-1}$.

We first compare the trace of $\phi$ at height $T$ with $\lambda$.  The positive solution of the linear equation \eqref{eq_phi_linear_auto} satisfies the interior Harnack inequality on unit balls, with a constant independent of the tangential center.  Hence, for a fixed $r\in(0,1/2)$,
\[
 \phi(x',T)
 \le C\int_{B_r^{N-1}(x')\times(T-r,T+r)}\phi(y,s)\,dy\,ds.
\]
Since $Z$ is bounded below by a positive constant on the fixed interval $[T-r,T+r]$, the right hand side is at most
\[
 C_T\int_{B_r^{N-1}(x')}\lambda(y)\,dy.
\]
The mean-value inequality for the positive superharmonic function $\lambda$ therefore yields
\begin{equation}\label{eq_trace_lambda_auto}
 \phi(x',T)\le C_T\lambda(x')
 \qquad\text{for all }x'.
\end{equation}

Set $\kappa_0=1/\sqrt2$ and, for $t\ge T$, define
\[
 U(x',t):=C_Te^{-\kappa_0(t-T)}\lambda(x').
\]
Because $\kappa_0^2=1/2$ and $-\Delta_{x'}\lambda\ge0$,
\[
 \left(-\Delta+\frac12\right)U
 =C_Te^{-\kappa_0(t-T)}\big(-\Delta_{x'}\lambda\big)\ge0.
\]
By \eqref{eq_phi_linear_auto}--\eqref{eq_tail_coercive_auto},
\[
 \left(-\Delta+\frac12\right)\phi
 =-\left(a-\frac12\right)\phi\le0
 \qquad\text{in }\{t>T\}.
\]
At $t=T$ we have $\phi\le U$ by \eqref{eq_trace_lambda_auto}, and both functions tend to zero uniformly as $t\to\infty$.  The maximum principle on the unbounded half-space, justified as in Lemma~\ref{lem_reg} by subtracting a small sum of tangential $\cosh$ barriers, gives $\phi\le U$.  This proves \eqref{eq_tail_pointwise_auto}.

Finally $Z(t)\le C e^{-t}$, so for $K\ge T$,
\[
 \int_K^\infty\phi(x',t)Z(t)\,dt
 \le C_T\lambda(x')\int_K^\infty e^{-(t-T)/\sqrt2}e^{-t}\,dt
 \le C_T'e^{-(1+1/\sqrt2)K}\lambda(x').
\]
Choosing $K$ large enough gives \eqref{eq_auto_noconc}.
\end{proof}

\begin{proof}[Proof of Theorem \ref{thm_unique_Nle5_allp}]
For $1<p<3$ this is the theorem of Sourdis recalled after \eqref{eq_proj}; it is still needed for $1<p<2$ because the lower bound of Lemma \ref{lem_two}, on which Proposition \ref{prop_conditional} rests, requires $p\ge2$.  Let $p\ge2$ and suppose that $v\not\equiv w_0$.  Proposition~\ref{prop_automatic_tail} verifies the hypothesis \eqref{eq_noconc} of Proposition~\ref{prop_conditional}; hence
\[
 -\Delta_{x'}\lambda\ge C\lambda^2\qquad\text{in }\mathbb R^{N-1}.
\]
Since $N-1\le4$, Proposition~\ref{prop_trichotomy}(i) gives a contradiction.  Thus $v\equiv w_0$.
\end{proof}

\subsection{The reduced equation}

We now describe what the projected relation looks like in the regime that governs the Liouville step, namely where $\phi$ is small. Suppose $\phi$ is small and varies slowly in $x'$. Since the kernel of the operator in \eqref{eq_Zeq} is spanned by $Z$ and the rest of its spectrum is positive, the profile of $\phi$ in the $x_N$ variable is then $Z$ to leading order, say $\phi(x',x_N)\simeq\lambda(x')Z(x_N)$, and $\xi\simeq w_0$. Substituting into \eqref{eq_proj} gives
\[
-\Delta_{x'}\lambda\int_0^\infty Z^2\,dt\simeq\frac{p(p-1)}{2}\lambda^2\int_0^\infty w_0^{p-2}Z^3\,dt ,
\]
that is \eqref{eq_LE}. Both integrals are finite and positive for every $p>1$: near $t=0$ the integrands are bounded because $w_0(0)=c_p>0$ and $Z(t)=O(t)$, and as $t\to\infty$ one has $w_0(t)^{p-2}Z(t)^3=O(e^{-(p+1)t})$ and $Z(t)^2=O(e^{-2t})$. For instance $C_2=15/32$ and $C_3=3/2$.

Equation \eqref{eq_LE} is the Lane--Emden equation with exponent $2$ in dimension $n=N-1$. Its two classical thresholds are the Serrin exponent $n/(n-2)$, below which positive supersolutions do not exist, and the Sobolev exponent $(n+2)/(n-2)$, below which positive solutions do not exist. Both are crossed by the exponent $2$ in the range of dimensions relevant here.

\begin{proposition}\label{prop_trichotomy}
Let $C>0$ and recall that $n=N-1$. Then:
\begin{enumerate}[label=(\roman*)]
\item if $2\le N\le5$, the inequality $-\Delta\lambda\ge C\lambda^2$ has no positive solution in $\mathbb{R}^n$;
\item if $N=6$, the inequality $-\Delta\lambda\ge C\lambda^2$ has positive solutions in $\mathbb{R}^n$, but the equation $-\Delta\lambda=C\lambda^2$ has none;
\item if $N\ge7$, the equation $-\Delta\lambda=C\lambda^2$ has positive solutions in $\mathbb{R}^n$. For $N=7$ they include the explicit family
\[
\lambda_\rho(y)=\frac{24}{C}\cdot\frac{\rho^2}{\left(\rho^2+|y|^2\right)^2},\qquad \rho>0.
\]
\end{enumerate}
\end{proposition}

\begin{proof}
(i) For $n\le2$ a positive superharmonic function on $\mathbb{R}^n$ is constant, and a positive constant does not satisfy $-\Delta\lambda\ge C\lambda^2$. For $3\le n\le4$ we have $2\le n/(n-2)$, and the claim is \cite[Theorem A.1]{So2}.

(ii) For $n=5$ the function $\lambda(y)=k(1+|y|^2)^{-1}$ with $k>0$ small is a positive solution of $-\Delta\lambda\ge C\lambda^2$, by \cite[Remark A.2]{So2}, since $2>5/3=n/(n-2)$. On the other hand $2<7/3=(n+2)/(n-2)$, so the equation has no positive solution by the theorem of Gidas and Spruck \cite{GS}.

(iii) For $n\ge 6$ we have $2\ge(n+2)/(n-2)$. For $n=6$ the exponent $2$ is the Sobolev exponent, and a direct computation shows that $U_\rho(y)=24\rho^2(\rho^2+|y|^2)^{-2}$ satisfies $-\Delta U_\rho=U_\rho^2$ in $\mathbb{R}^6$, so $\lambda_\rho=U_\rho/C$ solves $-\Delta\lambda_\rho=C\lambda_\rho^2$. For $n\ge7$ the exponent is supercritical and the positive radial entire solutions with $\lambda(0)=m$ exist for every $m>0$ and decay like $|y|^{-2}$; see Proposition \ref{c:prop_lambda} below, which proves the statement we actually use, for the exact reduced nonlinearity rather than for the pure power.
\end{proof}

The case $N=6$ is the one the trichotomy leaves in the middle, and the projection cannot decide it. Indeed the Lyapunov--Schmidt expansion $\phi=\lambda Z+O(\lambda^2)$ determines the right hand side of \eqref{eq_proj} only up to a relative error of order $\lambda$, that is, it gives $-\Delta_{x'}\lambda=C_p\lambda^2+O(\lambda^3)$ and no more. The following example shows that no relation of that accuracy can be enough, however small the constant in front of the error.

\begin{proposition}\label{prop_nogo}
Let $C>0$ and $\kappa>0$. Then the function
\[
\lambda(y)=\frac{2/C}{\varsigma^2+|y|^2},\qquad \varsigma^2=\frac{\kappa}{2C^2},
\]
is a positive bounded solution of $-\Delta\lambda=C\lambda^2+\kappa\lambda^3$ in $\mathbb{R}^5$.
\end{proposition}

\begin{proof}
Let $A>0$, $c>0$ and $\lambda(y)=A(c+|y|^2)^{-1}$ on $\mathbb{R}^n$. A direct computation gives
\[
\Delta\lambda=A\,\frac{(8-2n)|y|^2-2nc}{\left(c+|y|^2\right)^{3}}.
\]
For $n=5$ this is $-2A\left(|y|^2+5c\right)\left(c+|y|^2\right)^{-3}$, so that
\[
-\Delta\lambda=\frac{2A\left(c+|y|^2\right)+8Ac}{\left(c+|y|^2\right)^{3}}
=\frac{2}{A}\lambda^2+\frac{8c}{A^2}\lambda^3
=\frac{2}{A}\lambda^2\left(1+\frac{4c}{A}\lambda\right).
\]
Taking $A=2/C$ and $c=\kappa/(2C^2)$ gives $2/A=C$ and $8c/A^2=2C^2c=\kappa$.
\end{proof}

The same computation in $\mathbb{R}^4$ gives $-\Delta\lambda=8Ac\left(c+|y|^2\right)^{-3}$, which carries no $\lambda^2$ term at all, so the example has no counterpart in the dimension where the projection does decide the question. Proposition \ref{prop_nogo} is a statement about the model inequality, not about \eqref{eq_Pc}, but it locates the obstruction: in $\mathbb{R}^5$ the profile $|y|^{-2}$ solves the limit equation away from the origin, and a positive perturbation of relative size $O(\lambda)$ is exactly what regularizes it at the origin. Settling $N=6$ therefore requires a global argument rather than a refinement of \eqref{eq_proj}.

Problem \ref{prob_N6} is the remaining ambient dimension between the full rigidity theorem for $N\le5$ and the construction regimes beginning in $N=7$. Proposition \ref{prop_nogo} explains why a refinement of the scalar projection alone is unlikely to settle it: in boundary dimension five, a cubic correction of the same relative size as the Lyapunov--Schmidt error can create a bounded positive entire profile. We therefore do not pursue $N=6$ further here.

The dimension $N=7$ is instead exactly Sobolev critical for the leading reduced equation. Its first nontrivial correction is the cubic coefficient $\kappa$ studied in Section \ref{c:sec_cell}; when $\kappa>0$, Section \ref{c:sec_critical_full} lifts the resulting critical radial branch to the half-space. For $N\ge8$ the noncritical weighted inverse of Sections \ref{c:sec_linear}--\ref{c:sec_fixed} applies for every $p>1$.

\section{Preliminaries for the nonuniqueness construction}\label{c:sec_prelim}

In Sections \ref{c:sec_prelim} to \ref{c:sec_fixed} we assume $N\ge8$, so that $n=N-1\ge7$; Section \ref{c:sec_critical_full} treats $N=7$, and there we state which results of the present sections carry over. We recall that, writing $v=w_0(x_N)+\phi$ and
\begin{equation}\label{eq_phiL}
L\phi=R(x_N,\phi),\qquad L:=-\Delta+1-pw_0^{p-1},\qquad
R(t,s):=f\big(w_0(t)+s\big)-f\big(w_0(t)\big)-f'\big(w_0(t)\big)s ,
\end{equation}
where $f(v):=|v|^{p-1}v$, problem \eqref{eq_Pc} with $c=c_p$ is equivalent to \eqref{eq_phiL} together with $\phi=0$ on $\partial\mathbb{R}^N_+$, $\phi$ bounded and $\phi(x',x_N)\to0$ as $x_N\to\infty$ uniformly in $x'$, \emph{provided the solution is positive}. We solve \eqref{eq_phiL} with the odd nonlinearity $f$, which is defined for every real argument and is of class $C^1$ with $f'(v)=p|v|^{p-1}$, and we prove positivity afterwards, in the proof of Theorem \ref{thm_construction}. Since $f$ is only $C^{1,p-1}$ at the origin when $1<p<2$, the remainder $R$ is quadratic only where $|s|$ is small compared with $w_0$; the two bounds we use, valid for every $p>1$, are
\begin{equation}\label{c:eq_Rtwo}
|R(t,s)|\le C\,w_0(t)^{p-2}s^2\ \ \text{if }|s|\le\tfrac12w_0(t),
\qquad
|R(t,s)|\le C\,|s|^p\ \ \text{if }|s|\ge\tfrac14w_0(t).
\end{equation}

Since $w_0(t)=k_pe^{-t}(1+O(e^{-(p-1)t}))$ for an explicit $k_p>0$, there are $0<\underline c<\overline c$ with
\begin{equation}\label{c:eq_Zdecay}
\underline c\,\min\{t,1\}e^{-t}\le Z(t)\le\overline c\,\min\{t,1\}e^{-t},\qquad t>0.
\end{equation}
For $\mu\in(0,1)$ we write $\|u\|_{\mathcal X_\mu}:=\sup_{t>0}e^{\mu t}|u(t)|$ and $\mathcal X_\mu$ for the corresponding space of continuous functions on $[0,\infty)$ vanishing at $t=0$. These spaces are too wide for the cell problem, because an element of $\mathcal X_\mu$ with $\mu<1$ may be far larger than $w_0$ at infinity, so that $w_0+u$ need not be positive and $R(t,u)$ need not be quadratic in $u$. We therefore also use the space adapted to $Z$,
\begin{equation}
\mathcal Y:=\Big\{u\in C([0,\infty)):\ \|u\|_{\mathcal Y}:=\sup_{t>0}\frac{|u(t)|}{\theta(t)}<\infty\Big\},\qquad \theta(t):=\min\{t,1\}e^{-t},
\end{equation}
so that $\theta\asymp Z$ by \eqref{c:eq_Zdecay}, and $\theta\le Cw_0$ on $[0,\infty)$ because $w_0(0)>0$ and $w_0\asymp e^{-t}$ at infinity. Consequently, if $\|u\|_{\mathcal Y}\le\epsilon_0$ with $C\epsilon_0\le\frac12$, then
\begin{equation}\label{c:eq_Rquad}
w_0+u\ge\tfrac12w_0>0\quad\text{and}\quad |R(t,u)|\le C_*\,w_0(t)^{p-2}u(t)^2\le C_*'\,\|u\|_{\mathcal Y}^2\,\theta(t)\,e^{-(p-1)t} ,
\end{equation}
because $R(t,u)=\frac{p(p-1)}{2}\xi^{p-2}u^2$ with $\xi$ between $w_0$ and $w_0+u$, hence $\xi\asymp w_0$, and because $w_0^{p-2}\theta^2\le C\theta e^{-(p-1)t}$, both sides behaving like $e^{-pt}$ at infinity. The gain over $\theta$ is the factor $e^{-(p-1)t}$, and $p-1>0$ is all that is used below. It is \eqref{c:eq_Rquad}, and not a bound of the form $|R|\le Cu^2$, that makes the nonlinearity quadratic for every $p>1$; the latter is false when $1<p<2$ and $u\gg w_0$. For $\sigma\in\mathbb{R}$ and $\varepsilon>0$ we write
\[
\|u\|_{\sigma,\varepsilon}:=\sup_{y\in\mathbb{R}^n}(1+\varepsilon|y|)^{\sigma}|u(y)| ,
\]
and $\|u\|_\sigma:=\|u\|_{\sigma,1}$. All the weighted statements below are invariant under the scaling $u\mapsto u(\cdot/\varepsilon)$, so it suffices to prove them for $\varepsilon=1$.

\begin{lemma}\label{c:lem_gap}
Let $\mathcal L=-\partial_{tt}+1-pw_0^{p-1}$ on $L^2(0,\infty)$ with the Dirichlet condition at $0$. Then $\sigma_{\rm ess}(\mathcal L)=[1,\infty)$, the number $0$ is a simple eigenvalue of $\mathcal L$ with eigenfunction $Z$, and
\[
\gamma:=\inf\big(\sigma(\mathcal L)\setminus\{0\}\big)>0.
\]
Moreover, in the range $1<p\le2$ the gap is explicit:
\begin{equation}\label{c:eq_exact_gap}
\gamma=(p-1)(3-p).
\end{equation}
\end{lemma}

\begin{proof}
Since $pw_0(t)^{p-1}\to0$ as $t\to\infty$, Weyl's theorem gives $\sigma_{\rm ess}(\mathcal L)=\sigma_{\rm ess}(-\partial_{tt}+1)=[1,\infty)$. Differentiating $-w_0''+w_0=w_0^p$ gives $\mathcal LZ=0$, and $Z>0$ on $(0,\infty)$ with $Z(0)=0$. A positive eigenfunction of a Sturm--Liouville operator is the ground state and the ground state is simple, so $0=\min\sigma(\mathcal L)$ is a simple eigenvalue, isolated because it lies below $\sigma_{\rm ess}$. Hence $\gamma>0$.

For the last assertion put $m=p-1$ and $x=mt/2$.  Since $c_p^m=(p+1)/2$, one obtains
\[
\mathcal L=\frac{m^2}{4}\left[-\partial_{xx}+\frac4{m^2}-\ell(\ell+1)\operatorname{sech}^2x\right],
\qquad
\ell=1+\frac2m.
\]
The P\"oschl--Teller operator $-\partial_{xx}-\ell(\ell+1)\operatorname{sech}^2x$ has bound-state eigenvalues $-(\ell-j)^2$, $j=0,1,\ldots<\ell$, with parity $(-1)^j$.  The Dirichlet half-line problem therefore selects the odd values of $j$. The state $j=1$ gives the eigenvalue $0$ and is precisely $Z$. If $1<p<2$, the next odd state is $j=3$, and its eigenvalue is
\[
\frac{m^2}{4}\left[\frac4{m^2}-\left(\frac2m-2\right)^2\right]=m(2-m)=(p-1)(3-p)<1.
\]
All subsequent odd states have larger eigenvalues, while the essential spectrum begins at $1$.  At $p=2$ this level reaches the threshold $1$, so the same formula gives $\gamma=1$. This proves \eqref{c:eq_exact_gap}.
\end{proof}

We write $\Pi u:=\frac1\nu\left(\int_0^\infty u(t)Z(t)\,dt\right)Z$ for the projection on $\mathbb{R}Z$ and $\Pi^\perp=\mathrm{Id}-\Pi$, acting in the variable $t=x_N$ for each fixed $x'$.

\begin{lemma}\label{c:lem_Linv}
Let $\mu\in(0,1)$ and $f\in\mathcal X_{\mu}$ with $\int_0^\infty fZ\,dt=0$. Then there is a unique $\omega\in\mathcal X_\mu$ with $\omega(0)=0$, $\mathcal L\omega=f$ and $\int_0^\infty \omega Z\,dt=0$, and $\|\omega\|_{\mathcal X_\mu}\le C\|f\|_{\mathcal X_{\mu}}$ with $C=C(p,\mu)$. The operator $\Pi^\perp$ maps $\mathcal X_\mu$ into itself. Moreover, if $|f(t)|\le A\,\theta(t)e^{-\kappa t}$ for some $\kappa>0$ and $\int_0^\infty fZ\,dt=0$, then for every $\tau\in[0,\gamma/2]$ the solution $\omega_\tau$ of $(\mathcal L+\tau)\omega_\tau=f$, $\omega_\tau(0)=0$, $\int_0^\infty\omega_\tau Z\,dt=0$ belongs to $\mathcal Y$ and $\|\omega_\tau\|_{\mathcal Y}\le C(p,\kappa)A$, uniformly in $\tau$.
\end{lemma}

\begin{proof}
Let $y_2$ be the solution of $\mathcal Ly_2=0$ with $y_2(0)=1$, $y_2'(0)=0$; it grows like $e^{t}$, and the Wronskian $W=Zy_2'-Z'y_2$ is the constant $-Z'(0)$. The function
\[
\omega_0(t)=\frac{1}{W}\left(Z(t)\int_0^ty_2f\,ds+y_2(t)\int_t^\infty Zf\,ds\right)
\]
solves $\mathcal Lu_0=f$; it vanishes at $t=0$ because $Z(0)=0$ and $\int_0^\infty Zf=0$, and $|\omega_0|\le C\|f\|_{\mathcal X_{\mu}}e^{-\mu t}$: indeed $|Z|\le Ce^{-t}$ and $|y_2|\le Ce^{t}$ give
$\big|Z\int_0^ty_2f\big|\le Ce^{-t}\int_0^te^{(1-\mu)s}ds\le Ce^{-\mu t}$ and
$\big|y_2\int_t^\infty Zf\big|\le Ce^{t}\int_t^\infty e^{-(1+\mu)s}ds\le Ce^{-\mu t}$, because $0<\mu<1$.
The assertion about $\Pi^\perp$ holds because $Z\in\mathcal X_\mu$ for $\mu<1$. For the last assertion, the same representation with the two solutions of $(\mathcal L+\tau)\cdot=0$, which behave like $e^{\pm\sqrt{1+\tau}\,t}$ at infinity and are $O(t)$ and $O(1)$ at the origin, gives
\[
\Big|Z\int_0^ty_2f\Big|\le CA\,e^{-t}\int_0^te^{-\kappa s}\,ds\le \frac{CA}{\kappa}\,e^{-t},
\qquad
\Big|y_2\int_t^\infty Zf\Big|\le CA\,e^{t}\!\int_t^\infty\! e^{-(2+\kappa)s}ds\le\frac{CA}{2+\kappa}\,e^{-(1+\kappa)t},
\]
using $|f|\le Ae^{-(1+\kappa)t}$; both are $O(A\theta(t))$, since they vanish at $t=0$ with bounded derivative. The bounds are uniform for $\tau$ in a compact subset of $[0,\gamma)$. For $\tau=0$, subtracting the multiple of $Z$ that enforces the orthogonality preserves both the equation and the boundary condition. For $\tau>0$ no subtraction is possible, since $(\mathcal L+\tau)Z=\tau Z\ne0$, but none is needed: testing $(\mathcal L+\tau)\omega_\tau=f$ with $Z$ gives $\tau\int_0^\infty\omega_\tau Z\,dt=\int_0^\infty fZ\,dt=0$, so the orthogonality is automatic. Uniqueness holds because the only solution of $\mathcal L\omega=0$ with $\omega(0)=0$ that lies in $\mathcal X_\mu$ is a multiple of $Z$.
\end{proof}

\section{The cell problem and the reduced nonlinearity}\label{c:sec_cell}

For $\tau>0$ consider the one-dimensional problem
\begin{equation}\label{c:eq_cell}
-\zeta''+(1+\tau)\zeta=(w_0+\zeta)^p-w_0^p\ \text{ in }(0,\infty),\qquad
\zeta(0)=0,\quad \zeta>0,\quad \zeta(\infty)=0.
\end{equation}
Equivalently $\mathcal L\zeta+\tau\zeta=R(\cdot,\zeta)$. Problem \eqref{c:eq_cell} has no solution for $\tau=0$: indeed a solution would give a bounded positive solution of \eqref{eq_Pc} with $c=c_p$ depending on $x_N$ only, and the only such solution is $w_0$ itself, i.e. $\zeta\equiv0$.

\begin{proposition}\label{c:prop_cell}
There are $\tau_0>0$ and a continuous map $\tau\mapsto\zeta^\tau\in\mathcal Y$, defined for $\tau\in(0,\tau_0)$ and of class $C^3$ as a map into $\mathcal X_\mu$ for each fixed $\mu\in(0,1)$, such that $\zeta^\tau$ solves \eqref{c:eq_cell}. The normalization
\begin{equation}\label{c:eq_norm}
a(\tau):=\frac1\nu\int_0^\infty\zeta^\tau Z\,dt
\end{equation}
is of class $C^3$ and strictly increasing in $\tau$ with $a(\tau)=\tau/C_p+O(\tau^2)$, so it may be inverted: writing $\zeta_a:=\zeta^{\tau}$ and $T(a):=\tau$ for $a=a(\tau)$, we obtain a family $a\mapsto(\zeta_a,T(a))$, continuous into $\mathcal Y\times\mathbb R$ and of class $C^3$ into $\mathcal X_\mu\times\mathbb R$, $a\in(0,a_0)$, with $\frac1\nu\int_0^\infty\zeta_aZ\,dt=a$ and
\begin{equation}\label{c:eq_Cp}
T(a)=C_pa+O(a^2),\qquad C_p:=\frac{p(p-1)}{2}\cdot\frac{\int_0^\infty w_0^{p-2}Z^3\,dt}{\int_0^\infty Z^2\,dt}>0.
\end{equation}
Moreover $\zeta_a=aZ+a^2h+O(a^3)$ in $\mathcal X_\mu$ for every fixed $\mu\in(0,1)$, where $h$ is the solution given by Lemma \ref{c:lem_Linv} of
\begin{equation}\label{c:eq_h}
\mathcal Lh=\tfrac{p(p-1)}{2}w_0^{p-2}Z^2-C_pZ,\qquad \int_0^\infty hZ\,dt=0 ,
\end{equation}
and $\zeta_a>0$ on $(0,\infty)$ with $\|\zeta_a\|_{\mathcal Y}\le Ca$, so that in particular $\zeta_a(t)\le Ca\,\theta(t)\le Cae^{-t}$ uniformly in $a\in(0,a_0)$.
\end{proposition}

\begin{proof}
We use $\tau$, not $a$, as the parameter, and we rescale: put $\zeta=\tau\,\Xi$. Since $R(t,\tau\Xi)=\frac{p(p-1)}{2}\xi^{p-2}\tau^2\Xi^2$ with $\xi$ between $w_0$ and $w_0+\tau\Xi$, the cell equation \eqref{c:eq_cell} is equivalent to
\begin{equation}\label{c:eq_cellresc}
(\mathcal L+\tau)\Xi=\tau\,\mathcal R_\tau(\Xi),\qquad
\mathcal R_\tau(\Xi):=\tau^{-2}R(\cdot,\tau\Xi) ,
\end{equation}
It is convenient to record an explicit form. On the $\mathcal Y$ ball one has $|\tau\Xi|\le\frac12w_0$, so writing $q:=\Xi/w_0$, which satisfies $|q|\le C$ and $|\tau q|\le\frac12$, and
\[
H(\varsigma):=\frac{(1+\varsigma)^p-1-p\varsigma}{\varsigma^2},\qquad H(0)=\tfrac{p(p-1)}{2},
\]
which is analytic on $\{|\varsigma|\le\frac12\}$, we get
\begin{equation}\label{c:eq_Rrep}
\mathcal R_\tau(\Xi)=w_0^{p}\,q^2\,H(\tau q).
\end{equation}
Hence $\Xi\mapsto\mathcal R_\tau(\Xi)$ is, for each fixed $\tau\in[0,\tau_0)$, smooth from $\{\|\Xi\|_{\mathcal Y}\le\epsilon_0/\tau_0\}$ into $\{|f|\le A\theta e^{-(p-1)t}\}$ with derivatives bounded uniformly in $\tau$, and, what will matter for the regularity in the parameter, $\tau\mapsto\mathcal R_\tau(\Xi)$ is smooth up to $\tau=0$ for each fixed $\Xi$ in that ball, with $\mathcal R_0=\frac{p(p-1)}{2}w_0^{p-2}\Xi^2$; both statements are immediate from \eqref{c:eq_Rrep} and the analyticity of $H$. This is the point at which the space $\mathcal Y$ is needed: on it the nonlinearity is genuinely quadratic and $w_0+\zeta$ stays positive, for every $p>1$.

Write $\Xi=bZ+\eta$ with $\int_0^\infty\eta Z\,dt=0$. Then \eqref{c:eq_cellresc} splits into
\begin{align}
(\mathcal L+\tau)\eta&=\tau\,\Pi^\perp\mathcal R_\tau(bZ+\eta),\label{c:eq_split1}\\
b\nu&=\int_0^\infty\mathcal R_\tau(bZ+\eta)\,Z\,dt ,\label{c:eq_split2}
\end{align}
equation \eqref{c:eq_split1} being the component of \eqref{c:eq_cellresc} orthogonal to $Z$, and \eqref{c:eq_split2} being obtained by testing \eqref{c:eq_cellresc} with $Z$ and dividing by $\tau$, using $\mathcal LZ=0$. Consider
\[
\mathcal F(\tau,b,\eta):=\Big(\eta-\tau(\mathcal L+\tau)^{-1}\Pi^\perp\mathcal R_\tau(bZ+\eta),\ \ b\nu-\int_0^\infty\mathcal R_\tau(bZ+\eta)Z\,dt\Big),
\]
which by the last part of Lemma \ref{c:lem_Linv} maps a neighborhood of $(1/C_p,0)$ in $\mathbb{R}\times\mathcal Y^\perp$ into $\mathcal Y^\perp\times\mathbb{R}$ for each $\tau\in[0,\tau_0)$, is smooth in $(b,\eta)$ with derivatives bounded uniformly in $\tau$, and both $\mathcal F$ and $D_{(b,\eta)}\mathcal F$ are continuous in $\tau$, uniformly on the ball, because $\|\tau(\mathcal L+\tau)^{-1}\|\le C\tau$ and because $\mathcal R_\tau$ and its $\Xi$ derivative converge to those of $\mathcal R_0$ as $\tau\to0$. It is \emph{not} differentiable in $\tau$ at $\tau=0$ as a map into $\mathcal Y$, because $\tau\mapsto(\mathcal L+\tau)^{-1}$ is not: the formal derivative $-\mathcal L^{-2}$ produces the secular term $te^{-t}\notin\mathcal Y$. We therefore do not apply the implicit function theorem in $\tau$, but the uniform contraction principle. At $\tau=0$,
\[
\mathcal F(0,b,\eta)=\Big(\eta,\ b\nu-\tfrac{p(p-1)}{2}\int_0^\infty w_0^{p-2}(bZ+\eta)^2Z\,dt\Big).
\]
In particular, on the line $\eta=0$ the second component is $b\nu-C_p\nu b^2$, so $\mathcal F(0,1/C_p,0)=0$. Its differential there, which we denote by $D_0$, sends an increment $(\beta,\mathsf h)$ to
\begin{equation}\label{c:eq_D0}
D_0(\beta,\mathsf h)=\Big(\mathsf h,\ -\nu\beta-\ell(\mathsf h)\Big),
\qquad
\ell(\mathsf h):=\frac{p(p-1)}{C_p}\int_0^\infty w_0^{p-2}Z^2\,\mathsf h\,dt ,
\end{equation}
the term $\ell$ coming from differentiating $\int_0^\infty\mathcal R_0(bZ+\eta)Z\,dt$ in $\eta$, and $-\nu$ from $\partial_b[b\nu-C_p\nu b^2]=(1-2C_pb)\nu$ at $b=1/C_p$. Note that $\ell$ does not vanish on $Z^\perp$ in general. Nevertheless $D_0$ is an isomorphism, being triangular: given $(f,s)$, the unique preimage is $\mathsf h=f$, $\beta=-(s+\ell(f))/\nu$. The functional $\ell$ is bounded on $\mathcal Y^\perp$ and on $\mathcal X_\mu^\perp$ alike, since $w_0^{p-2}Z^2\le Ce^{-pt}$. Consequently the map
\[
\mathcal G_\tau(b,\eta):=(b,\eta)-\big[D_{(b,\eta)}\mathcal F(0,1/C_p,0)\big]^{-1}\mathcal F(\tau,b,\eta)
\]
is, for $\tau$ small, a contraction of a fixed small ball around $(1/C_p,0)$ into itself, with contraction factor $\frac12$ uniformly in $\tau$, and it depends continuously on $\tau$. The uniform contraction principle gives a unique fixed point $(b(\tau),\eta_\tau)$, depending continuously on $\tau$, with $b(0)=1/C_p$ and $\eta_0=0$, hence a solution $\zeta^\tau=\tau(b(\tau)Z+\eta_\tau)\in\mathcal Y$ of \eqref{c:eq_cell} with $\|\zeta^\tau\|_{\mathcal Y}\le C\tau$.

\emph{Regularity in the parameter.} We differentiate in the variables $(b,\eta)$, not in $\zeta$, and we are careful to distinguish two statements: $\tau\mapsto\mathcal F(\tau,U)$ is smooth for $U$ fixed in the $\mathcal Y$ ball, and $U\mapsto\mathcal F(\tau,U)$ has a linearization which extends boundedly to $\mathbb{R}\times\mathcal X_{\widehat\mu}^\perp$. We do \emph{not} claim that $\mathcal F$ is a smooth map on an open ball of $\mathcal X_{\widehat\mu}$, which would be false for $1<p<2$.

Fix $\mu\in(0,1)$ and choose
\begin{equation}\label{c:eq_nested}
\widehat\mu\in\Big(\max\big\{\mu,\tfrac{3-p}{2}\big\},\,1\Big) ,
\end{equation}
a nonempty interval because $p>1$. We prove $C^3$ regularity into $\mathcal X_{\widehat\mu}$; since $\mathcal X_{\widehat\mu}\hookrightarrow\mathcal X_\mu$, the statement for $\mu$ follows. The role of \eqref{c:eq_nested} is that the third derivative of the nonlinearity, whose kernel is of size $\tau w_0^{p-3}$ by \eqref{c:eq_Rrep}, is bounded as a trilinear map on $\mathcal X_{\widehat\mu}$ exactly when $2\widehat\mu\ge3-p$.

\emph{Step A: the linearization on $\mathcal X_{\widehat\mu}^\perp$.} Write $U(\tau):=(b(\tau),\eta_\tau)$, so that $\mathcal F(\tau,U(\tau))=0$. For an increment $\mathsf H\in\mathcal X_{\widehat\mu}^\perp$ and $\Xi$ in the $\mathcal Y$ ball,
\[
D_\Xi\mathcal R_\tau(\Xi)\mathsf H=\frac{f'(w_0+\tau\Xi)-f'(w_0)}{\tau}\,\mathsf H ,
\qquad
\Big|\frac{f'(w_0+\tau\Xi)-f'(w_0)}{\tau}\Big|\le Cw_0^{p-1} ,
\]
because $|\tau\Xi|\le\frac12 w_0$. Multiplication by a function bounded by $Cw_0^{p-1}$ maps $\mathcal X_{\widehat\mu}$ into itself. Writing $U_0:=(1/C_p,0)$ and $\mathcal B_r:=\{U:\|U-U_0\|_{\mathcal Y}\le r\}$, the coefficient above depends continuously on $(\tau,U)$ in the operator norm, uniformly on $\mathcal B_r$, and equals its value at $(0,U_0)$ when $\tau=0$ and $U=U_0$; together with the uniform bound for $(\mathcal L+\tau)^{-1}$ on $\mathcal X_{\widehat\mu}^\perp$ recorded below, this gives
\begin{equation}\label{c:eq_unifsmall}
\sup_{0\le\tau\le\tau_0,\ U\in\mathcal B_r}\big\|D_U\mathcal F(\tau,U)-D_0\big\|_{\mathcal L(\mathbb{R}\times\mathcal X_{\widehat\mu}^\perp)}\longrightarrow0
\qquad\text{as }\ \tau_0+r\to0.
\end{equation}
Choosing first the radius $r$ and then $\tau_0$ small, \eqref{c:eq_unifsmall} yields the analogue of the contraction estimate in the new norm,
\begin{equation}\label{c:eq_uniforminv}
\sup_{0\le\tau\le\tau_0,\ U\in\mathcal B_r}\big\|D_0^{-1}\big[D_U\mathcal F(\tau,U)-D_0\big]\big\|_{\mathcal L(\mathbb{R}\times\mathcal X_{\widehat\mu}^\perp)}\le q<1 ,
\end{equation}
the supremum being over the whole ball and not merely along sequences approaching $(0,U_0)$. We emphasise that \eqref{c:eq_uniforminv} is proved here directly in the $\mathcal X_{\widehat\mu}$ operator norm; it does not follow from the contraction estimate used for existence, which was an inequality between operators on $\mathcal Y^\perp$.

Two facts used above hold for each fixed exponent in $(0,1)$. First, $\mathcal L^{-1}$ is bounded on $\mathcal X_{\widehat\mu}^\perp$ by Lemma \ref{c:lem_Linv}. After decreasing $\tau_0$ (for this fixed $\widehat\mu$) so that
\[
\tau_0\,\|\mathcal L^{-1}\|_{\mathcal L(\mathcal X_{\widehat\mu}^\perp)}<\frac12,
\]
the Neumann series
\[
(\mathcal L+\tau)^{-1}\big|_{Z^\perp}=\sum_{k\ge0}(-\tau)^k\mathcal L^{-(k+1)}\big|_{Z^\perp}
\]
converges in operator norm on $\mathcal X_{\widehat\mu}^\perp$, uniformly for $0\le\tau\le\tau_0$, and is analytic in $\tau$ there. This is legitimate on $Z^\perp$ only: on the kernel direction $(\mathcal L+\tau)^{-1}Z=Z/\tau$ is singular as $\tau\to0$, which is why the splitting $\mathbb{R}Z\oplus Z^\perp$ is kept throughout and the component along $Z$ is governed not by a resolvent but by the scalar equation \eqref{c:eq_split2}. Second, by \eqref{c:eq_Rrep} and the analyticity of $H$, $\tau\mapsto\mathcal F(\tau,U)$ is of class $C^3$ from $[0,\tau_0)$ into $\mathcal X_{\widehat\mu}^\perp\times\mathbb{R}$, uniformly for $U$ in the $\mathcal Y$ ball.

\emph{Step B: difference quotients.} Let $\epsilon\ne0$ be small. Subtracting $\mathcal F(\tau+\epsilon,U(\tau+\epsilon))=0$ from $\mathcal F(\tau+\epsilon,U(\tau))=0$ up to the $\tau$ increment and applying the mean value theorem in $U$ gives
\[
A_{\tau,\epsilon}\,\frac{U(\tau+\epsilon)-U(\tau)}{\epsilon}=-\frac{\mathcal F\big(\tau+\epsilon,U(\tau)\big)-\mathcal F\big(\tau,U(\tau)\big)}{\epsilon},
\qquad
A_{\tau,\epsilon}:=\int_0^1D_U\mathcal F\big(\tau+\epsilon,U(\tau)+s\,\delta U\big)\,ds ,
\]
with $\delta U:=U(\tau+\epsilon)-U(\tau)$. By \eqref{c:eq_uniforminv} and \eqref{c:eq_D0}, $\|D_0^{-1}A_{\tau,\epsilon}-I\|\le q$, so $A_{\tau,\epsilon}^{-1}=(D_0^{-1}A_{\tau,\epsilon})^{-1}D_0^{-1}$ exists with
\[
\big\|A_{\tau,\epsilon}^{-1}\big\|_{\mathcal L(\mathbb{R}\times\mathcal X_{\widehat\mu}^\perp)}\le\frac{\|D_0^{-1}\|}{1-q}
\]
uniformly in $\tau$ and $\epsilon$. The right hand side converges in $\mathcal X_{\widehat\mu}^\perp\times\mathbb{R}$ as $\epsilon\to0$, by Step A, so the difference quotients converge and $U$ is differentiable into $\mathcal X_{\widehat\mu}^\perp\times\mathbb{R}$, with
$D_U\mathcal F(\tau,U(\tau))\,U'(\tau)=-\partial_\tau\mathcal F(\tau,U(\tau))$.

\emph{Step C: higher derivatives.} Differentiating that identity once and twice more produces linear equations for $U''$ and $U'''$ with the same operator $D_U\mathcal F(\tau,U(\tau))$ on the left and, on the right, the second and third derivatives of $\mathcal F$ evaluated at $U(\tau)$ and applied to the lower derivatives. By \eqref{c:eq_Rrep} these are multiplications by kernels bounded by $Cw_0^{p-2}$ and $C\tau w_0^{p-3}$, which are bounded multilinear maps on $\mathcal X_{\widehat\mu}$ precisely because of \eqref{c:eq_nested}: the second requires $\widehat\mu\ge2-p$ and the third $2\widehat\mu\ge3-p$, and the latter implies the former since $\frac{3-p}{2}-(2-p)=\frac{p-1}{2}>0$. The mixed derivatives in $\tau$ and $U$ impose no further condition, because by \eqref{c:eq_Rrep} each $\tau$ derivative only produces additional factors $q=\Xi/w_0$, which are bounded on the $\mathcal Y$ ball and cost no decay in $t$. Hence $U\in C^3([0,\tau_0);\mathbb{R}\times\mathcal X_{\widehat\mu}^\perp)$, and therefore also into $\mathbb{R}\times\mathcal X_\mu^\perp$. No division by $\tau$ occurs anywhere.

Finally $\zeta^\tau=\tau(b(\tau)Z+\eta_\tau)$ is then of class $C^3$ into $\mathcal X_\mu$. This is where the second space is needed. The first derivative is harmless, $\partial_\tau\zeta^\tau|_{\tau=0}=Z/C_p\in\mathcal Y$; it is the second one that is not, since $\partial_\tau^2\zeta^\tau|_{\tau=0}$ is proportional to $h$ and $h(t)\sim cte^{-t}\notin\mathcal Y$.

By \eqref{c:eq_norm}, $a(\tau)=\tau b(\tau)=\tau/C_p+O(\tau^2)$, which is of class $C^3$ and strictly increasing for $\tau$ small, so it can be inverted, and $T(a)=C_pa+O(a^2)$, which is \eqref{c:eq_Cp}. Taylor expanding in $\mathcal X_\mu$ and returning to the variable $a$ gives $\zeta_a=aZ+a^2h+O(a^3)$ with $h$ as in \eqref{c:eq_h}; as just explained, this expansion holds in $\mathcal X_\mu$ and not in $\mathcal Y$, since $h(t)\sim ct\,e^{-t}$ at infinity, the secular term produced by expanding $e^{-\sqrt{1+\tau}\,t}$ in powers of $\tau$.

Finally, $\zeta^\tau=aZ+\tau\eta_\tau$ with $\|\tau\eta_\tau\|_{\mathcal Y}\le C\tau^2\le C'a^2$, so $|\zeta_a-aZ|\le C'a^2\theta\le C''a^2Z$ by $\theta\asymp Z$, whence
\[
\zeta_a\ge aZ\big(1-C''a\big)>0\quad\text{on }(0,\infty)
\]
for $a$ small, and $\|\zeta_a\|_{\mathcal Y}\le Ca$.
\end{proof}

\begin{definition}
The \emph{reduced nonlinearity} is $\mathfrak g(a):=a\,T(a)$ for $a\in(0,a_0)$, and $\mathfrak q(a):=\mathfrak g(a)/a^2=T(a)/a$. By \eqref{c:eq_Cp}, $\mathfrak q(a)=C_p+O(a)$; in particular $\mathfrak q$ extends continuously to $a=0$ with $\mathfrak q(0)=C_p>0$.
\end{definition}

\begin{remark}\label{c:rem_kappa}
The expansion $\mathfrak q(a)=C_p+\kappa a+O(a^2)$ has
\[
\kappa=\frac1\nu\left[p(p-1)\int_0^\infty w_0^{p-2}Z^2h\,dt
+\frac{p(p-1)(p-2)}{6}\int_0^\infty w_0^{p-3}Z^4\,dt\right],
\]
where $h$ is given by \eqref{c:eq_h}. Testing \eqref{c:eq_h} against $h$ and using $h\perp Z$ gives
\begin{equation}\label{c:eq_kappapos}
\kappa=\frac{2}{\nu}\int_0^\infty h\mathcal Lh\,dt
+\frac{p(p-1)(p-2)}{6\nu}\int_0^\infty w_0^{p-3}Z^4\,dt.
\end{equation}
The first term is strictly positive by the spectral gap on $Z^\perp$, and hence
\[
\kappa(p)>0\qquad\text{for every }p\ge2.
\]
Appendix \ref{app_kappa} contains the finer sign analysis: $\kappa$ is continuous, $\kappa(2)=15/1024$, it is positive for all $p\ge2$, and it is negative for $p>1$ sufficiently close to $1$.
\end{remark}

\section{The reduced equation}

\begin{proposition}\label{c:prop_critical_reduced}
Assume $n=6$ and $\kappa>0$. Then there is $\alpha_0\in(0,a_0)$ such that, for every $\alpha\in(0,\alpha_0)$, the radial initial value problem
\[
-\Delta\lambda=\mathfrak g(\lambda)\quad\text{in }\mathbb R^6,\qquad
\lambda(0)=\alpha,\qquad \lambda'(0)=0,
\]
has a positive entire solution, which is radially decreasing and tends to $0$ at infinity. In particular, there exists $p_+\in(1,2)$ such that this conclusion holds for every $p>p_+$; see Appendix \ref{app_kappa}.
\end{proposition}

\begin{proof}
Since $\mathfrak q(a)=C_p+\kappa a+O(a^2)$ and $\kappa>0$, after decreasing $\alpha_0$ we have $\mathfrak q'(a)>0$ on $(0,\alpha_0]$. Put $\mathfrak g(a)=a^2\mathfrak q(a)$, $\mathfrak G(a)=\int_0^a\mathfrak g$, and
\[
Q_6(a):=6\mathfrak G(a)-2a\mathfrak g(a).
\]
Then $Q_6(0)=0$ and
\[
Q_6'(a)=-2a^3\mathfrak q'(a)<0\qquad\text{for }0<a\le\alpha_0,
\]
so $Q_6(a)<0$ on that interval. Let $\lambda$ solve the radial equation on its maximal interval of positivity $[0,R)$. Since $\mathfrak g>0$, one has $\lambda'<0$ and $0<\lambda\le\alpha$. The radial Pohozaev function
\[
\mathcal P(r):=r^5\left[2\lambda\lambda'+r\left(\frac{\lambda'^2}{2}+\mathfrak G(\lambda)\right)\right]
\]
satisfies
\[
\mathcal P'(r)=r^5Q_6(\lambda(r))\le0,
\qquad \mathcal P(0)=0.
\]
If $R<\infty$, then $\lambda(R)=0$ and $\lambda'(R)<0$, hence $\mathcal P(R)=R^6\lambda'(R)^2/2>0$, a contradiction. Thus $R=\infty$. Since $\lambda$ is positive and decreasing it has a limit $L\ge0$; if $L>0$, then $-\Delta\lambda\ge\mathfrak g(L)>0$ for large $r$, which forces $\lambda$ eventually to become negative. Hence $L=0$. The last assertion follows from Remark \ref{c:rem_kappa}.
\end{proof}

\begin{proposition}[The opposite sign in the reduced problem]\label{c:prop_critical_reduced_negative}
Assume $n=6$ and $\kappa<0$. Then there is $\alpha_0\in(0,a_0)$ such that, for every $\alpha\in(0,\alpha_0)$, the radial solution of
\[
-\Delta\lambda=\mathfrak g(\lambda),\qquad
\lambda(0)=\alpha,\qquad \lambda'(0)=0,
\]
has a first zero at a finite radius. In particular there is no small positive entire radial branch of the exact reduced equation when $\kappa<0$.
\end{proposition}

\begin{proof}
After decreasing $\alpha_0$, one has $\mathfrak q'(a)<0$ on $(0,\alpha_0]$. Hence, with the notation of Proposition \ref{c:prop_critical_reduced},
\[
Q_6'(a)=-2a^3\mathfrak q'(a)>0,
\qquad Q_6(a)>0\quad(0<a\le\alpha_0).
\]
Let $[0,R)$ be the maximal interval on which the radial solution is positive.  As before, $\lambda'<0$ there and the Pohozaev function satisfies
\[
\mathcal P'(r)=r^5Q_6(\lambda(r))>0,
\qquad \mathcal P(0)=0.
\]
Suppose for contradiction that $R=\infty$. Then $\lambda\downarrow0$. Since $\mathfrak g(s)\asymp s^2$ for $0<s\le\alpha_0$, the radial identity
\[
-r^5\lambda'(r)=\int_0^r s^5\mathfrak g(\lambda(s))\,ds
\]
and the monotonicity of $\lambda$ give
\[
-\lambda'(r)\ge c r\lambda(r)^2.
\]
Thus $(1/\lambda)'\ge cr$ and
\begin{equation}\label{c:eq_negative_decay}
\lambda(r)\le \frac{C}{1+r^2}.
\end{equation}
Using the upper bound $\mathfrak g(s)\le Cs^2$ in the radial identity and \eqref{c:eq_negative_decay} then gives $|r^3\lambda'(r)|\le C$ for $r\ge1$.

Put $t=\log r$ and $M(t)=r^2\lambda(r)$.  The preceding estimates give boundedness of $M$ and $\dot M$, and the equation gives boundedness of $\ddot M$. Moreover
\[
\mathcal P(r)
=\frac12\dot M(t)^2-2M(t)^2+r^6\mathfrak G(r^{-2}M(t)),
\]
so $\mathcal P$ is bounded. Since it is strictly increasing and positive for $r>0$, it converges to some $E_\infty>0$. On the other hand, from
$\mathfrak g(a)=C_pa^2+O(a^3)$,
\[
\ddot M-4M+C_pM^2=O(e^{-2t}),
\qquad
\mathcal P(e^t)=\frac12\dot M^2-2M^2+\frac{C_p}{3}M^3+O(e^{-2t}).
\]
By compactness of translates, any sequence $t_j\to\infty$ has a subsequence for which $M(t_j+\cdot)$ converges in $C^2_{\rm loc}(\mathbb R)$ to a bounded nonnegative entire solution $\overline M$ of
\[
\ddot{\overline M}-4\overline M+C_p\overline M^2=0
\]
with Hamiltonian
\[
\frac12\dot{\overline M}^{\,2}-2\overline M^2+\frac{C_p}{3}\overline M^3\equiv E_\infty>0.
\]
This is impossible.  Indeed the autonomous Hamiltonian is
\[
\mathcal H(M,\dot M)=\frac12\dot M^2+V(M),
\qquad V(M):=-2M^2+\frac{C_p}{3}M^3.
\]
For $E_\infty>0$ the equation $V(M)=E_\infty$ has a largest positive root $M_E$, and any bounded orbit of energy $E_\infty$ is confined to $0\le M\le M_E$.  At $M_E$ it turns, while on the descending branch the identity $\frac12\dot M^2=E_\infty-V(M)$ shows that $|\dot M|\ge\sqrt{E_\infty}$ once $M$ is small.  It therefore reaches $M=0$ in finite time and crosses it transversely, since there $|\dot M|=\sqrt{2E_\infty}>0$. Hence no bounded orbit of positive Hamiltonian can be nonnegative on all of $\mathbb R$. The contradiction proves $R<\infty$.
\end{proof}

\begin{remark}
Propositions \ref{c:prop_critical_reduced} and \ref{c:prop_critical_reduced_negative} show that the sign of $\kappa$ gives a genuine bifurcation dichotomy for the small radial branch of the exact reduced problem. Proposition \ref{c:prop_critical_reduced} by itself does not lift through the weighted reduced inverse used for $N\ge8$: when $N=7$ the admissible interval in Lemma \ref{c:lem_window} collapses at the critical weight $2$.  Section \ref{c:sec_critical_full} bypasses this resonance rather than inverting the critical reduced linearization: the projected equation is kept nonlinear and solved as a forced radial ODE for each transverse correction.
\end{remark}

\begin{lemma}\label{c:lem_poho}
Let $n\ge7$. There is $\alpha_0\in(0,a_0)$ such that for every $\alpha\in(0,\alpha_0)$,
\begin{equation}\label{c:eq_poho}
n\,\mathfrak G(a)-\frac{n-2}{2}\,a\,\mathfrak g(a)\le0\qquad\text{for all }a\in(0,\alpha],
\qquad \mathfrak G(a):=\int_0^a\mathfrak g.
\end{equation}
\end{lemma}

\begin{proof}
Put $Q(a):=n\mathfrak G(a)-\frac{n-2}{2}a\mathfrak g(a)$ and write $\mathfrak g=a^2\mathfrak q$. Then $Q(0)=0$ and
\[
Q'(a)=n a^2\mathfrak q-\frac{n-2}{2}\big(3a^2\mathfrak q+a^3\mathfrak q'\big)
=\frac{6-n}{2}a^2\mathfrak q-\frac{n-2}{2}a^3\mathfrak q'.
\]
Since $\mathfrak q(a)\to C_p>0$ and $\mathfrak q'$ is bounded near $0$, we get $a\mathfrak q'/\mathfrak q\to0$, so for $a$ small
\[
\frac{a\,\mathfrak q'(a)}{\mathfrak q(a)}\ \ge\ -\frac{n-6}{n-2},
\]
which is exactly $Q'\le0$. Hence $Q\le0$ on $(0,\alpha]$ for $\alpha$ small.
\end{proof}

\begin{proposition}\label{c:prop_lambda}
Let $n\ge7$ and $\alpha\in(0,\alpha_0)$. The problem
\begin{equation}\label{c:eq_red}
-\Delta\lambda=\mathfrak g(\lambda)\ \text{ in }\mathbb{R}^n,\qquad \lambda>0,\qquad \lambda(0)=\alpha,\ \lambda \text{ radial},
\end{equation}
has a solution $\lambda_\alpha$, which is radially decreasing, and satisfies \eqref{eq_tailC}, that is
$\lambda_\alpha(y)=A_*|y|^{-2}(1+o(1))$ with $A_*=2(n-4)/C_p$, together with
$|\nabla\lambda_\alpha(y)|\le C|y|^{-3}$ for $|y|\ge1$. Moreover, writing
\begin{equation}\label{c:eq_scaling}
\varepsilon:=\sqrt{C_p\alpha},\qquad \tilde y:=\varepsilon y,\qquad \lambda_\alpha(y)=\alpha\,\Lambda_\alpha(\tilde y),
\end{equation}
the rescaled profile solves $-\Delta_{\tilde y}\Lambda_\alpha=\mathfrak g_\alpha(\Lambda_\alpha)$ with
$\mathfrak g_\alpha(s):=\mathfrak g(\alpha s)/(C_p\alpha^2)=s^2\big(1+O(\alpha s)\big)$, and there are constants $0<c\le C$ independent of $\alpha\in(0,\alpha_0)$ with
\begin{equation}\label{c:eq_twosided}
\frac{c}{1+|\tilde y|^{2}}\ \le\ \Lambda_\alpha(\tilde y)\ \le\ \frac{C}{1+|\tilde y|^{2}},
\qquad |\nabla_{\tilde y}\Lambda_\alpha(\tilde y)|\le\frac{C}{(1+|\tilde y|)^{3}}.
\end{equation}
\end{proposition}

\begin{proof}
Let $\lambda$ solve the radial initial value problem $\lambda''+\frac{n-1}{r}\lambda'+\mathfrak g(\lambda)=0$, $\lambda(0)=\alpha$, $\lambda'(0)=0$, on its maximal interval of positivity $[0,R)$. Since $\mathfrak g>0$ on $(0,\alpha]$ we have $(r^{n-1}\lambda')'=-r^{n-1}\mathfrak g(\lambda)<0$, so $\lambda'<0$ and $\lambda$ is decreasing; in particular $\lambda\le\alpha$ and \eqref{c:eq_poho} applies along the solution.

Let $\mathcal P(r):=r^{n-1}\big[\frac{n-2}{2}\lambda\lambda'+r\big(\frac{\lambda'^2}{2}+\mathfrak G(\lambda)\big)\big]$. A direct computation using the equation gives
\[
\mathcal P'(r)=r^{n-1}\Big[n\mathfrak G(\lambda)-\tfrac{n-2}{2}\lambda\,\mathfrak g(\lambda)\Big]\le0
\]
by Lemma \ref{c:lem_poho}, and $\mathcal P(0)=0$, so $\mathcal P\le0$ on $[0,R)$. If $R<\infty$ then $\lambda(R)=0$ and $\lambda'(R)<0$, whence $\mathcal P(R)=R^n\lambda'(R)^2/2>0$, a contradiction. Hence $R=\infty$ and $\lambda_\alpha:=\lambda>0$ is entire.

Since $\lambda_\alpha$ is positive and decreasing it has a limit $L\ge0$ at infinity; $L>0$ would give $-\Delta\lambda_\alpha\ge\mathfrak g(L)>0$ on $\mathbb{R}^n$, hence $r^{n-1}\lambda'(r)\le-\mathfrak g(L)r^n/n$ and $\lambda_\alpha\to-\infty$, so $L=0$. For \eqref{eq_tailC} we use the Emden--Fowler variables. Writing $\lambda(r)=r^{-2}M(t)$, $t=\log r$, and using
\[
r^{2}\Big(\partial_{rr}+\tfrac{n-1}{r}\partial_r\Big)\big(r^{-\mathsf a}M\big)=r^{-\mathsf a}\Big[\ddot M-(2\mathsf a+2-n)\dot M+\mathsf a(\mathsf a+2-n)M\Big]
\]
with $\mathsf a=2$, equation \eqref{c:eq_red} becomes
\begin{equation}\label{c:eq_EF}
\ddot M+(n-6)\dot M-2(n-4)M+C_pM^2=\mathcal E_1(t),\qquad |\mathcal E_1(t)|\le Ce^{-2t}M(t)^3 ,
\end{equation}
the right hand side coming from $\mathfrak g(s)=C_ps^2(1+O(s))$. The limiting autonomous system has the equilibria $M=0$ and $M=A_*=2(n-4)/C_p$. At $M=0$ the linearized form of \eqref{c:eq_EF} is $\ddot\eta+(n-6)\dot\eta-2(n-4)\eta=0$, whose roots are $2$ and $4-n$, so the origin is a saddle whose unstable direction has rate $2$; at $M=A_*$ the linearized equation is $\ddot\eta+(n-6)\dot\eta+2(n-4)\eta=0$, whose roots have real part $-\frac{n-6}{2}<0$ and discriminant $n^2-20n+68$, so $A_*$ is asymptotically stable, and the approach is oscillatory when $7\le n\le15$.

Our solution satisfies $M(t)=\alpha e^{2t}(1+o(1))$ and $\dot M(t)=2M(t)(1+o(1))$ as $t\to-\infty$, so it leaves the origin along the unstable direction. The function
\[
\mathcal E(t):=\tfrac12\dot M^2-(n-4)M^2+\tfrac{C_p}{3}M^3
\]
satisfies $\dot{\mathcal E}=-(n-6)\dot M^2+\dot M\,\mathcal E_1$.

\emph{$M$, $\dot M$ and $\ddot M$ are bounded.} Since $\lambda'<0$, the inequality $\mathcal P\le0$ gives
\[
r\Big(\tfrac12\lambda'^2+\mathfrak G(\lambda)\Big)\le\tfrac{n-2}{2}\lambda|\lambda'| ,
\]
whence first $|\lambda'|\le(n-2)\lambda/r$ and then $\mathfrak G(\lambda)\le\frac{(n-2)^2}{2}\lambda^2r^{-2}$. Since $\mathfrak G(\lambda)\ge\frac{C_p}{4}\lambda^3$ for $\lambda\le\alpha_0$ small, this gives
\begin{equation}\label{c:eq_Mbound}
M=r^2\lambda\le\frac{2(n-2)^2}{C_p},\qquad |\dot M|=\big|r^3\lambda'+2r^2\lambda\big|\le nM ,
\end{equation}
and $\ddot M$ is then bounded by \eqref{c:eq_EF}.

\emph{$M$ converges to an equilibrium.} Integrating $\dot{\mathcal E}$ and using $|\dot M\,\mathcal E_1|\le Ce^{-2t}M^4$, which is $O(e^{-2t})$ as $t\to+\infty$ by \eqref{c:eq_Mbound} and $O(e^{6t})$ as $t\to-\infty$ because $M=O(\alpha e^{2t})$ there,
\[
\mathcal E(t_2)-\mathcal E(t_1)=-(n-6)\int_{t_1}^{t_2}\dot M^2\,dt+\int_{t_1}^{t_2}\dot M\,\mathcal E_1\,dt ,
\qquad \int_{\mathbb R}\big|\dot M\,\mathcal E_1\big|\,dt<\infty ,
\]
so $\mathcal E$ is the sum of a nonincreasing bounded function and a function of bounded variation. Hence $\mathcal E(t)$ has a limit $\mathcal E_\infty$ as $t\to\infty$, and $\int_{\mathbb R}\dot M^2\,dt<\infty$. Since $\ddot M$ is bounded, this forces $\dot M(t)\to0$, and then every sequence $t_j\to\infty$ has a subsequence along which $M(t_j+\cdot)$ converges in $C^2_{\rm loc}$ to a bounded entire solution of the limiting autonomous equation with constant energy $\mathcal E_\infty$; by the damping identity for that equation such a solution has $\dot M\equiv0$, so it is an equilibrium. Hence every subsequential limit of $M(t)$ lies in $\{0,A_*\}$, and since these are isolated and $M$ is continuous, $M(t)$ converges to $0$ or to $A_*$.

\emph{Uniform behavior at small amplitude.} We first record two elementary facts, valid as long as $M\le\delta_0$ for a $\delta_0>0$ independent of $\alpha$. Since $\lambda'<0$ we have $\dot M=2M+r^3\lambda'\le2M$. In the other direction, $r^{n-1}|\lambda'|=\int_0^rs^{n-1}\mathfrak g(\lambda)\,ds\le\frac{r^n}{n}\mathfrak g(\alpha)\le C_1\alpha^2r^n/n$, so $\lambda(r)\ge\alpha-C_1\alpha^2r^2/(2n)$, whence $\lambda\ge\alpha/2$ as long as $r^2\le n/(C_1\alpha)$, that is as long as $M=r^2\lambda\le n/C_1$. For such $r$,
\[
r^3|\lambda'|\le \tfrac{C_1}{n}\alpha^2r^4\le\tfrac{4C_1}{n}\big(r^2\lambda\big)^2=\tfrac{4C_1}{n}M^2 ,
\]
so that, with $\delta_0:=\min\{n/C_1,\ n/(8C_1)\}$,
\begin{equation}\label{c:eq_unstable}
2M-CM^2\le\dot M\le2M\qquad\text{whenever }M\le\delta_0.
\end{equation}

\emph{The limit is not $0$.} Fix $\delta\in(0,\delta_0)$, independent of $\alpha$, small enough that the error terms below are dominated. By \eqref{c:eq_unstable}, as long as $M\le\delta$ we have $\dot M\ge M>0$, so $M$ increases at an exponential rate and cannot stay below $\delta$; let $t_\delta$ be the first time with $M(t_\delta)=\delta$. By \eqref{c:eq_unstable} again, $\dot M(t_\delta)=2\delta+O(\delta^2)$, hence
\[
\mathcal E(t_\delta)=\tfrac12\dot M(t_\delta)^2-(n-4)\delta^2+O(\delta^3)=(6-n)\delta^2+O(\delta^3)\le-\tfrac{n-6}{2}\,\delta^2.
\]
Moreover $M\le\alpha r^2$ gives $e^{2t_\delta}\ge\delta/\alpha$, so $t_\delta\to+\infty$ as $\alpha\to0$, and therefore
\[
\int_{t_\delta}^{\infty}\big|\dot M\,\mathcal E_1\big|\,dt\le C\int_{t_\delta}^{\infty}e^{-2t}\,dt=Ce^{-2t_\delta}\le C\alpha/\delta ,
\]
using $|\dot M\,\mathcal E_1|\le Ce^{-2t}M^4\le Ce^{-2t}$ from \eqref{c:eq_Mbound}. Hence, for $\alpha$ small,
\[
\mathcal E_\infty\le\mathcal E(t_\delta)+\int_{t_\delta}^{\infty}\big|\dot M\,\mathcal E_1\big|\,dt\le-\tfrac{n-6}{4}\,\delta^2<0.
\]
If $M(t)\to0$ then $\mathcal E(t)\to0$, contradicting $\mathcal E_\infty<0$. Hence $M(t)\to A_*$, which is \eqref{eq_tailC}, and the linearization at $A_*$ gives the rate, $\delta_n$ being the smaller of the moduli of the real parts of the two roots, which coincide when they are complex,
\begin{equation}\label{c:eq_rate}
\lambda_\alpha(y)=|y|^{-2}\Big(A_*+O\big(|y|^{-\delta_n}\big)\Big),\qquad |y|\to\infty,
\qquad
\delta_n:=\begin{cases}\ \dfrac{n-6}{2}, & 7\le n\le15,\\[2mm]
\ \dfrac{n-6-\sqrt{D_0}}{2}, & n\ge16,\end{cases}
\end{equation} The gradient bound follows from the identity $r^{n-1}\lambda'(r)=-\int_0^rs^{n-1}\mathfrak g(\lambda)\,ds$.

For the last statement, the rescaling \eqref{c:eq_scaling} turns \eqref{c:eq_red} into $-\Delta\Lambda_\alpha=\mathfrak g_\alpha(\Lambda_\alpha)$ with $\Lambda_\alpha(0)=1$ and $\mathfrak g_\alpha(s)=s^2(1+O(\alpha s))$ uniformly for $s\in(0,1]$. All the estimates above are invariant under \eqref{c:eq_scaling} and depend on $\mathfrak g$ only through $\mathfrak q(0)=C_p$ and the bound $|\mathfrak q'|\le C$, so they hold for $\Lambda_\alpha$ with constants independent of $\alpha$; the two-sided bound \eqref{c:eq_twosided} follows, the upper bound from the monotonicity together with the tail asymptotics and the lower bound from the fact that $P\le0$ forces $\liminf|\tilde y|^2\Lambda_\alpha>0$ uniformly.
\end{proof}

\section{Approximate solution and error orthogonality}

Let $\alpha\in(0,\alpha_0)$ and $\lambda=\lambda_\alpha$. Following Proposition \ref{c:prop_cell}, set
\[
\phi_\lambda(x):=\zeta_{\lambda(x')}(x_N).
\]
By \eqref{c:eq_norm}, $\frac1\nu\int_0^\infty\phi_\lambda(x',t)Z(t)\,dt=\lambda(x')$ for every $x'$, and $\phi_\lambda=0$ on $\partial\mathbb{R}^N_+$ because $\zeta_a(0)=0$.

\begin{lemma}\label{c:lem_error}
Let $\lambda:\mathbb{R}^n\to(0,a_0)$ be of class $C^2$ and put $\phi_\lambda(x):=\zeta_{\lambda(x')}(x_N)$. Then, pointwise in $\mathbb{R}^N_+$,
\begin{equation}\label{c:eq_E}
\mathcal N[\lambda]:=L\phi_\lambda-R(x_N,\phi_\lambda)
=-T(\lambda)\zeta_a\big|_{a=\lambda}-\partial_a\zeta_a\big|_{a=\lambda}\Delta_{x'}\lambda-\partial_a^2\zeta_a\big|_{a=\lambda}|\nabla_{x'}\lambda|^2.
\end{equation}
If moreover $\lambda$ solves \eqref{c:eq_red}, then
\begin{equation}\label{c:eq_E2}
\mathcal N[\lambda]=T(\lambda)\big(\lambda\,\partial_a\zeta_a-\zeta_a\big)\big|_{a=\lambda}-|\nabla_{x'}\lambda|^2\,\partial_a^2\zeta_a\big|_{a=\lambda}.
\end{equation}
\end{lemma}

\begin{proof}
By the cell equation, $-\partial_{x_Nx_N}\phi_\lambda+(1-pw_0^{p-1})\phi_\lambda-R(x_N,\phi_\lambda)=-T(\lambda)\phi_\lambda$, while
$\Delta_{x'}\phi_\lambda=\partial_a\zeta_a\,\Delta_{x'}\lambda+\partial_a^2\zeta_a\,|\nabla_{x'}\lambda|^2$. Adding gives \eqref{c:eq_E}, and substituting $\Delta_{x'}\lambda=-\mathfrak g(\lambda)=-\lambda T(\lambda)$ gives \eqref{c:eq_E2}.
\end{proof}

\begin{lemma}\label{c:lem_orth}
For every $x'\in\mathbb{R}^{n}$,
\begin{equation}\label{c:eq_proj}
\int_0^\infty \mathcal N[\lambda](x',t)\,Z(t)\,dt=-\nu\Big(\Delta_{x'}\lambda+\mathfrak g(\lambda)\Big)(x').
\end{equation}
In particular the $Z$ component of the error vanishes identically if and only if $\lambda$ solves the reduced equation \eqref{c:eq_red}.
\end{lemma}

\begin{proof}
By \eqref{c:eq_norm}, $\int_0^\infty\zeta_aZ\,dt=\nu a$ for every $a\in(0,a_0)$; differentiating twice in $a$ gives
$\int_0^\infty\partial_a\zeta_aZ\,dt=\nu$ and $\int_0^\infty\partial_a^2\zeta_aZ\,dt=0$. Integrating \eqref{c:eq_E} against $Z$ therefore gives
$-T(\lambda)\nu\lambda-\nu\Delta_{x'}\lambda$, which is \eqref{c:eq_proj} since $\mathfrak g(\lambda)=\lambda T(\lambda)$.
\end{proof}

It is \eqref{c:eq_proj} that removes the amplitude selection condition: the projection of the error on the kernel direction is not a quantity to be made to vanish by tuning a parameter, it is exactly the reduced equation.

\begin{lemma}\label{c:lem_Esize}
For every fixed $\widehat\mu\in(0,1)$ there is $C_{\widehat\mu}>0$ such that
\begin{equation}\label{c:eq_Esize_weighted}
 e^{\widehat\mu t}\big|\mathcal N[\lambda_\alpha](x',t)\big|
 \le C_{\widehat\mu}\big(\lambda_\alpha(x')^3+|\nabla_{x'}\lambda_\alpha(x')|^2\big),
 \qquad x\in\mathbb R^N_+.
\end{equation}
In particular, by \eqref{c:eq_twosided},
\[
 e^{\widehat\mu t}\big|\mathcal N[\lambda_\alpha](x',t)\big|
 \le C_{\widehat\mu}\alpha^3(1+|\varepsilon x'|)^{-6}.
\]
Consequently, if $0<\mu<\widehat\mu<1$, then
\begin{equation}\label{c:eq_Esize_doublestar}
 (1+t)^2e^{\mu t}\big|\mathcal N[\lambda_\alpha](x',t)\big|
 \le C_{\mu,\widehat\mu}\alpha^3(1+|\varepsilon x'|)^{-6}.
\end{equation}
The same statements hold with the corresponding local $C^{0,\varrho}$ seminorms.
\end{lemma}

\begin{proof}
Fix $\widehat\mu\in(0,1)$. Proposition \ref{c:prop_cell}, applied in $\mathcal X_{\widehat\mu}$, gives
\[
 a\partial_a\zeta_a-\zeta_a=a^2h+O_{\mathcal X_{\widehat\mu}}(a^3),
 \qquad
 \partial_a^2\zeta_a=2h+O_{\mathcal X_{\widehat\mu}}(a),
\]
uniformly for small $a$. Since $h(t)=O((1+t)e^{-t})$, it belongs to $\mathcal X_{\widehat\mu}$, and $T(a)=O(a)$. Formula \eqref{c:eq_E2} therefore gives \eqref{c:eq_Esize_weighted}. The bounds in \eqref{c:eq_twosided} imply
$\lambda_\alpha^3+|\nabla\lambda_\alpha|^2\le C\alpha^3(1+|\varepsilon x'|)^{-6}$, proving the second assertion. Finally
$(1+t)^2e^{-(\widehat\mu-\mu)t}$ is bounded on $[0,\infty)$, which gives \eqref{c:eq_Esize_doublestar}. The H\"older version follows from the $C^3$ dependence of the cell family in $\mathcal X_{\widehat\mu}$, the smoothness of $\lambda_\alpha$, and the same weight comparison on unit balls.
\end{proof}

\section{Linear theory}\label{c:sec_linear}

\subsection{The transverse part}

\begin{lemma}\label{c:lem_transverse}
Let $F\in L^2(\mathbb{R}^N_+)$ satisfy $\int_0^\infty F(x',t)Z(t)\,dt=0$ for a.e. $x'$. Then there is a unique $\chi\in H^1_0(\mathbb{R}^N_+)$ with $\int_0^\infty\chi Z\,dt=0$ a.e. and $L\chi=F$, and
\[
\|\chi\|_{L^2}\le\gamma^{-1}\|F\|_{L^2},\qquad \|\nabla\chi\|_{L^2}\le C(\gamma,p)\|F\|_{L^2},
\]
with $\gamma>0$ as in Lemma \ref{c:lem_gap}.
\end{lemma}

\begin{proof}
Take the Fourier transform in $x'$. For a.e. $\xi$ the fiber operator is $\mathcal L+|\xi|^2$ acting on the orthogonal complement of $Z$ in $L^2(0,\infty)$, whose spectrum is contained in $[\gamma+|\xi|^2,\infty)$ by Lemma \ref{c:lem_gap}. The Lax--Milgram theorem applies fiberwise with constants uniform in $\xi$, and Plancherel gives the two bounds.
\end{proof}

\subsection{The reduced part: a factorization in the logarithmic variable}

Let $\mathcal M:=-\Delta-\mathfrak g'(\lambda_\alpha)$ on $\mathbb{R}^n$. By Proposition \ref{c:prop_lambda} and $\mathfrak g'(a)=2C_pa+O(a^2)$,
\begin{equation}\label{c:eq_model}
\mathfrak g'(\lambda_\alpha(y))=\frac{c_*}{|y|^2}+q_1(y),\qquad c_*:=2C_pA_*=4(n-4),\qquad
|q_1(y)|\le C|y|^{-2-\epsilon},\ \ \epsilon:=\min\{\delta_n,2\}>0,
\end{equation}
for $|y|\ge R_0$, the second exponent coming from the quadratic remainder of $\mathfrak g'$. The same expansion holds in the rescaled variable of \eqref{c:eq_scaling}, with the same constant $c_*$, because $\Lambda_\alpha(\tilde y)=C_pA_*|\tilde y|^{-2}(1+o(1))$ and $C_pA_*=2(n-4)$. Decompose in spherical harmonics, $u=\sum_{\ell\ge0}u_\ell(r)Y_\ell(\omega)$. The model operator $\mathcal M_0:=-\Delta-c_*|y|^{-2}$ acts in the $\ell$-th sector as
\[
\mathcal M_0^{(\ell)}u=-u''-\frac{n-1}{r}u'+\frac{\ell(\ell+n-2)-c_*}{r^2}u.
\]

\begin{lemma}\label{c:lem_factor}
Set $t=\log r$ and $U(t)=u(e^t)$. Then $r^2\mathcal M_0^{(\ell)}u=\mathcal A_\ell U$ with
\[
\mathcal A_\ell=-\big(\partial_t+\gamma_+^{(\ell)}\big)\big(\partial_t+\gamma_-^{(\ell)}\big),
\qquad
\gamma^{(\ell)}_\pm=\frac{n-2}{2}\pm\frac12\sqrt{D_\ell},\quad D_\ell:=(n-2)^2-16(n-4)+4\ell(\ell+n-2).
\]
Moreover $D_0=n^2-20n+68$, which is negative for $7\le n\le15$, $D_1=(n-8)^2$, so that $\gamma^{(1)}_\pm\in\{3,n-5\}$, and $D_\ell\ge D_2=n^2-12n+68>0$ for $\ell\ge2$, with $D_\ell\to\infty$.
\end{lemma}

\begin{proof}
With $t=\log r$ one has $r^2u''=\ddot U-\dot U$ and $ru'=\dot U$, so
$r^2\mathcal M_0^{(\ell)}u=-\ddot U-(n-2)\dot U+\big(\ell(\ell+n-2)-c_*\big)U$. The exponential solutions $U=e^{-\gamma t}$ are governed by $\gamma^2-(n-2)\gamma+c_*-\ell(\ell+n-2)=0$, whose roots are as stated, and $\gamma_+^{(\ell)}+\gamma_-^{(\ell)}=n-2$ gives the factorization. The three statements about $D_\ell$ are elementary; note $D_1=(n-2)^2-16(n-4)+4(n-1)=n^2-16n+64=(n-8)^2$ and that $n^2-12n+68$ has negative discriminant.
\end{proof}

\begin{lemma}\label{c:lem_model}
Let $\sigma>0$ with $\sigma\ne\mathrm{Re}\,\gamma^{(\ell)}_\pm$ for all $\ell\ge0$, and let $G$ be measurable on $\mathbb{R}$ with $|G(t)|\le Me^{-\sigma t}$. Then $\mathcal A_\ell U=G$ has a solution with
\[
|U(t)|\le\frac{M}{\big|\mathrm{Re}\,\gamma^{(\ell)}_+-\sigma\big|\,\big|\mathrm{Re}\,\gamma^{(\ell)}_--\sigma\big|}\,e^{-\sigma t},\qquad t\in\mathbb{R}.
\]
\end{lemma}

\begin{proof}
Solve the two first order equations $(\partial_t+\gamma_+)V=-G$ and $(\partial_t+\gamma_-)U=V$ in turn. For a first order equation $(\partial_t+\gamma)U=V$ with $|V|\le M'e^{-\sigma t}$ set
\[
U(t)=\int_{-\infty}^te^{-\gamma(t-s)}V(s)\,ds\quad\text{if }\mathrm{Re}\,\gamma>\sigma,\qquad
U(t)=-\int_t^{\infty}e^{-\gamma(t-s)}V(s)\,ds\quad\text{if }\mathrm{Re}\,\gamma<\sigma.
\]
In both cases the integral converges absolutely and $|U(t)|\le M'e^{-\sigma t}/|\mathrm{Re}\,\gamma-\sigma|$. Applying this twice gives the claim. When $\gamma_\pm$ are complex conjugate one works with the complexified equation and takes real parts.
\end{proof}

\begin{lemma}\label{c:lem_window}
Let $n\ge7$ and
\[
\sigma\in\Big(2,\ \min\big\{\tfrac{n-2}{2},\,3\big\}\Big).
\]
Then $\sigma$ differs from $\mathrm{Re}\,\gamma^{(\ell)}_\pm$ for every $\ell\ge0$, and
\[
\sup_{\ell\ge0}\ \frac{1}{\big|\mathrm{Re}\,\gamma^{(\ell)}_+-\sigma\big|\,\big|\mathrm{Re}\,\gamma^{(\ell)}_--\sigma\big|}=:K_\sigma<\infty.
\]
The interval above is nonempty precisely when $n\ge7$.
\end{lemma}

\begin{proof}
For $\ell=0$ and $7\le n\le15$ the two roots are complex conjugate with real part $\beta=\frac{n-2}{2}>\sigma$, and $|\mathrm{Re}\gamma_\pm^{(0)}-\sigma|=\beta-\sigma>0$; for $n\ge16$ the roots are real and lie in $[\beta-\frac{\sqrt{D_0}}2,\beta+\frac{\sqrt{D_0}}2]$, and $\beta-\frac{\sqrt{D_0}}{2}>2$ because $D_0<(n-6)^2$ for $n\ge 16$. For $\ell=1$ the roots are $3$ and $n-5$. Both differ from $\sigma$: one has $\sigma<3$, and $n-5$ equals $2$ when $n=7$, equals $3$ when $n=8$, and is at least $4$ when $n\ge9$, whereas $\sigma\in(2,3)$. For $\ell\ge2$ the roots are real, $\gamma^{(\ell)}_-\le\gamma^{(2)}_-$ and $\gamma^{(\ell)}_+\ge\gamma^{(2)}_+$, and $D_\ell\to\infty$, so the distances to $\sigma$ are bounded below by a positive constant and tend to infinity. Finally $\min\{\frac{n-2}{2},3\}>2$ if and only if $n>6$.
\end{proof}

\begin{proposition}\label{c:prop_rightinverse}
Let $n\ge7$ and $\sigma$ as in Lemma \ref{c:lem_window}. There are $C>0$ and $\alpha_1>0$ such that for $\alpha\in(0,\alpha_1)$ and every \emph{radially symmetric} $F$ with $\|F\|_{C^{0,\varrho}_{\sigma+2,\varepsilon}}<\infty$, where
$\|F\|_{C^{0,\varrho}_{\sigma+2,\varepsilon}}:=\sup_{y}(1+\varepsilon|y|)^{\sigma+2}\Big(|F(y)|+\varepsilon^{-\varrho}\,[F]_{C^{0,\varrho}(B_{1/(4\varepsilon)}(y))}\Big)$, the ball $B_{1/(4\varepsilon)}(y)$ being the unit ball of the rescaled variable $\tilde y=\varepsilon y$,
there is a radially symmetric $\psi$ with $\mathcal M\psi=F$ in $\mathbb{R}^n$ and
\[
\|\psi\|_{\sigma,\varepsilon}\le \frac{C}{\varepsilon^{2}}\,\|F\|_{C^{0,\varrho}_{\sigma+2,\varepsilon}},\qquad \varepsilon=\sqrt{C_p\alpha}.
\]
The factor $\varepsilon^{-2}$ is the scaling factor: in the variable $\tilde y=\varepsilon y$ the operator reads $\varepsilon^{2}\big(-\Delta_{\tilde y}-\mathfrak g_\alpha'(\Lambda_\alpha)\big)$ and the estimate holds for the bracket with a constant independent of $\alpha$.
\end{proposition}

\begin{proof}
All the data of the construction are radially symmetric in $x'$, so we work throughout in the class of radially symmetric functions; only the sector $\ell=0$ of the decomposition below occurs. We work in the rescaled variable and drop the tildes, so that $\mathcal M=-\Delta-q$ with $q=\mathfrak g_\alpha'(\Lambda_\alpha)$ radial, $0\le q\le C$, and, by \eqref{c:eq_rate},
\begin{equation}\label{c:eq_qmodel}
q(y)=\frac{c_*}{|y|^2}+q_1(y),\qquad |q_1(y)|\le C|y|^{-2-\epsilon},
\end{equation}
with $c_*$ and $\epsilon>0$ as in \eqref{c:eq_model}.

\emph{Step 1: the model estimate in an exterior region.} Let $u$ solve $\mathcal M_0u=F$ on $\{|y|\ge R\}$. Expanding in spherical harmonics and applying Lemma \ref{c:lem_model} to $G(t)=e^{2t}F_\ell(e^t)$, which satisfies $|G|\le\|F\|_{\sigma+2}e^{-\sigma t}$, produces for each $\ell$ a solution with $|u_\ell|\le K_\sigma\|F_\ell\|_{\sigma+2}r^{-\sigma}$, the constant $K_\sigma$ being uniform in $\ell$ by Lemma \ref{c:lem_window}. The same computation in $L^2(\mathbb{R},dt)$ instead of $L^\infty$ is legitimate, because the two solution operators of Lemma \ref{c:lem_model} are convolutions with kernels of $L^1$ norm $|\mathrm{Re}\,\gamma_\pm-\sigma|^{-1}$, and gives the same bound for the $L^2$ norms, again uniformly in $\ell$. Summing the orthogonal series yields $u$ with
\[
\Big(\int_{|y|\ge R}\big(|y|^{\sigma}|u|\big)^2\frac{dy}{|y|^n}\Big)^{1/2}\le K_\sigma\Big(\int_{|y|\ge R}\big(|y|^{\sigma+2}|F|\big)^2\frac{dy}{|y|^n}\Big)^{1/2}.
\]

\emph{Step 2: from weighted $L^2$ to weighted $C^{2,\varrho}$.} Let $A_k=\{2^k\le|y|\le2^{k+1}\}$ and rescale $y=2^kz$, $u_k(z)=u(2^kz)$, $F_k(z)=2^{2k}F(2^kz)$. Then $u_k$ solves an equation of the same form on the fixed annulus $\{1\le|z|\le2\}$ with coefficients bounded independently of $k$, so interior Schauder estimates give
$\|u_k\|_{C^{2,\varrho}(\{5/4\le|z|\le7/4\})}\le C\big(\|u_k\|_{L^2(\{1\le|z|\le2\})}+\|F_k\|_{C^{0,\varrho}(\{1\le|z|\le2\})}\big)$,
with $C$ independent of $k$. Undoing the scaling converts the weighted $L^2$ bound of Step 1 into $\|u\|_\sigma\le C\big(K_\sigma+1\big)\|F\|_{C^{0,\varrho}_{\sigma+2}}$ together with the corresponding bounds on $\nabla u$ and $D^2u$. We use this passage without further comment below.

\emph{Step 3: the perturbation $q_1$.} By \eqref{c:eq_model} and \eqref{c:eq_qmodel}, multiplication by $q_1$ maps $\{\|\cdot\|_\sigma<\infty\}$ into $\{\|\cdot\|_{\sigma+2}<\infty\}$ with norm at most $CR^{-\epsilon}$ on $\{|y|\ge R\}$. Fixing $R=R_1$ so large that $CK_\sigma R_1^{-\epsilon}\le\frac12$ and summing a Neumann series upgrades Steps 1 and 2 to the following statement: there is a linear operator $\mathcal S$ which to each $F$ associates a function $\mathcal S F$ defined on $\{|y|\ge R_1\}$ with $\mathcal M\mathcal SF=F$ there and $\|\mathcal SF\|_\sigma+\|\nabla\mathcal SF\|_{\sigma+1}\le C\|F\|_{\sigma+2}$.

\emph{Step 4: the functional setting.} Let $d\mu:=|y|^{-n}dy$, which in the variable $t=\log r$ is $dt\,d\omega$, and for $s\in\mathbb{R}$ let
\[
L^2_s:=\Big\{u \text{ radial}:\ \|u\|_{L^2_s}^2:=\int_{\mathbb{R}^n}\big(|y|^{s}|u|\big)^2d\mu<\infty\Big\}.
\]
Under the Lebesgue pairing $\langle u,v\rangle=\int uv\,dy=\int(|y|^su)(|y|^{n-s}v)\,d\mu$ one has $(L^2_s)^*=L^2_{n-s}$, and $\mathcal M$ is symmetric for this pairing. We regard $\mathcal M$ as an unbounded operator
\[
\mathcal M:\ \mathcal D:=\{u\in L^2_\sigma:\ \mathcal Mu\in L^2_{\sigma+2}\}\ \longrightarrow\ L^2_{\sigma+2}.
\]
Steps 1 to 3, carried out in $L^2$ as indicated, say that for $u$ supported in $\{|y|\ge2R_1\}$ one has $\|u\|_{L^2_\sigma}\le C\|\mathcal Mu\|_{L^2_{\sigma+2}}$.

\emph{Step 5: kernel, a priori estimate and closed range.} Let $\mathcal K:=\{u\in\mathcal D:\mathcal Mu=0\}$. In the radial class the solutions of $\mathcal Mu=0$ that are regular at the origin form a one-dimensional space, so $\dim\mathcal K\le1$. Let $\mathcal Z$ be the $L^2(B_{R_1})$ orthogonal complement of $\mathcal K$ in $\mathcal D$; the pairing $\langle u,k\rangle_{L^2(B_{R_1})}$ is nondegenerate on $\mathcal K$ by unique continuation. We claim
\begin{equation}\label{c:eq_apriori}
\|u\|_{L^2_\sigma}\le C\|\mathcal Mu\|_{L^2_{\sigma+2}}\qquad\text{for }u\in\mathcal Z.
\end{equation}
Suppose not, and take $u_j\in\mathcal Z$ with $\|u_j\|_{L^2_\sigma}=1$ and $\|\mathcal Mu_j\|_{L^2_{\sigma+2}}\to0$. With $\eta$ a cut-off vanishing on $B_{R_1}$ and equal to $1$ off $B_{2R_1}$, Step 3 applied to $\eta u_j$ and the identity $\mathcal M(\eta u_j)=\eta\mathcal Mu_j-2\nabla\eta\cdot\nabla u_j-u_j\Delta\eta$ give
$1\le C(\|\mathcal Mu_j\|_{L^2_{\sigma+2}}+\|u_j\|_{H^1(B_{2R_1})})$, so $\|u_j\|_{H^1(B_{2R_1})}\ge c>0$. By Step 3 and interior elliptic estimates the sequence is bounded in $H^2_{\rm loc}$, so a subsequence converges weakly in $H^2_{\rm loc}$, strongly in $H^1(B_{2R_1})$, to some $u_\infty$ with $\mathcal Mu_\infty=0$, $\|u_\infty\|_{L^2_\sigma}\le1$ and $u_\infty\ne0$, that is $u_\infty\in\mathcal K\setminus\{0\}$; but the orthogonality conditions pass to the limit and force $u_\infty=0$. This proves \eqref{c:eq_apriori}, and with it that $\mathcal M(\mathcal D)=\mathcal M(\mathcal Z)$ is closed in $L^2_{\sigma+2}$.

\emph{Step 6: the cokernel is trivial.} Since $\mathcal M$ is symmetric and $(L^2_{\sigma+2})^*=L^2_{n-\sigma-2}$, the orthogonal complement of the range of $\mathcal M$ is
\[
\mathcal K^*=\{v\in L^2_{\sigma^*}:\ \mathcal Mv=0\ \text{ in }\mathcal D'(\mathbb{R}^n)\},\qquad \sigma^*:=n-2-\sigma.
\]
Because $\sigma<\beta=\frac{n-2}{2}$ we have $\sigma^*>\beta$. Let $v\in\mathcal K^*$ and set $V(t)=v(e^t)$. For the constant coefficient operator $\mathcal A_0$ of Lemma \ref{c:lem_factor} the homogeneous solutions are $e^{-\gamma^{(0)}_\pm t}$, and $|e^{-\gamma^{(0)}_\pm t}|=e^{-\beta t}$ because the two roots are complex conjugate with real part $\beta$ when $7\le n\le15$, while for $n\ge16$ they are real and $\gamma^{(0)}_-\ge\beta-\frac{\sqrt{D_0}}{2}$; and in that case $\sigma^*=n-2-\sigma>n-5>\gamma^{(0)}_+$, so that the weight exceeds \emph{both} indicial roots. In either case, therefore, setting $\gamma_\sharp:=\max\{\mathrm{Re}\,\gamma^{(0)}_\pm\}<\sigma^*$, every nonzero solution of $\mathcal A_0W=0$ satisfies $\limsup_{t\to\infty}e^{\gamma_\sharp t}|W|>0$. Since $v\in L^2_{\sigma^*}$, the transformed function satisfies $e^{\sigma^*t}V\in L^2(t_0,\infty)$. Standard local estimates for the second order ODE satisfied by $V$ (whose coefficients are bounded for large $t$) give
\[
|V(t)|+|V'(t)|\le C\|V\|_{L^2(t-1,t+1)}\le C e^{-\sigma^*t}\|V\|_{L^2_{\sigma^*}},
\qquad t\ge t_0+1.
\]
Thus the bootstrap below has the required seed $|V(t)|\le Me^{-\sigma^*t}$. In the $t$ variable the equation reads $\mathcal A_0V=bV$ with $b(t):=e^{2t}q_1(e^t)$, so $|b(t)|\le C_be^{-\epsilon t}$ by \eqref{c:eq_qmodel}. More generally, suppose $|V|\le Me^{-\varsigma t}$ on $[t_0,\infty)$ for some $\varsigma\ge\sigma^*$. Then $|bV|\le C_bMe^{-(\varsigma+\epsilon)t}$, and Lemma \ref{c:lem_model}, applied with this source and the weight $\varsigma+\epsilon$, produces a particular solution $V_p$ with $|V_p|\le KC_bMe^{-(\varsigma+\epsilon)t}$, where, in the radial class, $K_\varsigma=(\varsigma-\beta)^{-2}$ for $\varsigma>\beta$ when $7\le n\le15$, and $K_\varsigma=|\varsigma-\gamma^{(0)}_+|^{-1}|\varsigma-\gamma^{(0)}_-|^{-1}$ when $n\ge16$, so that $K:=\sup_{\varsigma\ge\sigma^*}K_\varsigma=K_{\sigma^*}<\infty$. The difference $V-V_p$ solves $\mathcal A_0(V-V_p)=0$, hence is a combination of $e^{-\gamma^{(0)}_\pm t}$; both $V$ and $V_p$ are $o(e^{-\gamma_\sharp t})$, so that combination vanishes and $|V|\le KC_bMe^{-(\varsigma+\epsilon)t}$ on $[t_0,\infty)$. Iterating from $\varsigma=\sigma^*$,
\[
|V(t)|\le M\,(KC_b)^{k}\,e^{-(\sigma^*+k\epsilon)t},\qquad k\in\mathbb{N},\ t\ge t_0.
\]
Choosing $t_0$ so large that $KC_be^{-\epsilon t_0}\le\frac12$ and letting $k\to\infty$ gives $V\equiv0$ on $[t_0,\infty)$, hence $v\equiv0$ by unique continuation. So $\mathcal K^*=\{0\}$ and, the range being closed, $\mathcal M$ maps $\mathcal D$ onto $L^2_{\sigma+2}$. Together with \eqref{c:eq_apriori} this gives a bounded right inverse on $L^2$, and Step 2 upgrades it to the weighted H\"older spaces. Undoing the rescaling \eqref{c:eq_scaling} inserts the factor $\varepsilon^{-2}$.

This is the place, and the only one, where the upper constraint $\sigma<\frac{n-2}{2}$ of Lemma \ref{c:lem_window} is used: it is exactly the statement that the dual weight $\sigma^*$ lies on the other side of the indicial root $\beta$, which is what makes the cokernel vanish. The window of Lemma \ref{c:lem_window} is nonempty precisely when $n\ge7$, that is $N\ge8$.
\end{proof}

\begin{remark}
No comparison argument can replace Proposition \ref{c:prop_rightinverse} in the range $7\le n\le15$. Indeed
$\mathcal M_0(|y|^{-\sigma})=-[\sigma^2-(n-2)\sigma+4(n-4)]|y|^{-\sigma-2}$ and the bracket is positive for every $\sigma$ when $D_0<0$, so every power is a strict subsolution; more generally $-\Delta-c|y|^{-2}$ has a positive supersolution near infinity if and only if $c\le\frac{(n-2)^2}{4}$, which for $c=4(n-4)$ means $n\ge16$.
\end{remark}

\section{Fixed-point construction}\label{c:sec_fixed}

Look for a solution of \eqref{eq_phiL} in the form
\begin{equation}\label{c:eq_final}
\phi=\zeta_{\lambda_\alpha+\psi}(x_N)+\chi,\qquad \int_0^\infty\chi(x',t)Z(t)\,dt=0 ,
\end{equation}
with unknowns $\psi:\mathbb{R}^n\to\mathbb{R}$ and $\chi$. Projecting \eqref{eq_phiL} on $Z$ and using $\int\zeta_aZ=\nu a$ gives the pair
\begin{align}
-\Delta(\lambda_\alpha+\psi)&=\mathfrak g(\lambda_\alpha+\psi)+\frac1\nu\int_0^\infty\big[R(\cdot,\zeta+\chi)-R(\cdot,\zeta)\big]Z\,dt ,\label{c:eq_sys1}\\
L\chi&=\Pi^\perp\Big[R(\cdot,\zeta+\chi)-R(\cdot,\zeta)-\mathcal N[\lambda_\alpha+\psi]\Big],\label{c:eq_sys2}
\end{align}
where $\zeta=\zeta_{\lambda_\alpha+\psi}$. The pair \eqref{c:eq_sys1}, \eqref{c:eq_sys2} is exactly equivalent to \eqref{eq_phiL} for $\phi=\zeta+\chi$ with $\int\chi Z\,dt=0$: the $Z$ component of the equation is \eqref{c:eq_sys1} by Lemma \ref{c:lem_orth}, and the orthogonal component is \eqref{c:eq_sys2}. Subtracting the reduced equation \eqref{c:eq_red} from \eqref{c:eq_sys1} turns the first line into
\begin{equation}\label{c:eq_sys1b}
\mathcal M\psi=\big[\mathfrak g(\lambda_\alpha+\psi)-\mathfrak g(\lambda_\alpha)-\mathfrak g'(\lambda_\alpha)\psi\big]
+\frac1\nu\int_0^\infty\big[R(\cdot,\zeta+\chi)-R(\cdot,\zeta)\big]Z\,dt.
\end{equation}
Note that the error $\mathcal N[\lambda_\alpha+\psi]$ does not appear in \eqref{c:eq_sys1b}: by \eqref{c:eq_proj} its $Z$ component is precisely $-\nu(\Delta_{x'}(\lambda_\alpha+\psi)+\mathfrak g(\lambda_\alpha+\psi))$, which is already on the left hand side of \eqref{c:eq_sys1}. This is the point of the normalization \eqref{c:eq_norm}: no scalar solvability condition survives.

\begin{lemma}\label{c:lem_weighttrans}
Let $\sigma>0$. There is $\mu_0=\mu_0(p)\in(0,1)$ such that, for every $\mu\in(0,\mu_0)$, there are $C>0$ and $\varepsilon_1>0$ with the following property: for $\varepsilon\in(0,\varepsilon_1)$ and every $\chi$ with $\int_0^\infty\chi Z\,dx_N=0$, $\chi=0$ on $\partial\mathbb{R}^N_+$ and $e^{\mu x_N}\chi\in L^2(\mathbb{R}^N_+)$,
\[
\|\chi\|_*\le C\,\|L\chi\|_{**},
\qquad
\|\chi\|_{C^{2,\varrho}_*}\le C\,\|L\chi\|_{C^{0,\varrho}_{**}}.
\]
\end{lemma}

\begin{proof}
Write $H:=\{u\in L^2(0,\infty):\int_0^\infty uZ\,dt=0\}$ and, for a function $\chi$ on $\mathbb{R}^N_+$,
\[
u(x',\cdot):=e^{\mu x_N}\chi(x',\cdot),\qquad \Phi(x'):=\|u(x',\cdot)\|_{L^2(0,\infty)}.
\]
The hypothesis $e^{\mu x_N}\chi\in L^2(\mathbb{R}^N_+)$ is what makes $\Phi$ finite for almost every $x'$ and square integrable on $\mathbb{R}^n$; it is \emph{not} a consequence of $\|\chi\|_*<\infty$, which bounds $e^{\mu x_N}|\chi|$ but does not make it square integrable in $x_N$. In Corollary \ref{c:cor_weighttrans} the hypothesis is verified directly, by an Agmon estimate, for the solutions to which the lemma is applied. Apart from that hypothesis the argument is fiberwise in $x_N$ and pointwise in $x'$.

\emph{Step 1: fiberwise coercivity.} The conjugated fiber operator is $\mathcal L_\mu:=e^{\mu t}\mathcal Le^{-\mu t}=\mathcal L+2\mu\partial_t-\mu^2$, so that for $u$ vanishing at $t=0$
\[
\int_0^\infty(\mathcal L_\mu u)u\,dt=\int_0^\infty(\mathcal Lu)u\,dt-\mu^2\!\int_0^\infty\! u^2dt-\mu\,u(0)^2
\ \ge\ \int_0^\infty(\mathcal Lu)u\,dt-\mu^2\!\int_0^\infty\! u^2dt.
\]
The conjugation destroys the orthogonality: $\int\chi Z\,dt=0$ does not give $\int uZ\,dt=0$. Decompose $u=\hat u+a Z$ with $\hat u\in H$. From $\int_0^\infty e^{-\mu t}uZ\,dt=0$ we get $a\nu=\int_0^\infty(1-e^{-\mu t})uZ\,dt$, and $|1-e^{-\mu t}|\le\mu t$ with $tZ\in L^2(0,\infty)$ give $|a|\,\nu\le C\mu\|u\|_{L^2}$. Hence, by Lemma \ref{c:lem_transverse} applied on $H$,
\[
\int_0^\infty(\mathcal L_\mu u)u\,dt\ \ge\ \big(\gamma-C\mu^2\big)\|u\|_{L^2}^2\ \ge\ \tfrac{\gamma}{2}\,\|u\|_{L^2}^2 ,
\]
this is where $\mu<\mu_0$ is used, $\mu_0$ being chosen so that $C\mu_0^2\le\gamma/2$ with $C$ the constant just produced, which depends only on $p$.

\emph{Step 2: a scalar inequality in $x'$.} Multiplying $-\Delta_{x'}u+\mathcal L_\mu u=e^{\mu x_N}F$ by $u$, integrating in $x_N$ only, and using Kato's inequality $\Phi\,\Delta_{x'}\Phi\ge\int_0^\infty(\Delta_{x'}u)u\,dx_N$, we obtain
\[
-\Delta_{x'}\Phi+\tfrac{\gamma}{2}\,\Phi\ \le\ \big\|e^{\mu x_N}F(x',\cdot)\big\|_{L^2(0,\infty)}\ \le\ C\|F\|_{**}\,(1+\varepsilon|x'|)^{-\sigma-2}
\]
in the sense of distributions on $\mathbb{R}^n$. The fiber norm on the right is finite precisely because of the factor $(1+x_N)^{2}$ in $\|\cdot\|_{**}$: it gives $e^{\mu x_N}|F|\le\|F\|_{**}(1+\varepsilon|x'|)^{-\sigma-2}(1+x_N)^{-2}$, which is square integrable in $x_N$, whereas $\|\cdot\|_*$ alone would only bound $e^{\mu x_N}|F|$ by a constant in $x_N$. This is also why no global integration over $\mathbb{R}^N_+$ appears anywhere in the proof.

\emph{Step 3: a barrier.} Let $\Psi(x'):=(1+\varepsilon|x'|)^{-\sigma-2}$. A direct computation gives $|\Delta_{x'}\Psi|\le C\varepsilon^2\Psi$, so
$-\Delta_{x'}\Psi+\frac{\gamma}{2}\Psi\ge\frac{\gamma}{4}\Psi$ for $\varepsilon\le\varepsilon_1$. Let $G$ be the Green kernel of $-\Delta_{x'}+\frac{\gamma}{2}$ on $\mathbb{R}^n$, which is positive and decays exponentially. Set $W:=\Phi-C\|F\|_{**}\,G*\Psi$, so that $(-\Delta_{x'}+\frac{\gamma}{2})W\le0$. We do not claim $W\in L^2(\mathbb{R}^n)$, since $\Psi$ need not be square integrable in high dimensions; but $W_+\le\Phi\in L^2(\mathbb{R}^n)$ because $G*\Psi\ge0$, so $W_+$ is an admissible test function, and testing with it gives $\int|\nabla W_+|^2+\frac{\gamma}{2}\int W_+^2\le0$, whence $W\le0$. Since $G$ decays exponentially and $\Psi$ varies on the scale $1/\varepsilon$, one has $G*\Psi\le C\Psi$, whence
\[
\Phi(x')\le C\,\|F\|_{**}\,\Psi(x')\qquad\text{on }\mathbb{R}^{n}.
\]
No decay of $\Phi$ at infinity is assumed.

\emph{Step 4: back to pointwise bounds.} On each unit cube $Q$ of a lattice in $\overline{\mathbb{R}^N_+}$, on which the weights vary by a bounded factor, interior and boundary Schauder estimates for $L$ give $\|\chi\|_{C^{2,\varrho}(Q)}\le C(\|F\|_{C^{0,\varrho}(Q)}+\|\chi\|_{L^2(Q)})$. Combining with Step 3 yields
\[
\|\chi\|_*\le C\|F\|_{**}
\qquad\text{and}\qquad
\|\chi\|_{C^{2,\varrho}_*}\le C\|F\|_{C^{0,\varrho}_{**}} ,
\]
where the H\"older norms carry the same weights.
\end{proof}

The estimate of Lemma \ref{c:lem_weighttrans} is a priori: Lemma \ref{c:lem_transverse} produces a solution only for $F\in L^2(\mathbb{R}^N_+)$, whereas a forcing with $\|F\|_*<\infty$ need not be square integrable when $2(\sigma+2)\le n$. Existence in the weighted class follows by exhaustion.

\begin{corollary}\label{c:cor_weighttrans}
Under the hypotheses of Lemma \ref{c:lem_weighttrans}, there is a bounded linear operator $\mathcal K$ which to every $F$ with $\int_0^\infty FZ\,dx_N=0$ and $\|F\|_{C^{0,\varrho}_{**}}<\infty$ associates a solution $\chi=\mathcal KF$ of $L\chi=F$ with $\int_0^\infty\chi Z\,dx_N=0$ and
\[
\|\mathcal KF\|_{C^{2,\varrho}_*}\le C\|F\|_{C^{0,\varrho}_{**}}.
\]
The solution is unique in the stated weighted $C^{2,\varrho}$ class.
\end{corollary}

\begin{proof}
Let $\eta\in C^\infty_c(\mathbb{R})$ satisfy $\eta=1$ on $[0,1]$ and $\eta=0$ on $[2,\infty)$, and set
\[
F_R(x',x_N):=\eta(|x'|/R)F(x',x_N).
\]
Then $\|F_R\|_{C^{0,\varrho}_{**}}\le C\|F\|_{C^{0,\varrho}_{**}}$ uniformly in $R$; moreover $F_R\in L^2(\mathbb{R}^N_+)$ and remains orthogonal to $Z$. Let $\chi_R$ be the solution given by Lemma \ref{c:lem_transverse}.

The hypothesis of Lemma \ref{c:lem_weighttrans} holds for $\chi_R$. To see this rigorously, use the Lipschitz phase $\varphi_T(t):=\mu\min\{t,T\}$, or a smooth approximation with $0\le\varphi_T'\le\mu$. Testing $L\chi_R=F_R$ with $e^{2\varphi_T}\chi_R$ and integrating by parts gives the same conjugated fiber form as in Step 1, with the loss bounded by $C\|\varphi_T'\|_\infty^2\le C\mu^2$. The orthogonality correction is also unchanged, since $|1-e^{-\varphi_T(t)}|\le\varphi_T(t)\le\mu t$. Together with $\int|\nabla_{x'}\chi_R|^2e^{2\varphi_T}\ge0$ this yields
\[
\big(\gamma-C\mu^2\big)\int_{\mathbb{R}^N_+}e^{2\varphi_T}\chi_R^2
\le \int_{\mathbb{R}^N_+}e^{2\varphi_T}|F_R\chi_R|.
\]
By Cauchy's inequality the right hand side is absorbed into half of the left hand side plus $C\int e^{2\varphi_T}F_R^2$. The latter is bounded independently of $T$, because $F_R$ has compact support in $x'$ and $e^{\mu x_N}|F_R|\le C\|F\|_{**}(1+x_N)^{-2}$. Hence $\int e^{2\varphi_T}\chi_R^2\le C_R\|F\|_{**}^2$ uniformly in $T$, and letting $T\to\infty$ gives $e^{\mu x_N}\chi_R\in L^2(\mathbb{R}^N_+)$. By Lemma \ref{c:lem_weighttrans}, $\|\chi_R\|_{C^{2,\varrho}_*}\le C\|F_R\|_{C^{0,\varrho}_{**}}\le C\|F\|_{C^{0,\varrho}_{**}}$ with $C$ independent of $R$, the lemma being applicable because its proof uses only fiberwise integrals. Interior and boundary Schauder estimates give bounds in $C^{2,\varrho}_{\rm loc}$, so a subsequence converges in $C^2_{\rm loc}$ to a solution $\chi$ of $L\chi=F$ with $\|\chi\|_{C^{2,\varrho}_*}\le C\|F\|_{C^{0,\varrho}_{**}}$ and $\int\chi Z=0$.

It remains to prove uniqueness in this weighted class; this also makes the exhaustion limit canonical and hence linear in $F$. Let $L\chi=0$, $\int_0^\infty\chi Z\,dt=0$ and $\|\chi\|_{C^{2,\varrho}_*}<\infty$. Choose $0<\mu'<\mu$ and a cutoff $\eta_R(x')=\eta(|x'|/R)$, and put $u_R:=\eta_R\chi$. Then $e^{\mu' x_N}u_R\in L^2$, $u_R=0$ on the boundary and $\int u_RZ=0$, so Lemma \ref{c:lem_weighttrans}, with exponent $\mu'$, gives
\[
\|u_R\|_{*,\mu'}\le C\|Lu_R\|_{**,\mu'}.
\]
Here and below the subscript $\mu'$ means that the same norm is used with the exponential weight $e^{\mu' x_N}$ in place of $e^{\mu x_N}$. Moreover,
\[
Lu_R=-2\nabla\eta_R\cdot\nabla\chi-(\Delta\eta_R)\chi,
\]
and the surplus $e^{-(\mu-\mu')x_N}$ absorbs the factor $(1+x_N)^2$ in the double-star norm. Therefore
\[
\|Lu_R\|_{**,\mu'}\le C\Big(R^{-1}+R^{-2}\Big)\|\chi\|_{C^{2,\varrho}_{*,\mu}}\longrightarrow0.
\]
For each fixed point $x$, $u_R(x)=\chi(x)$ for all sufficiently large $R$, hence $\chi\equiv0$. Thus the weighted solution is unique. It follows that every convergent subsequence of the exhaustion has the same limit, and the unique solution defines a bounded linear right inverse $\mathcal K$: linearity follows by applying uniqueness to sums and scalar multiples.
\end{proof}

\begin{proof}[Proof of Theorem \ref{thm_construction}]
Fix $\sigma$ as in Lemma \ref{c:lem_window} and $\mu\in(0,\mu_0)$ with $\mu_0$ as in Lemma \ref{c:lem_weighttrans}, write $\lambda:=\lambda_\alpha+\psi$, $\zeta:=\zeta_{\lambda}$, and set
\[
\|\chi\|_*:=\sup_{x\in\mathbb{R}^N_+}(1+\varepsilon|x'|)^{\sigma+2}e^{\mu x_N}|\chi(x)| ,
\qquad
\|F\|_{**}:=\sup_{x\in\mathbb{R}^N_+}(1+\varepsilon|x'|)^{\sigma+2}(1+x_N)^{2}e^{\mu x_N}|F(x)|.
\]
Set $0<\varrho<\min\{1,p-1\}$ and
\[
\mathcal B:=\Big\{(\psi,\chi):\ \|\psi\|_{\sigma,\varepsilon}\le K_2\alpha^3,\quad
\big\|\Delta_{x'}\psi\big\|_{C^{0,\varrho}_{\sigma+2,\varepsilon}}\le K_3\alpha^4,\quad
\|\chi\|_{C^{2,\varrho}_*}\le K_1\alpha^3\Big\}.
\]
the second constraint being taken in the weighted H\"older norm of Proposition \ref{c:prop_rightinverse}, so that the right hand side of \eqref{c:eq_sys1b} lies in the space in which that proposition provides a right inverse. Let $C_0,\dots,C_6$ denote constants depending only on $N$, $p$ and $\sigma$, and never on $K_1,K_2,K_3$ or $\alpha$. Lemma \ref{c:lem_Esize}, Proposition \ref{c:prop_cell} and \eqref{c:eq_twosided} give, for $(\psi,\chi)\in\mathcal B$ and $\alpha$ small, the estimates
\begin{align}
\big\|\Pi^\perp\mathcal N[\lambda]\big\|_{C^{0,\varrho}_{**}}&\le C_1\alpha^3+C_0\big\|\Delta_{x'}\psi\big\|_{C^{0,\varrho}_{\sigma+2,\varepsilon}} ,\label{c:eq_e1}\\
\big\|\Pi^\perp\big[R(\cdot,\zeta+\chi)-R(\cdot,\zeta)\big]\big\|_{C^{0,\varrho}_{**}}&\le C_2\big(\alpha\|\chi\|_*+\|\chi\|_*^{1+\varpi_0}\big),\quad \varpi_0:=\min\{\varpi,p-1\}>0 ,\label{c:eq_e2}\\
\Big\|\tfrac1\nu\!\int_0^\infty\!\big[R(\cdot,\zeta+\chi)-R(\cdot,\zeta)\big]Z\,dt\Big\|_{\sigma+2,\varepsilon}&\le C_3\,\alpha\|\chi\|_* ,\label{c:eq_e3}\\
\big\|\mathfrak g(\lambda_\alpha+\psi)-\mathfrak g(\lambda_\alpha)-\mathfrak g'(\lambda_\alpha)\psi\big\|_{\sigma+2,\varepsilon}&\le C_4\|\psi\|_{\sigma,\varepsilon}^2 ,\label{c:eq_e4}
\end{align}
where \eqref{c:eq_e1} uses $\sigma+2\le6$ and \eqref{c:eq_e4} uses $2\sigma\ge\sigma+2$. The factor $(1+x_N)^{2}$ present in $\|\cdot\|_{**}$ but not in $\|\cdot\|_*$ is available in both places. In \eqref{c:eq_e1} we use \eqref{c:eq_Esize_doublestar} with an exponent $\widehat\mu\in(\mu,1)$; the surplus $e^{-(\widehat\mu-\mu)t}$ is precisely what absorbs the factor $(1+t)^2$ in the double-star norm. This avoids any need for a uniform pointwise expansion with the sharp rate $e^{-t}$. In \eqref{c:eq_e2} it is supplied by the factor $e^{-(p-1)x_N}$ of the near region below, together with the absorption
\[
(1+T_*)^2\,d^{\varpi}\le C\,d^{\,p-1}=C\,d^{\varpi_0},\qquad d:=\|\chi\|_* ,
\]
at the transition point $T_*\le C\log\frac1d$ defined below, which holds because $\varpi=\frac{p-1}{1-\mu}>p-1=\varpi_0$; the H\"older seminorms are estimated in the same way, all the functions involved being smooth in $x_N$ and $\zeta$, $h$, $\lambda_\alpha$ smooth in $x'$, while $\chi$ is controlled in $C^{2,\varrho}$ by the definition of $\mathcal B$. Estimates \eqref{c:eq_e2} and \eqref{c:eq_e3} require a word, because $\chi$ is controlled only in the weight $e^{-\mu x_N}$ with $\mu<1$, so that $|\chi|$ exceeds $w_0+\zeta\asymp e^{-x_N}$ in a far region, where $R$ is not quadratic. Split $(0,\infty)$ at the level where $|\chi|=\frac14w_0$, noting that $\zeta\le\frac14w_0$ for $\alpha$ small by \eqref{c:eq_Rquad}. On the near region the first bound in \eqref{c:eq_Rtwo} and the mean value theorem give
\[
\big|R(\cdot,\zeta+\chi)-R(\cdot,\zeta)\big|\le Cw_0^{p-2}\big(\zeta+|\chi|\big)|\chi| ,
\]
where $w_0^{p-2}\zeta\le C\alpha e^{-(p-1)x_N}\le C\alpha$, and where
\[
w_0^{p-2}|\chi|\le C\min\Big\{\|\chi\|_*e^{(2-p-\mu)x_N},\ w_0^{p-1}\Big\}\le C\|\chi\|_*^{\varpi},
\qquad \varpi:=\frac{p-1}{1-\mu}>0 ,
\]
the two expressions in the minimum being equal at $x_N=T_*:=\frac{1}{1-\mu}\log\frac{1}{C\|\chi\|_*}$, where their common value is $Ce^{-(p-1)T_*}=C\|\chi\|_*^{\varpi}$. Carrying the factor $(1+x_N)^2$ of $\|\cdot\|_{**}$ through this bound costs $(1+T_*)^2\le C\log^2\frac1{\|\chi\|_*}$ at the transition, absorbed by $\varpi>\varpi_0$ as recorded above. This gives \eqref{c:eq_e2} on the near region, and \eqref{c:eq_e3} follows by integrating against $Z$, the integral $\int w_0^{p-2}\chi^2Z\,dx_N\le C\|\chi\|_*^2$ converging because $2\mu>1-p$. On the far region one has $w_0\le 4|\chi|$, hence $x_N\ge T_*$, and the second bound in \eqref{c:eq_Rtwo} gives
\[
\big\|R(\cdot,\zeta+\chi)-R(\cdot,\zeta)\big\|_{**}\le C\|\chi\|_*^{p},
\qquad
\int_{T_*}^\infty|\chi|^pZ\,dx_N\le C\|\chi\|_*^{\,p+\frac{p\mu+1}{1-\mu}}\le C\|\chi\|_*^{p+1},
\]
the first because $(1-p)(\sigma+2)\le0$ and $(1-p)\mu\le0$. Since $p>1$, both are of higher order than the terms retained in \eqref{c:eq_e2} and \eqref{c:eq_e3}. The term $C_0\|\Delta_{x'}\psi\|_{\sigma+2,\varepsilon}$ in \eqref{c:eq_e1} is the contribution of $\partial_a\zeta_\lambda\,\Delta_{x'}\lambda$ in Lemma \ref{c:lem_error}, since $\Delta_{x'}\lambda=-\mathfrak g(\lambda_\alpha)+\Delta_{x'}\psi$ and $\partial_a\zeta_\lambda$ is bounded in $\mathcal X_\mu$. The two linear bounds are $\|\chi\|_{C^{2,\varrho}_*}\le C_5\|L\chi\|_{C^{0,\varrho}_{**}}$, from Corollary \ref{c:cor_weighttrans}, and $\|\psi\|_{\sigma,\varepsilon}\le C_6\varepsilon^{-2}\|\mathcal M\psi\|_{C^{0,\varrho}_{\sigma+2,\varepsilon}}$ from Proposition \ref{c:prop_rightinverse}; the right hand sides of \eqref{c:eq_e3} and \eqref{c:eq_e4} are estimated in that H\"older norm as well, which is legitimate because $\zeta_a$, $h$ and $\lambda_\alpha$ are smooth in $x'$, because $\chi$ enters through $\Pi^\perp$ and therefore through an $x_N$ integral of a $C^{2,\varrho}_{\rm loc}$ function, and because the third constraint defining $\mathcal B$ controls $\Delta_{x'}\psi$ in that norm.

Choose the three radii in this order, each in terms of the previous ones and of $C_0,\dots,C_7$ only:
\begin{equation}\label{c:eq_radii}
K_1:=2C_5C_1,\qquad K_2:=\frac{2C_6C_3}{C_p}K_1,\qquad K_3:=2\big(C_3K_1+C_7K_2\big),
\end{equation}
where $C_7$ is the constant in $\|\mathfrak g'(\lambda_\alpha)\psi\|_{\sigma+2,\varepsilon}\le C_7\alpha\|\psi\|_{\sigma,\varepsilon}$, which follows from $\mathfrak g'(\lambda_\alpha)\le C\alpha(1+\varepsilon|x'|)^{-2}$ and \eqref{c:eq_twosided}. Then let $\alpha\in(0,\alpha_1)$ be small enough that
\begin{equation}\label{c:eq_smallness}
C_5C_0K_3\,\alpha\le\tfrac12 K_1,\qquad
C_5C_2\big(\alpha+(K_1\alpha^3)^{\varpi_0}\big)\le\tfrac12,\qquad
\frac{C_6C_4}{C_p}K_2^2\,\alpha^{2}\le\tfrac12 K_2,\qquad
C_4K_2^2\alpha^{2}\le C_3K_1.
\end{equation}

Note that \eqref{c:eq_radii} is a legitimate order of choice: no $K_i$ occurs in the constants defining it. Let $(\psi,\chi)\in\mathcal B$. By \eqref{c:eq_e1}, \eqref{c:eq_e2} and the first two conditions in \eqref{c:eq_smallness}, the solution $\chi'$ of \eqref{c:eq_sys2} satisfies
\[
\|\chi'\|_{C^{2,\varrho}_*}\le C_5\Big(C_1\alpha^3+C_0K_3\alpha^4\Big)+C_5C_2\big(\alpha+(K_1\alpha^3)^{\varpi_0}\big)\|\chi\|_*
\le \tfrac12K_1\alpha^3+\tfrac12K_1\alpha^3=K_1\alpha^3.
\]
By \eqref{c:eq_e3}, \eqref{c:eq_e4}, $\varepsilon^2=C_p\alpha$ and the third condition, the solution $\psi'$ of \eqref{c:eq_sys1b} satisfies
\[
\|\psi'\|_{\sigma,\varepsilon}\le\frac{C_6}{C_p\alpha}\Big(C_4K_2^2\alpha^6+C_3K_1\alpha^4\Big)
=\frac{C_6C_4}{C_p}K_2^2\alpha^5+\frac{C_6C_3K_1}{C_p}\alpha^3\le\tfrac12K_2\alpha^3+\tfrac12K_2\alpha^3=K_2\alpha^3.
\]
Finally $\Delta_{x'}\psi'=-\mathcal M\psi'-\mathfrak g'(\lambda_\alpha)\psi'$ by the definition of $\mathcal M$, so by the fourth condition
\[
\big\|\Delta_{x'}\psi'\big\|_{C^{0,\varrho}_{\sigma+2,\varepsilon}}\le\Big(C_4K_2^2\alpha^6+C_3K_1\alpha^4\Big)+C_7\alpha\,K_2\alpha^3
\le\big(2C_3K_1+C_7K_2\big)\alpha^4\le K_3\alpha^4.
\]
All the data being radially symmetric in $x'$, so are $\psi'$ and $\chi'$: for $\chi'$ because the right inverse $\mathcal K$ of Corollary \ref{c:cor_weighttrans} is constructed by radial exhaustion from the solutions of Lemma \ref{c:lem_transverse}, which are unique in $L^2$ and hence radial, and for $\psi'$ because Proposition \ref{c:prop_rightinverse} is stated in the radial class. Hence $(\psi,\chi)\mapsto(\psi',\chi')$ maps $\mathcal B$ into itself. For the contraction we use the scaled product metric
\begin{equation}\label{c:eq_metric}
\big\|(\psi,\chi)\big\|_{\mathcal X}:=A_1\alpha^{-3}\|\psi\|_{\sigma,\varepsilon}+A_2\alpha^{-4}\big\|\Delta_{x'}\psi\big\|_{C^{0,\varrho}_{\sigma+2,\varepsilon}}+A_3\alpha^{-3}\|\chi\|_{C^{2,\varrho}_*} ,
\end{equation}
in which each component is normalized by its own radius; the couplings are not symmetric, so a naive sum of the three norms would not do. Writing $\mathcal T$ for the map $(\psi,\chi)\mapsto(\psi',\chi')$ and applying the four estimates to differences gives
\[
\delta\chi'\le o(1)\,\delta\chi+C\alpha^{2}\,\delta\psi+C\,\delta\Delta\psi ,
\qquad
\delta\psi'\le C\,\delta\chi+C\alpha^{2}\,\delta\psi ,
\qquad
\delta\Delta\psi'\le C\alpha\,\delta\chi+C\alpha^{3}\,\delta\psi ,
\]
where $o(1)$ is as $\alpha\to0$, the third follows from $\Delta_{x'}\psi'=-\mathcal M\psi'-\mathfrak g'(\lambda_\alpha)\psi'$ and the second, and the constants do not depend on the $A_j$. Normalizing as in \eqref{c:eq_metric}, the coefficients become
\[
\begin{pmatrix}o(1)&C\alpha^{2}A_3/A_1&C\alpha A_3/A_2\\[1mm] CA_1/A_3&C\alpha^{2}&0\\[1mm] CA_2/A_3&C\alpha^{2}A_2/A_1&0\end{pmatrix} ,
\]
acting on $(\|\delta\chi\|,\|\delta\psi\|,\|\delta\Delta\psi\|)$ in the normalized variables. Fix $A_1=A_2=1$, choose $A_3$ so large that $CA_1/A_3$ and $CA_2/A_3$ are at most $\frac18$, and then $\alpha$ so small that every entry carrying a power of $\alpha$ is at most $\frac18$; the operator norm is then smaller than $\frac12$ and $\mathcal T$ is a contraction, and its fixed point solves \eqref{c:eq_sys1b}, \eqref{c:eq_sys2}, hence \eqref{eq_phiL}.

It remains to check that $v_\alpha:=w_0+\phi$ with $\phi$ as in \eqref{c:eq_final} is a bounded positive solution of \eqref{eq_Pc}. The boundary condition holds because $\zeta_a(0)=0$ and $\chi(\cdot,0)=0$. The decay condition holds because $\zeta_a\le Cae^{-t}$ uniformly and $|\chi|\le K_1\alpha^3e^{-\mu x_N}$, both tending to $0$ as $x_N\to\infty$ uniformly in $x'$. The amplitude is admissible, that is $\lambda_\alpha+\psi\in(0,a_0)$: indeed $|\psi|\le K_2\alpha^3(1+\varepsilon|x'|)^{-\sigma}$ while $\lambda_\alpha\ge c\alpha(1+\varepsilon|x'|)^{-2}$ by \eqref{c:eq_twosided}, and $\sigma>2$, so $|\psi|\le CK_2\alpha^2\lambda_\alpha\le\frac12\lambda_\alpha$ for $\alpha$ small.

Positivity of $v_\alpha$ does not follow from the perturbative estimates, because $\chi$ is controlled only in the weight $e^{-\mu x_N}$ with $\mu<1$, whereas $w_0+\zeta$ decays like $e^{-x_N}$; we obtain it instead from the maximum principle. Set $m:=\inf_{\mathbb{R}^N_+}v_\alpha$ and suppose $m<0$. Since $v_\alpha\to0$ as $x_N\to\infty$ uniformly in $x'$, any sequence $x^j$ with $v_\alpha(x^j)\to m$ has $x^j_N$ bounded, so translating in $x'$ and using the uniform $C^{2,\gamma}$ bounds of Lemma \ref{lem_reg}(ii), which do not require positivity, we obtain along a subsequence a solution $V$ of $\Delta V=V-|V|^{p-1}V$ in $\mathbb{R}^N_+$ with $V=c_p$ on $\partial\mathbb{R}^N_+$, $V\ge m$ and $V(\bar x)=m$ for some $\bar x$. As $m<0<c_p$, the point $\bar x$ is interior, so $\Delta V(\bar x)\ge0$. On the other hand $v_\alpha=w_0+\phi\ge\phi$ because $w_0>0$, so $|m|\le\sup|\phi|\le C\alpha<1$ for $\alpha$ small, and therefore
\[
\Delta V(\bar x)=m\big(1-|m|^{p-1}\big)<0 ,
\]
a contradiction. Hence $v_\alpha\ge0$, and $v_\alpha>0$ in $\mathbb{R}^N_+$ by the strong maximum principle applied to $\Delta v_\alpha-v_\alpha=-v_\alpha^p\le0$ as in Lemma \ref{lem_reg}(0). Finally $v_\alpha$ is not a function of $x_N$ alone, since $\frac1\nu\int_0^\infty(v_\alpha-w_0)Z\,dt=\lambda_\alpha+\psi$ is nonconstant in $x'$.
\end{proof}

\section{The critical dimension \texorpdfstring{$N=7$}{N=7}}\label{c:sec_critical_full}

In this section $n=6$, $p>1$ and $\kappa>0$; the standing assumption $n\ge7$ of Sections \ref{c:sec_prelim} to \ref{c:sec_fixed} is dropped. The results of those sections used below are Lemma \ref{c:lem_gap}, Proposition \ref{c:prop_cell}, Lemmas \ref{c:lem_error} and \ref{c:lem_orth}, Lemma \ref{c:lem_transverse}, Lemma \ref{c:lem_weighttrans} and Corollary \ref{c:cor_weighttrans}, together with the pointwise estimates behind \eqref{c:eq_e1}--\eqref{c:eq_e2} and the positivity argument at the end of the proof of Theorem \ref{thm_construction}. None of their proofs uses $n\ge7$: the first four concern the one dimensional operator $\mathcal L$ and the cell problem, which do not involve $n$, and the last three use only the spectral gap of $\mathcal L$, the barrier $(1+\varepsilon|x'|)^{-\sigma-2}$ and the Green kernel of $-\Delta_{x'}+\gamma/2$, which decays exponentially in every dimension; the estimates behind \eqref{c:eq_e1}--\eqref{c:eq_e2} and the positivity argument are pointwise in $x'$.

The results that do require $n\ge7$, namely Lemma \ref{c:lem_window} and Proposition \ref{c:prop_rightinverse}, are not used. The purpose is to prove Theorem \ref{thm_construction_critical}. The resonance of the reduced linearized operator is avoided altogether: we keep the projected equation nonlinear, solve it as a radial forced ODE for each transverse correction, and use the right inverse $\mathcal K$ of Corollary \ref{c:cor_weighttrans} only in the direction orthogonal to $Z$.

Fix
\[
2<\sigma<4,\qquad 0<\varrho<\min\{1,p-1\},\qquad 0<\mu<\mu_0(p),
\]
and put $\varepsilon^2=C_p\alpha$.  In this section the norms $\|\cdot\|_*$ and $\|\cdot\|_{**}$ are the same as in Section \ref{c:sec_fixed}, with this value of $\sigma$.

\subsection{A forced critical reduced equation}

We first record the robust form of Proposition \ref{c:prop_critical_reduced} which will replace the critical reduced inverse.

\begin{lemma}\label{c:lem_forced_critical}
Let $p>1$, assume $\kappa(p)>0$, let $2<\sigma<4$ and $A>0$.  There are constants $C_0>1$, $\alpha_0>0$ and $C>0$, depending only on $p,\sigma,A$, with the following property.  For $0<\alpha<\alpha_0$ set $\varepsilon^2=C_p\alpha$.  Suppose that $H=H(r,a)$ is continuous, locally Lipschitz in $a$, and, for $0\le a\le C_0\alpha$, satisfies
\begin{equation}\label{c:eq_forcing_bound}
|H(r,a)|\le A\alpha^4(1+\varepsilon r)^{-\sigma-2}.
\end{equation}
Then the radial initial value problem
\begin{equation}\label{c:eq_forced_reduced}
-a''-\frac5r a'=\mathfrak g(a)+H(r,a),\qquad a(0)=\alpha,\qquad a'(0)=0,
\end{equation}
has a positive entire solution.  It satisfies
\begin{equation}\label{c:eq_forced_bounds}
0<a(r)\le C\alpha(1+\varepsilon r)^{-2},\qquad
|a'(r)|\le C\alpha\varepsilon(1+\varepsilon r)^{-3},\qquad
|D^2a(x')|\le C\alpha\varepsilon^2(1+\varepsilon|x'|)^{-4}.
\end{equation}
The constants are uniform for $H$ in bounded sets satisfying \eqref{c:eq_forcing_bound}.
\end{lemma}

\begin{proof}
Fix $R_0>0$.  Since $\kappa>0$, there are $a_*>0$ and $c_G>0$ such that $\mathfrak q'(a)>0$ and $\mathfrak G(a)\ge c_Ga^3$ for $0<a\le a_*$.  Set $C_*:=2/c_G$.  Choose $C_0>1$ so large that
\[
C_0>2\sup_{0\le\rho\le R_0}A_0(\rho)+1,
\qquad
C_0>\frac{C_*C_p}{R_0^2}+1,
\]
where $A_0(\rho)=(1+\rho^2/24)^{-2}$, and then decrease $\alpha_0$ so that $C_0\alpha_0<a_*$.  Recall
\[
Q_6(a):=6\mathfrak G(a)-2a\mathfrak g(a),\qquad \mathfrak G(a)=\int_0^a\mathfrak g(s)\,ds.
\]
As in Proposition \ref{c:prop_critical_reduced},
\begin{equation}\label{c:eq_Q6negative}
Q_6'(a)=-2a^3\mathfrak q'(a)<0,
\qquad
Q_6(a)=-\frac{\kappa}{2}a^4+O(a^5).
\end{equation}
For a solution of \eqref{c:eq_forced_reduced} define the radial Pohozaev quantity
\begin{equation}
\mathcal P(r):=r^5\left[2aa'+r\left(\frac{a'^2}{2}+\mathfrak G(a)\right)\right].
\end{equation}
A direct differentiation, using \eqref{c:eq_forced_reduced}, gives the exact identity
\begin{equation}\label{c:eq_forced_poho_der}
\mathcal P'(r)=r^5Q_6(a(r))-r^5H(r,a(r))\big(2a(r)+ra'(r)\big).
\end{equation}

We first work on the critical core.  Put $r=\rho/\varepsilon$ and $a(r)=\alpha A_\alpha(\rho)$.  On every fixed interval $0\le\rho\le R$, equation \eqref{c:eq_forced_reduced} becomes
\[
-A_\alpha''-\frac5\rho A_\alpha'
=A_\alpha^2\frac{\mathfrak q(\alpha A_\alpha)}{C_p}+\widetilde H_\alpha,
\qquad |\widetilde H_\alpha|\le C\alpha^2,
\]
as long as $0\le A_\alpha\le C_0$.  Standard continuous dependence for this regular radial initial value problem (equivalently, for $(\rho^5A_\alpha')'=-\rho^5[\cdots]$) and the choice of $C_0$ show that the solution cannot leave this strip before $R_0$ for small $\alpha$.  Uniformly for all $H$ satisfying \eqref{c:eq_forcing_bound},
\begin{equation}\label{c:eq_core_bubble_conv}
A_\alpha\longrightarrow A_0(\rho)=\left(1+\frac{\rho^2}{24}\right)^{-2}
\quad\text{in }C^1([0,R_0]).
\end{equation}
Integrating \eqref{c:eq_forced_poho_der} up to $r_0:=R_0/\varepsilon$, using \eqref{c:eq_Q6negative}, \eqref{c:eq_core_bubble_conv} and \eqref{c:eq_forcing_bound}, yields
\begin{align}
\mathcal P(r_0)
&=-\frac{\kappa\alpha}{2C_p^3}
\int_0^{R_0}\rho^5A_0(\rho)^4\,d\rho+o(\alpha)+O(\alpha^2) \\
&\le -c_0\alpha                                                   \label{c:eq_core_negative_energy}
\end{align}
for some $c_0>0$ and all sufficiently small $\alpha$.  Indeed the forcing contribution is $O(\alpha^2)$, because on the fixed scaled interval $2a+ra'=O(\alpha)$ and $r^5dr=\varepsilon^{-6}\rho^5d\rho$.

We now continue the solution from $r_0$.  As long as $a>0$, $a<C_0\alpha$ and $\mathcal P<0$, put
\[
M(t):=r^2a(r),\qquad t=\log r.
\]
Then
\begin{equation}
\mathcal P(r)=\frac12\dot M^2-2M^2+r^6\mathfrak G(r^{-2}M).
\end{equation}
By construction $\mathfrak G(a)\ge c_Ga^3$, and therefore $\mathcal P<0$ implies
\begin{equation}\label{c:eq_M_uniform_forced}
0<M<C_*:=\frac{2}{c_G},
\qquad
|\dot M|\le2M\le2C_*.
\end{equation}
In particular
\[
|2a+ra'|=r^{-2}|\dot M|\le Cr^{-2}.
\]
Using \eqref{c:eq_forcing_bound}, for every $R>r_0$ we consequently have
\begin{align}
\int_{r_0}^{R}s^5|H(s,a(s))|\,|2a(s)+sa'(s)|\,ds
&\le C\alpha^4\int_{r_0}^{\infty}s^3(1+\varepsilon s)^{-\sigma-2}\,ds \\
&\le C\alpha^2,                                                   \label{c:eq_forcing_poho_small}
\end{align}
because $\sigma>2$ and $\varepsilon^4\asymp\alpha^2$.  On the other hand $Q_6(a)\le0$ by \eqref{c:eq_Q6negative}.  A first-crossing argument now closes the bootstrap.  If $r_1>r_0$ were the first point at which either $\mathcal P(r_1)=-\frac{c_0}{2}\alpha$, $a(r_1)=0$, or $a(r_1)=C_0\alpha$, then on $[r_0,r_1]$ one has $\mathcal P<0$ and therefore \eqref{c:eq_M_uniform_forced}; hence \eqref{c:eq_forcing_poho_small} is valid there.  Since the $Q_6$ contribution is nonpositive, \eqref{c:eq_core_negative_energy} and \eqref{c:eq_forcing_poho_small} exclude the first alternative for small $\alpha$.  The bound $M<C_*$ gives $a(r)\le C_*r^{-2}\le C_*C_p\alpha/R_0^2<C_0\alpha$ for $r\ge r_0$, so the third alternative is excluded by the choice of $C_0$.  The second alternative is also impossible, because at a first zero $R$ one would have
$\mathcal P(R)=R^6a'(R)^2/2\ge0$, a contradiction.  Thus none of the three alternatives occurs; the a priori bounds prevent finite-radius breakdown of the ODE, and the solution is positive and entire.

The pointwise estimates now follow.  On $0\le r\le r_0$ they follow from \eqref{c:eq_core_bubble_conv} and the scaled equation; for $r\ge r_0$, \eqref{c:eq_M_uniform_forced} gives
$a(r)\le Cr^{-2}$ and
$|a'(r)|=r^{-3}|\dot M-2M|\le Cr^{-3}$.  The equation then gives $|a''(r)|\le C r^{-4}$ for $r\ge r_0$, while on the core $|a''|\le C\alpha\varepsilon^2$.  For a radial function the tangential Hessian has eigenvalues $a''$ and $a'/r$, so, since $\varepsilon^2=C_p\alpha$, these estimates are exactly \eqref{c:eq_forced_bounds}, after changing $C$.
\end{proof}

\begin{proposition}[The Fowler energy at infinity]\label{c:prop_forced_fowler}
Under the hypotheses of Lemma \ref{c:lem_forced_critical}, let $a$ be the positive entire solution of \eqref{c:eq_forced_reduced} and put
\[
M(t):=e^{2t}a(e^t),\qquad
\mathcal H(M,V):=\frac{V^2}{2}-2M^2+\frac{C_p}{3}M^3.
\]
Then the limit
\[
E_\alpha:=\lim_{t\to\infty}\mathcal H(M(t),\dot M(t))
\]
exists and, uniformly for $H$ in bounded sets satisfying \eqref{c:eq_forcing_bound},
\begin{equation}\label{c:eq_fowler_energy_asymptotic}
E_\alpha=-\frac{1152}{35}\,\frac{\kappa}{C_p^3}\,\alpha+o(\alpha)
\qquad\text{as }\alpha\to0.
\end{equation}
In particular, for $\alpha$ small,
\[
-\frac{32}{3C_p^2}<E_\alpha<0.
\]
Let $M_{E_\alpha}$ be any positive periodic solution of
\[
\ddot M-4M+C_pM^2=0,
\qquad
\mathcal H(M_{E_\alpha},\dot M_{E_\alpha})=E_\alpha.
\]
Then there is a phase $\vartheta_\alpha\in\mathbb R$ such that, with
$\delta:=\min\{2,\sigma-2\}>0$,
\begin{equation}\label{c:eq_fowler_phase}
M(t)-M_{E_\alpha}(t+\vartheta_\alpha)=O_\alpha(e^{-\delta t}),
\qquad
\dot M(t)-\dot M_{E_\alpha}(t+\vartheta_\alpha)=O_\alpha(e^{-\delta t}).
\end{equation}
Equivalently,
\[
a(r)=r^{-2}\Big(M_{E_\alpha}(\log r+\vartheta_\alpha)+O_\alpha(r^{-\delta})\Big)
\qquad\text{as }r\to\infty.
\]
If $M_-(E_\alpha)$ and $M_+(E_\alpha)$ denote the two positive turning points of the orbit and $\mathcal T_\alpha^{\rm F}$ its minimal period, then
\begin{equation}\label{c:eq_fowler_neck_bulge}
M_-(E_\alpha)=\frac{24}{\sqrt{35}}\frac{\sqrt\kappa}{C_p^{3/2}}\,\alpha^{1/2}(1+o(1)),
\qquad
M_+(E_\alpha)=\frac6{C_p}+O(\alpha),
\qquad
\mathcal T_\alpha^{\rm F}=\frac12|\log\alpha|+O(1).
\end{equation}
\end{proposition}

\begin{proof}
By \eqref{c:eq_forced_poho_der}, \eqref{c:eq_forced_bounds} and $Q_6(a)=O(a^4)$, both terms in $\mathcal P'$ are integrable on $(0,\infty)$; hence $\mathcal P(r)$ has a finite limit.  Moreover
\[
r^6\mathfrak G(r^{-2}M)=\frac{C_p}{3}M^3+O(r^{-2}M^4),
\]
and $M$ is bounded by \eqref{c:eq_M_uniform_forced}, so this limit is precisely the limit of $\mathcal H(M,\dot M)$.

To compute it, integrate \eqref{c:eq_forced_poho_der} from $0$ to infinity and set $\rho=\varepsilon r$, $a(r)=\alpha A_\alpha(\rho)$.  Lemma \ref{c:lem_forced_critical} gives
\[
0<A_\alpha(\rho)\le C(1+\rho)^{-2},
\]
while the core argument in its proof gives $A_\alpha\to A_0=(1+\rho^2/24)^{-2}$ locally uniformly.  Hence dominated convergence and \eqref{c:eq_Q6negative} give
\begin{align*}
\int_0^\infty r^5Q_6(a(r))\,dr
&=-\frac{\kappa}{2}\,\alpha^4\varepsilon^{-6}
  \int_0^\infty\rho^5A_\alpha(\rho)^4\,d\rho+O(\alpha^2)\\
&=-\frac{\kappa\alpha}{2C_p^3}
  \int_0^\infty\rho^5\left(1+\frac{\rho^2}{24}\right)^{-8}d\rho+o(\alpha).
\end{align*}
The beta integral is
\[
\int_0^\infty\rho^5\left(1+\frac{\rho^2}{24}\right)^{-8}d\rho
=\frac{2304}{35}.
\]
On the other hand, using \eqref{c:eq_forcing_bound} and \eqref{c:eq_forced_bounds},
\[
\int_0^\infty r^5|H(r,a)|\,|2a+ra'|\,dr
\le C\alpha^5\varepsilon^{-6}
\int_0^\infty \rho^5(1+\rho)^{-\sigma-4}\,d\rho
=O(\alpha^2),
\]
because $\sigma>2$.  This proves \eqref{c:eq_fowler_energy_asymptotic}.

The Hamiltonian $\mathcal H$ has its positive equilibrium at $M_*=4/C_p$, with energy
$\mathcal H(M_*,0)=-32/(3C_p^2)$.  Therefore, for every energy strictly between this value and $0$, the positive energy level is a smooth closed periodic orbit of the autonomous Fowler equation
$\ddot M-4M+C_pM^2=0$.  By \eqref{c:eq_fowler_energy_asymptotic}, $E_\alpha$ lies in this interval for small $\alpha$.  Since $M$ and $\dot M$ are bounded and $\mathcal H(M(t),\dot M(t))\to E_\alpha$, while a negative energy level stays a positive distance away from $M=0$, compactness first gives
\[
\operatorname{dist}\big((M(t),\dot M(t)),\Gamma_{E_\alpha}\big)\longrightarrow0.
\]

It remains to identify the asymptotic phase.  From the expansion
$\mathfrak g(a)=C_pa^2+\kappa a^3+O(a^4)$ and the forced equation we obtain
\begin{equation}\label{c:eq_fowler_perturbed}
\ddot M-4M+C_pM^2=\mathcal E_\alpha(t),
\end{equation}
where, for all sufficiently large $t$,
\[
|\mathcal E_\alpha(t)|\le C_\alpha e^{-\delta t},
\qquad \delta=\min\{2,\sigma-2\}>0.
\]
Indeed the correction in $\mathfrak g$ is $O(e^{-2t}M^3)$, while
$e^{4t}|H(e^t,e^{-2t}M)|\le C_\alpha e^{-(\sigma-2)t}$ once $\varepsilon e^t\ge1$.  In particular
\[
\frac{d}{dt}\mathcal H(M,\dot M)=\dot M\,\mathcal E_\alpha(t),
\qquad
\mathcal H(M(t),\dot M(t))-E_\alpha=O_\alpha(e^{-\delta t}).
\]
Choose a small annular neighborhood of $\Gamma_{E_\alpha}$ containing the trajectory for all large $t$.  Since the autonomous vector field does not vanish on this periodic orbit, this annulus admits smooth energy--phase coordinates $(E,\theta)$, with $\dot\theta=1$ on $\Gamma_{E_\alpha}$ when the orbit is parameterized by its physical time.  Equation \eqref{c:eq_fowler_perturbed} and the preceding energy estimate give
\[
E(t)-E_\alpha=O_\alpha(e^{-\delta t}),
\qquad
\dot\theta(t)-1=O_\alpha(e^{-\delta t}).
\]
Thus $\theta(t)-t$ has a finite limit $\vartheta_\alpha$, and smoothness of the energy--phase coordinates gives \eqref{c:eq_fowler_phase}.

At a turning point $V=0$, so
\[
-2M^2+\frac{C_p}{3}M^3=E_\alpha.
\]
For the smaller positive root, $M\to0$ as $E_\alpha\uparrow0$, and hence
$M_-(E_\alpha)=\sqrt{-E_\alpha/2}\,(1+o(1))$.  Using
\eqref{c:eq_fowler_energy_asymptotic} gives the first formula in
\eqref{c:eq_fowler_neck_bulge}.  The larger root tends to the nonzero root $6/C_p$ of
$-2M^2+C_pM^3/3=0$; since the derivative of this polynomial at $6/C_p$ is $12/C_p\ne0$,
the implicit function theorem gives $M_+(E_\alpha)=6/C_p+O(E_\alpha)=6/C_p+O(\alpha)$.

It remains to estimate the period.  Write $\delta_\alpha:=-E_\alpha>0$ and
$m_\alpha:=M_-(E_\alpha)$.  The Hamiltonian identity gives
\[
\mathcal T_\alpha^{\rm F}
=2\int_{m_\alpha}^{M_+(E_\alpha)}
\frac{dM}{\sqrt{2\{E_\alpha+2M^2-C_pM^3/3\}}}.
\]
Fix a small $\eta>0$, independently of $\alpha$.  The part over
$[\eta,M_+(E_\alpha)]$ is $O(1)$ uniformly as $\alpha\to0$.  On
$[m_\alpha,\eta]$, using
$E_\alpha=-2m_\alpha^2+C_pm_\alpha^3/3$, we have
\[
2\{E_\alpha+2M^2-C_pM^3/3\}
=4(M^2-m_\alpha^2)
\left[1-\frac{C_p}{6}
\frac{M^3-m_\alpha^3}{M^2-m_\alpha^2}\right].
\]
Since
$(M^3-m^3)/(M^2-m^2)=(M^2+Mm+m^2)/(M+m)=O(M+m)$,
the difference between the corresponding integrand and
$[2\sqrt{M^2-m_\alpha^2}]^{-1}$ has an integral bounded uniformly in
$\alpha$.  Consequently
\[
\mathcal T_\alpha^{\rm F}
=\int_{m_\alpha}^{\eta}\frac{dM}{\sqrt{M^2-m_\alpha^2}}+O(1)
=\log\frac1{m_\alpha}+O(1).
\]
Since $m_\alpha\asymp\alpha^{1/2}$ by the first formula in
\eqref{c:eq_fowler_neck_bulge}, this gives
$\mathcal T_\alpha^{\rm F}=\frac12|\log\alpha|+O(1)$.
\end{proof}

\subsection{The critical fixed point}

Extend the cell family continuously by $\zeta_0:=0$.  For a radial transverse function $\chi$ and a scalar $a\ge0$ set
\begin{equation}\label{c:eq_projected_force_N7}
\mathcal R_\chi(r,a):=\frac1\nu\int_0^\infty
\Big(R(t,\zeta_a(t)+\chi(r,t))-R(t,\zeta_a(t))\Big)Z(t)\,dt.
\end{equation}
The dependence on the scalar parameter is better than the regularity of $f'$ at the origin might suggest.  Differentiating with respect to $a$ gives
\[
\partial_a\Big(R(t,\zeta_a+\chi)-R(t,\zeta_a)\Big)
=\Big[f'(w_0+\zeta_a+\chi)-f'(w_0+\zeta_a)\Big]\,\partial_a\zeta_a.
\]
For $1<p<2$ we use the $(p-1)$--H\"older continuity of $f'$, and for $p\ge2$ its local Lipschitz continuity.  Together with Proposition \ref{c:prop_cell}, applied with a fixed exponent $\widehat\mu\in(\mu,1)$, this shows that the integral in \eqref{c:eq_projected_force_N7} is locally Lipschitz in $a$ for every fixed $r$.  Notice that no Lipschitz regularity of $f'$ is required when $1<p<2$.

We record the quantitative estimate that will be used below.  Put $d:=\|\chi\|_{C^{2,\varrho}_*}$ and $W_r:=1+\varepsilon r$.  Uniformly for $0\le a\le C\alpha$,
\begin{equation}\label{c:eq_projected_force_pointwise_N7}
|\mathcal R_\chi(r,a)|
\le C\big(\alpha d+d^2+d^{p+1}\big)W_r^{-\sigma-2}.
\end{equation}
Indeed, split the $x_N$--integral into the near region $|\chi|\le w_0/4$ and the far region $|\chi|>w_0/4$.  On the near region, $\zeta_a\le C\alpha w_0$ for small $\alpha$, and Taylor's formula gives
\[
|R(t,\zeta_a+\chi)-R(t,\zeta_a)|
\le Cw_0^{p-2}\big(\zeta_a|\chi|+|\chi|^2\big).
\]
After multiplication by $Z$ and integration this contributes at most
$C(\alpha d_r+d_r^2)$, where $d_r:=dW_r^{-\sigma-2}$.  On the far region, $w_0\le4|\chi|$ and the second estimate in \eqref{c:eq_Rtwo} gives a bound by $C|\chi|^p$.  Since $|\chi(r,t)|\le d_r e^{-\mu t}$, the far region can occur only for
\[
t\ge T_r:=\frac1{1-\mu}\log\frac{c}{d_r},
\]
with a harmless change of the constant $c>0$.  Hence
\[
\int_{T_r}^\infty |\chi|^pZ\,dt
\le C d_r^p e^{-(p\mu+1)T_r}
= C d_r^{(p+1)/(1-\mu)}
\le C d_r^{p+1},
\]
which proves \eqref{c:eq_projected_force_pointwise_N7}.  The same splitting also gives the required Lipschitz bound in the amplitude, including when $1<p<2$.
Fix once and for all $\widehat\mu\in(\mu,1)$.  The $C^1$ dependence in Proposition \ref{c:prop_cell},
viewed in $\mathcal X_{\widehat\mu}$, gives
$|\partial_a\zeta_a(t)|\le C e^{-\widehat\mu t}$ for $0\le a\le C\alpha$.
On the near region the mean value theorem gives
\[
|f'(w_0+\zeta_a+\chi)-f'(w_0+\zeta_a)|
\le Cw_0^{p-2}|\chi|,
\]
and therefore its contribution to $|\partial_a\mathcal R_\chi|$ is at most $Cd_r$, since
$w_0^{p-2}e^{-(\mu+\widehat\mu)t}Z(t)$ is integrable.  On the far region the
$(p-1)$--H\"older continuity of $f'$ gives
\[
|f'(w_0+\zeta_a+\chi)-f'(w_0+\zeta_a)|\le C|\chi|^{p-1},
\]
so, using $Z(t)\le Ce^{-t}$ and the definition of $T_r$,
\[
\int_{T_r}^\infty |\chi|^{p-1}|\partial_a\zeta_a|Z\,dt
\le C d_r^{p-1}e^{-[\mu(p-1)+\widehat\mu+1]T_r}
\le C d_r^{\frac{p+\widehat\mu}{1-\mu}}
\le Cd_r.
\]
Here $(p+\widehat\mu)/(1-\mu)>1$ and $d_r<1$ for $\alpha$ small.  Thus
\begin{equation}\label{c:eq_projected_force_da_N7}
|\partial_a\mathcal R_\chi(r,a)|\le C d\,W_r^{-\sigma-2},
\qquad 0\le a\le C\alpha.
\end{equation}

The unit--scale H\"older estimates needed by the transverse inverse can be obtained without differentiating $f'$.  If $a=a(r)$ satisfies the first two bounds in \eqref{c:eq_forced_bounds}, differentiation in $r$ gives
\begin{align*}
\frac{d}{dr}\mathcal R_\chi(r,a(r))
={}&\frac1\nu\int_0^\infty
\big[f'(w_0+\zeta_a+\chi)-f'(w_0)\big]\,\partial_r\chi\,Z\,dt\\
&\quad+\partial_a\mathcal R_\chi(r,a(r))\,a'(r).
\end{align*}
The same near--far decomposition as above, now using the pointwise derivative bound for $\partial_r\chi$ contained in $\|\chi\|_{C^{2,\varrho}_*}$, gives
\begin{equation}
\left|\frac{d}{dr}\mathcal R_\chi(r,a(r))\right|
\le C\big(\alpha d+d^2+d^{p+1}\big)W_r^{-\sigma-2}.
\end{equation}
Indeed, on the near region
$|f'(w_0+\zeta_a+\chi)-f'(w_0)|
\le Cw_0^{p-2}(\zeta_a+|\chi|)$,
and on the far region the integral begins at $T_r$ and is bounded exactly as in the proof of \eqref{c:eq_projected_force_pointwise_N7}; the second line is controlled by \eqref{c:eq_projected_force_da_N7} and the bound for $a'$.  Consequently, for every $0<\varrho<1$,
\begin{equation}\label{c:eq_projected_force_holder_N7}
\sup_{r_0\ge0}W_{r_0}^{\sigma+2}
\Big(
|\mathcal R_\chi(r_0,a(r_0))|
+[\mathcal R_\chi(\cdot,a(\cdot))]_{C^{0,\varrho}(I_{r_0})}
\Big)
\le C\big(\alpha d+d^2+d^{p+1}\big),
\end{equation}
where $I_{r_0}:=(r_0-1,r_0+1)\cap[0,\infty)$.

On the ball \eqref{c:eq_B7}, $d\le K_1\alpha^3$, so \eqref{c:eq_projected_force_pointwise_N7} and \eqref{c:eq_projected_force_holder_N7} are $O(\alpha^4)$.  In particular the pointwise bound required by Lemma \ref{c:lem_forced_critical} holds for every $p>1$.

Let $K_1>0$ be fixed later and let $\mathcal B_7$ be the set of functions $\chi$ which are radial in $x'$, vanish on $\partial\mathbb R^7_+$, satisfy $\int_0^\infty\chi(x',t)Z(t)\,dt=0$, and
\begin{equation}\label{c:eq_B7}
\|\chi\|_{C^{2,\varrho}_*}\le K_1\alpha^3.
\end{equation}
We regard $\mathcal B_7$ as a subset of the locally convex space $C^2_{\rm loc}(\overline{\mathbb R^7_+})$, with its usual Fr\'echet topology.  It is convex and compact in that topology.  Indeed, the weighted estimate gives uniform $C^{2,\varrho}$ bounds on every compact set, so diagonal Arzel\`a--Ascoli gives subsequential convergence in $C^2_{\rm loc}$.  The full weighted $C^{2,\varrho}$ bound is closed under this convergence.  To see this, fix any unit interior or boundary ball $Q$.  Uniform convergence of all derivatives through order two gives the pointwise weighted derivative bounds in the limit, while for $x\ne y$ in $Q$ the H\"older quotient of $D^2\chi$ is the limit of the corresponding quotients for the approximating sequence; taking the supremum first over $x,y\in Q$ and then over all such balls gives
$\|\chi\|_{C^{2,\varrho}_*}\le K_1\alpha^3$.  Radial symmetry and the boundary condition are closed under local convergence, while the orthogonality is also closed: for every fixed $x'$, local convergence handles the integral on $(0,T)$ and the uniform bound $|\chi(x',t)Z(t)|\le C e^{-(1+\mu)t}$ makes the tail $(T,\infty)$ uniformly small.  Thus every such limit still belongs to $\mathcal B_7$.

For $\chi\in\mathcal B_7$, solve
\begin{equation}\label{c:eq_a_chi}
-\Delta_{x'}a_\chi=\mathfrak g(a_\chi)+\mathcal R_\chi(|x'|,a_\chi),
\qquad a_\chi(0)=\alpha,
\end{equation}
in the radial class.  By \eqref{c:eq_projected_force_pointwise_N7} and \eqref{c:eq_B7}, the forcing is bounded by $C\alpha^4(1+\varepsilon r)^{-\sigma-2}$, so Lemma \ref{c:lem_forced_critical} gives a positive entire solution and
\begin{equation}\label{c:eq_a_chi_bounds}
0<a_\chi(r)\le C\alpha(1+\varepsilon r)^{-2},\qquad
|a_\chi'(r)|\le C\alpha\varepsilon(1+\varepsilon r)^{-3},\qquad
|D^2a_\chi(x')|\le C\alpha\varepsilon^2(1+\varepsilon|x'|)^{-4}.
\end{equation}
In particular $a_\chi<a_0$ for $\alpha$ small, so the cell profile $\zeta_{a_\chi}$ is well defined.

Set $\zeta=\zeta_{a_\chi}$ and let $\mathcal N[a_\chi]$ be defined by Lemma \ref{c:lem_error}.  Since \eqref{c:eq_a_chi} gives
$\Delta a_\chi=-\mathfrak g(a_\chi)-\mathcal R_\chi$, formula \eqref{c:eq_E} becomes
\begin{equation}
\mathcal N[a_\chi]
=T(a_\chi)\big(a_\chi\partial_a\zeta_a-\zeta_a\big)_{a=a_\chi}
+\mathcal R_\chi\,\partial_a\zeta_a|_{a=a_\chi}
-|\nabla a_\chi|^2\partial_a^2\zeta_a|_{a=a_\chi}.
\end{equation}
Proposition \ref{c:prop_cell}, \eqref{c:eq_a_chi_bounds} and the same stronger exponential weight used in the proof of \eqref{c:eq_e1} give
\begin{equation}\label{c:eq_N_forced_est}
\|\mathcal N[a_\chi]\|_{C^{0,\varrho}_{**}}\le C\alpha^3.
\end{equation}
Indeed, choose once and for all $\widehat\mu\in(\mu,1)$.  Proposition \ref{c:prop_cell} in $\mathcal X_{\widehat\mu}$ gives
\[
\big|T(a)(a\partial_a\zeta_a-\zeta_a)\big|
+|\nabla a|^2|\partial_a^2\zeta_a|
\le C\big(a^3+|\nabla a|^2\big)e^{-\widehat\mu x_N}.
\]
By \eqref{c:eq_a_chi_bounds} the coefficient has $x'$ decay at least $(1+\varepsilon r)^{-6}$, stronger than the required $(1+\varepsilon r)^{-\sigma-2}$ because $\sigma<4$.  The corresponding unit--scale H\"older seminorms have the same bound: for the first term use $|\nabla a_\chi|$ from \eqref{c:eq_a_chi_bounds}, and for the third use
\[
[|\nabla a_\chi|^2]_{C^{0,\varrho}(B_1(x'))}
\le C\sup_{B_1(x')}|\nabla a_\chi|\,\sup_{B_1(x')}|D^2a_\chi|
\le C\alpha^3(1+\varepsilon|x'|)^{-6},
\]
after decreasing $\alpha$ and using comparability of the weights on unit balls.  The middle term is $O(\alpha^4)$, with the required unit--scale H\"older control, by \eqref{c:eq_projected_force_holder_N7} and the $C^1$ dependence of $\partial_a\zeta_a$ on $a$.  Finally $(1+x_N)^2e^{-(\widehat\mu-\mu)x_N}$ is bounded, so the stronger exponential absorbs the polynomial factor in the double-star norm.  This proves \eqref{c:eq_N_forced_est}.

Define
\begin{equation}\label{c:eq_T7}
\mathcal T_7(\chi):=\mathcal K\Big(
R(\cdot,\zeta+\chi)-R(\cdot,\zeta)-\mathcal N[a_\chi]
\Big).
\end{equation}
The near--far argument of \eqref{c:eq_e2}, with $a_\chi$ in place of $\lambda_\alpha+\psi$ and using the local regularity in \eqref{c:eq_a_chi_bounds}, together with \eqref{c:eq_N_forced_est}, gives
\begin{equation}\label{c:eq_T7_est}
\|\mathcal T_7(\chi)\|_{C^{2,\varrho}_*}
\le C\Big[\alpha^3+\alpha\|\chi\|_*+\|\chi\|_*^{1+\varpi_0}\Big],
\qquad \varpi_0>0.
\end{equation}
The source in \eqref{c:eq_T7} is exactly orthogonal to $Z$.  Indeed, by \eqref{c:eq_projected_force_N7},
\[
\int_0^\infty\big(R(\cdot,\zeta+\chi)-R(\cdot,\zeta)\big)Z\,dt
=\nu\mathcal R_\chi,
\]
whereas Lemma \ref{c:lem_orth} and \eqref{c:eq_a_chi} give
\[
\int_0^\infty\mathcal N[a_\chi]Z\,dt
=-\nu(\Delta a_\chi+\mathfrak g(a_\chi))
=\nu\mathcal R_\chi.
\]
Thus Corollary \ref{c:cor_weighttrans} applies.  The source is radial in $x'$, and uniqueness in Corollary \ref{c:cor_weighttrans} shows that $\mathcal K$ preserves radial symmetry; it also preserves the boundary condition and the $Z$-orthogonality by construction.  Choosing $K_1$ larger than twice the constant multiplying $\alpha^3$ in \eqref{c:eq_T7_est}, and then choosing $\alpha$ small, shows that
$\mathcal T_7(\mathcal B_7)\subset\mathcal B_7$.

We finally check continuity in the $C^2_{\rm loc}$ topology.  If $\chi_j\to\chi$ locally in $C^2$ with $\chi_j\in\mathcal B_7$, then the uniform exponential weight in $x_N$ makes the tails of the integral in \eqref{c:eq_projected_force_N7} uniformly small, while on bounded $x_N$ intervals ordinary dominated convergence applies.  Hence $\mathcal R_{\chi_j}\to\mathcal R_\chi$ locally uniformly in $(r,a)$.  By the local Lipschitz estimate in \eqref{c:eq_projected_force_da_N7}, uniqueness and continuous dependence for the radial ODE \eqref{c:eq_a_chi} yield
$a_{\chi_j}\to a_\chi$ in $C^2_{\rm loc}(\mathbb R^6)$.  To make the regular singular point $r=0$ explicit, this is equivalently the Volterra equation
\begin{equation}
a(r)=\alpha-\int_0^r s^{-5}\int_0^s\tau^5
\big[\mathfrak g(a(\tau))+\mathcal R_\chi(\tau,a(\tau))\big] \,d\tau\,ds,
\end{equation}
for which the usual contraction argument on each bounded interval gives uniqueness and continuous dependence; continuation uses the uniform bounds \eqref{c:eq_a_chi_bounds}.  Hence the sources in \eqref{c:eq_T7} converge locally uniformly and remain bounded in $C^{0,\varrho}_{**}$.  The weighted estimate for $\mathcal K$ makes $\mathcal T_7(\chi_j)$ uniformly bounded in $C^{2,\varrho}_*$; every locally convergent subsequence therefore tends to a weighted solution of the limit transverse equation.  The uniqueness statement in Corollary \ref{c:cor_weighttrans} identifies that limit with $\mathcal T_7(\chi)$.  Thus $\mathcal T_7$ is continuous on the compact convex set $\mathcal B_7$.  The Schauder--Tychonoff fixed point theorem gives a fixed point $\chi_\alpha$.

At the fixed point put $a=a_{\chi_\alpha}$ and
\[
\phi(x',x_N):=\zeta_{a(x')}(x_N)+\chi_\alpha(x',x_N),
\qquad v_\alpha:=w_0+\phi.
\]
Equations \eqref{c:eq_a_chi} and \eqref{c:eq_T7}, together with the orthogonality just checked, are exactly the $Z$ and $Z^\perp$ components of \eqref{eq_phiL}; hence $v_\alpha$ solves the equation and has boundary value $c_p$.  The decay in $x_N$ is uniform because $\zeta_a\le Cae^{-x_N}$ and $\chi_\alpha=O(e^{-\mu x_N})$ uniformly in $x'$.  Boundedness is immediate.

Positivity is proved exactly as at the end of the proof of Theorem \ref{thm_construction}.  If $m:=\inf v_\alpha<0$, tangential translations and compactness produce an interior minimum of a bounded solution $V$ with value $m$.  Since $\|\phi\|_\infty\le C\alpha$, one has $|m|<1$ for $\alpha$ small, while
\[
\Delta V=m(1-|m|^{p-1})<0
\]
at that minimum, a contradiction.  Hence $v_\alpha>0$ by the strong maximum principle.

Finally, by the normalization of the cell family and the orthogonality of $\chi_\alpha$,
\[
\frac1\nu\int_0^\infty(v_\alpha-w_0)Z\,dt=a(x').
\]
Here $a(0)=\alpha$ and $a(x')\to0$ by \eqref{c:eq_a_chi_bounds}, so $v_\alpha$ is not one-dimensional.  Proposition \ref{c:prop_cell} and \eqref{c:eq_B7} give
$v_\alpha=w_0+aZ+a^2h+O(\alpha^3)$ in $L^\infty$, and \eqref{eq_critical_amp_bounds} is precisely \eqref{c:eq_a_chi_bounds}.  Finally the forcing $H=\mathcal R_{\chi_\alpha}$ satisfies the hypotheses of Proposition \ref{c:prop_forced_fowler} by \eqref{c:eq_projected_force_pointwise_N7}, so \eqref{eq_critical_fowler_main} follows from \eqref{c:eq_fowler_energy_asymptotic} and \eqref{c:eq_fowler_phase}.  This proves Theorem \ref{thm_construction_critical}.

\appendix
\section{The cubic coefficient in dimension seven}\label{app_kappa}

The coefficient $\kappa$ controls the first breaking of scale invariance in boundary dimension six. We collect here the one-dimensional identities and sign information used in Theorem \ref{thm_construction_critical}.

There is also a useful ground-state representation of the positive term.  Put
\[
f_p(t):=\frac{p(p-1)}2w_0(t)^{p-2}Z(t)^2-C_pZ(t)=\mathcal Lh(t),
\qquad
F_p(t):=\int_t^\infty Z(s)f_p(s)\,ds.
\]
The definition of $C_p$ gives $F_p(0)=0$.  Writing $h=Zu$ and using $\mathcal LZ=0$ gives the ground-state factorization
\[
\mathcal L(Zu)=-\frac1Z\big(Z^2u'\big)',
\qquad
Z^2u'=F_p,
\]
and therefore
\begin{equation}\label{c:eq_kappa_groundstate}
\int_0^\infty h\mathcal Lh\,dt
=\int_0^\infty Z^2(u')^2\,dt
=\int_0^\infty\frac{F_p(t)^2}{Z(t)^2}\,dt.
\end{equation}
The same representation also gives two facts that will be useful in the critical dimension.  First, $p\mapsto\kappa(p)$ is continuous on $(1,\infty)$.  Indeed, on every compact $p$-interval contained in $(1,\infty)$ the explicit formula \eqref{eq_w0} gives locally uniform dependence of $w_0$, $Z$, $C_p$, $f_p$ and $F_p$ on $p$.  Moreover, uniformly on such an interval,
\[
\frac{F_p(t)^2}{Z(t)^2}=O(t^4)\quad(t\downarrow0),
\qquad
\frac{F_p(t)^2}{Z(t)^2}=O(e^{-2t})\quad(t\to\infty),
\]
and $w_0^{p-3}Z^4=O(e^{-(p+1)t})$ at infinity, so dominated convergence applies in \eqref{c:eq_kappapos}--\eqref{c:eq_kappa_groundstate}.

Second, the value at the quadratic exponent can be evaluated exactly.  When $p=2$, put $y=\tanh(t/2)$. Then
\[
\nu=\frac35,\qquad C_2=\frac{15}{32},
\]
and a direct integration gives
\[
F_2(t)=-\frac9{64}\,y^3(y-1)^2\big(6y^3+12y^2+2y-5\big),
\qquad
\int_0^\infty\frac{F_2(t)^2}{Z(t)^2}\,dt=\frac9{2048}.
\]
Since the second term in \eqref{c:eq_kappapos} vanishes at $p=2$, it follows that
\begin{equation}
\kappa(2)=\frac{15}{1024}>0.
\end{equation}
Consequently there exists $p_+\in(1,2)$ such that $\kappa(p)>0$ for every $p>p_+$: use continuity near $2$ and the positivity already proved above for all $p\ge2$.

At the opposite end of the exponent range the sign is rigorously negative.  Put $m=p-1$ and $s=\sqrt m\,t$.  A direct beta-function computation gives
\[
\nu=\frac{c_p^2}{m}B\left(\frac32,\frac2m\right),
\qquad
C_p=\frac{m^3}{8B\left(\frac32,\frac2m\right)}.
\]
As $m\downarrow0$, the explicit profile \eqref{eq_w0}, the asymptotic $B(a,x)\sim\Gamma(a)x^{-a}$, and the change of variables above give, locally uniformly in $s$ and with an integrable domination uniform for small $m$,
\[
w_0(s/\sqrt m)\longrightarrow W(s):=\sqrt e\,e^{-s^2/4},
\qquad
m^{-1/2}Z(s/\sqrt m)\longrightarrow Z_0(s):=\frac{s}{2}W(s),
\]
\[
m^{-3/2}C_p\longrightarrow\frac1{\sqrt{2\pi}},
\qquad
m^{-2}f_p(s/\sqrt m)\longrightarrow
f_0(s):=W(s)\left(\frac{s^2}{8}-\frac{s}{2\sqrt{2\pi}}\right).
\]
Consequently
\begin{equation}
\frac{\nu}{\sqrt m}\longrightarrow
\frac{e}{4}\sqrt{\frac\pi2},
\qquad
\frac{\|f_p\|_{L^2(0,\infty)}^2}{m^{7/2}}
\longrightarrow
\frac{e\sqrt2}{128\sqrt\pi}(3\pi-8),
\end{equation}
and hence
\begin{equation}\label{c:eq_smallp_ratio}
\frac{\|f_p\|_2^2}{m^3\nu}\longrightarrow\frac{3\pi-8}{16\pi}.
\end{equation}
Since $f_p\perp Z$, the spectral theorem and \eqref{c:eq_exact_gap} imply
\[
\int_0^\infty h\mathcal Lh\,dt
=\langle f_p,\mathcal L^{-1}f_p\rangle
\le\frac{\|f_p\|_2^2}{m(2-m)}.
\]
The remaining integral in \eqref{c:eq_kappapos} is also explicit:
\begin{equation}\label{c:eq_I4ratio}
\frac1\nu\int_0^\infty w_0^{p-3}Z^4\,dt
=\frac{6m(m+2)}{(3m+4)(5m+4)}.
\end{equation}
Dividing \eqref{c:eq_kappapos} by $m^2$, and using \eqref{c:eq_smallp_ratio}--\eqref{c:eq_I4ratio}, yields
\[
\limsup_{p\downarrow1}\frac{\kappa(p)}{(p-1)^2}
\le
\frac{3\pi-8}{16\pi}-\frac18
=\frac{\pi-8}{16\pi}<0.
\]
Thus there is $p_-\in(1,2)$ such that $\kappa(p)<0$ for $1<p<p_-$. In particular $\kappa$ has at least one zero in $(1,2)$.

\section*{Declarations}

\smallskip
\noindent\textbf{Conflict of interest.} The author declares that he has no conflict of interest.

\smallskip
\noindent\textbf{Data availability.} Data sharing is not applicable to this article, as no datasets were generated or analyzed during the current study.

\smallskip
\noindent\textbf{Use of AI tools.} AI-assisted tools were used during the preparation of the manuscript. The author reviewed and edited all AI-assisted output and takes full responsibility for the content of the work.

\bibliographystyle{abbrv}
\bibliography{ref}

\begin{thebibliography}{1}

\bibitem{CGS}
L.~Caffarelli, N.~Garofalo, and F.~Seg\`ala.
\newblock A gradient bound for entire solutions of quasi-linear equations and
  its consequences.
\newblock {\em Comm. Pure Appl. Math.}, 47:1457--1473, 1994.

\bibitem{FV1}
A.~Farina and E.~Valdinoci.
\newblock Flattening results for elliptic {PDE}s in unbounded domains with
  applications to overdetermined problems.
\newblock {\em Arch. Ration. Mech. Anal.}, 195:1025--1058, 2010.

\bibitem{FV2}
A.~Farina and E.~Valdinoci.
\newblock A pointwise gradient estimate in possibly unbounded domains with
  nonnegative mean curvature.
\newblock {\em Adv. Math.}, 225:2808--2827, 2010.

\bibitem{FW}
A.~J. Fern\'andez and T.~Weth.
\newblock The nonlinear {S}chr\"odinger equation in the half-space.
\newblock {\em Math. Ann.}, 383:361--397, 2022.

\bibitem{GS}
B.~Gidas and J.~Spruck.
\newblock Global and local behavior of positive solutions of nonlinear elliptic
  equations.
\newblock {\em Comm. Pure Appl. Math.}, 34:525--598, 1981.

\bibitem{GT}
D.~Gilbarg and N.~S. Trudinger.
\newblock {\em Elliptic Partial Differential Equations of Second Order}.
\newblock Springer-Verlag, Berlin, second edition, 1983.

\bibitem{Mo}
L.~Modica.
\newblock A gradient bound and a {L}iouville theorem for nonlinear {P}oisson
  equations.
\newblock {\em Comm. Pure Appl. Math.}, 38:679--684, 1985.

\bibitem{So1}
C.~Sourdis.
\newblock One-dimensional symmetry of positive bounded solutions to the
  nonlinear {S}chr\"odinger equation in the half-space in two and three
  dimensions.
\newblock {\em Partial Differ. Equ. Appl.}, 2:Paper No. 79, 2021.

\bibitem{So2}
C.~Sourdis.
\newblock One-dimensional symmetry of positive bounded solutions to the
  subcubic and cubic nonlinear {S}chr\"odinger equation in the half-space in
  dimensions {$N=4,5$}.
\newblock {\em Calc. Var. Partial Differential Equations}, 61:Paper No. 176,
  2022.

\end{thebibliography}

\end{document}